\documentclass[reqno, 10pt]{article}
\usepackage{fullpage}
\usepackage[centertags]{amsmath}
\usepackage{amsmath}
\usepackage{amsfonts}
\usepackage{amssymb}
\usepackage{amsthm}
\usepackage{enumerate}

\usepackage{newlfont}
\usepackage{color}
\usepackage{authblk}

\usepackage{stmaryrd}
\SetSymbolFont{stmry}{bold}{U}{stmry}{m}{n}

\usepackage{bbm}
\include{amsthm_sc}

\usepackage{stmaryrd}
\usepackage{mathrsfs}

\usepackage{fancyhdr}
\usepackage{graphicx}
\usepackage{fancybox}
\usepackage{setspace}
\usepackage{cleveref}

\newtheorem{theorem}{Theorem}
\newtheorem{corollary}[theorem]{Corollary}
\newtheorem{lemma}[theorem]{Lemma}
\newtheorem{proposition}[theorem]{Proposition}

\theoremstyle{definition}
\newtheorem{definition}[theorem]{Definition}
\newtheorem{assumption}[theorem]{Assumption}

\theoremstyle{remark}
\newtheorem{remark}[theorem]{Remark}
\newtheorem{notation}[theorem]{Notational Remark}

\newcommand{\deq}{\mathrel{\mathop:}=}
\newcommand{\R} {\mathbb{R}}
\newcommand{\C} {\mathbb{C}}
\newcommand{\N} {\mathbb{N}}
\newcommand{\E} {\mathbb{E}}

\newcommand{\dist} {\mathrm{dist}}
\newcommand{\e}[1]{\mathrm{e}^{#1}}

\newcommand{\bilin}[2]{\langle #1,\,#2\rangle}

\DeclareMathOperator{\diag}{diag}
\DeclareMathOperator{\tr}{tr}
\DeclareMathOperator{\Tr}{Tr}
\DeclareMathOperator{\supp}{supp}

\DeclareMathOperator{\im}{\mathrm{Im}}

\newcommand{\caD}{{\mathcal D}}
\newcommand{\caF}{{\mathcal F}}
\newcommand{\caG}{{\mathcal G}}
\newcommand{\caO}{{\mathcal O}}
\newcommand{\caK}{{\mathcal K}}
\newcommand{\caL}{{\mathcal L}}
\newcommand{\caN}{{\mathcal N}}
\newcommand{\caH}{{\mathcal{H}}}
\newcommand{\frt}{{\mathfrak t}}
\newcommand{\wt}{\widetilde}

\newcommand{\wh}{\widehat}
\newcommand{\beq}{ \begin{equation} }
\newcommand{\eeq}{ \end{equation} }
\newcommand{\dd}{\mathrm{d}}
\newcommand{\ii}{\mathrm{i}}

\newcommand\norm[1]{\Vert#1\Vert}
\newcommand\Var[1]{\mathrm{Var}[#1]}

\newcommand\expct[1]{\mathbb{E}[#1]}
\newcommand\biggexpct[1]{\mathbb{E}\bigg[#1\bigg]}
\newcommand\Bigexpct[1]{\mathbb{E}\Big[#1\Big]}

\newcommand\cexpct[2]{\mathbb{E}[#1\vert#2]}
\newcommand\biggcexpct[2]{\mathbb{E}\bigg[#1\bigg\vert#2\bigg]}
\newcommand\Bigcexpct[2]{\mathbb{E}\Big[#1\Big\vert#2\Big]}

\newcommand\expctk[2]{\mathbb{E}_{#2}[#1]}
\newcommand\biggexpctk[2]{\mathbb{E}_{#2}\bigg[#1\bigg]}

\newcommand\prob[1]{\mathbf{P}[#1]}
\newcommand\biggprob[1]{\mathbf{P}\bigg[#1\bigg]}
\newcommand\Bigprob[1]{\mathbf{P}\Big[#1\Big]}

\newcommand\absv[1]{\vert#1\vert}
\newcommand\biggabsv[1]{\bigg\vert#1\bigg\vert}
\newcommand\Bigabsv[1]{\Big\vert#1\Big\vert}

\newcommand{\lone}{\mathbbm{1}}

\numberwithin{equation}{section} 
\numberwithin{theorem}{section}

\title{Gaussian Fluctuations for Linear Spectral Statistics of Deformed Wigner Matrices}

\author{Hong Chang Ji\footnote{Department of Mathematical Sciences, KAIST, Daejeon, 34141, Korea
		\newline email: \texttt{hcji@kaist.ac.kr}} 
	and Ji Oon Lee\footnote{Department of Mathematical Sciences, KAIST, Daejeon, 34141, Korea
		\newline email: \texttt{jioon.lee@kaist.edu}}}

\begin{document}
	
\maketitle

\begin{abstract}
We consider large-dimensional Hermitian or symmetric random matrices of the form $W=M+\vartheta V$ where $M$ is a Wigner matrix and $V$ is a real diagonal matrix whose entries are independent of $M$. For a large class of diagonal matrices $V$, we prove that the fluctuations of linear spectral statistics of $W$ for $C^{2}_{c}$ test function can be decomposed into that of $M$ and of $V$, and that each of those weakly converges to a Gaussian distribution. We also calculate the formulae for the means and variances of the limiting distributions.
\end{abstract}

\renewcommand{\thefootnote}{\fnsymbol{footnote}} 
\footnotetext{\emph{MSC2010 subject classifications.} 60B20, 60F05, 15B52.}
\renewcommand{\thefootnote}{\arabic{footnote}} 

\section{Introduction}

Ever since its discovery, the central limit theorem has been considered as one of the most fundamental concepts in probability theory. Corresponding to its motivation of studying the fluctuation of a sum of independent random variables, analogous objects in the random matrix theory stand out: for random matrices with large size $N$ and eigenvalues $\lambda_{1}^{(N)},\cdots,\lambda_{N}^{(N)}$, their linear eigenvalue statistics (LES) or linear spectral statistics (LSS), which is defined for an appropriate test function $\varphi$ as
\beq
\sum_{i=1}^{N}\varphi(\lambda_{i}^{(N)}).
\eeq

The fluctuations of LES have been studied by many different authors for various random matrix models including Wigner matrices \cite{Bai-Yao2005,Cabanal2001,Chatterjee2009,Chatterjee-Bose2004,Khorunzhy-Khorunzhenko-Pastur1995,Khorunzhy-Khorunzhenko-Pastur1996,Lytova-Pastur2009a,Shcherbina2011a}, sample covariance matrices \cite{Bai-Silverstein2004,Jonsson1982}, spiked Wigner matrices \cite{Guionnet2002,Baik-Lee2017}, Wishart ensembles \cite{Cabanal2001,Chatterjee2009,Chatterjee-Bose2004}, random band matrices \cite{Anderson-Zeitouni2006,Guionnet2002,Shcherbina2015}, and elliptic random matrices \cite{ORourke-Renfrew2016}. Let us consider the example of Wigner random matrices. Wigner matrices is $N\times N$ complex Hermitian or real symmetric random matrices whose upper triangular entries are independent and have mean zero and variance $1/N$. The celebrated Wigner semi-circle law states the convergence
\beq
\frac{1}{N}\sum_{i=1}^{N}\varphi(\lambda_{i}^{(N)}) \to \int_{\R}\varphi(\lambda)\rho_{sc}(\dd\lambda){\mathrel={\mathop:}}\bilin{\varphi}{\rho_{sc}}
\eeq
where $\rho_{sc}$ is the semi-circular distribution given by 
\beq
	\dd\rho_{sc}(x)\deq\frac{1}{2\pi}\sqrt{4-x^{2}}\lone_{[-2,2]}(x)\dd x.
\eeq
This result can be considered as an analogue of the law of large numbers, for being convergence of a random quantity toward a deterministic number. From this convergence it immediately follows that asymptotically the center of the fluctuation of the linear eigenvalue statistics is $\bilin{\varphi}{\rho_{sc}}$. With respect to the center, Bai and Yao proved in \cite{Bai-Yao2005} that under the existence and homogeneity of third and fourth moments of matrix entries, for analytic test function $\varphi$, the random variable
\beq
\sum_{i=1}^{N}\varphi(\lambda_{i}^{(N)})-N\bilin{\varphi}{\rho_{sc}}
\eeq
converges in distribution to a Gaussian random variable, giving also the explicit mean and variance of the limiting distribution.

In this paper we consider deformed Wigner matrices, given by $W_{N}=N^{-1/2}A_{N}+\vartheta_{N} V_{N}$ where $N^{-1/2}A_{N}$ is a real symmetric Wigner matrix and $V_{N}$ is a real, diagonal, random or deterministic matrix independent of $A_{N}$. Two matrices $N^{-1/2}A_{N}$ and $V_{N}$ are normalized so that each of them has eigenvalues of order one, and $\vartheta_{N}$ is a parameter that controls the order of deformation by $V_{N}$. Since the case of $\vartheta_{N}\ll N^{-1/2}$ or $\vartheta_{N}\sim N^{-1/2}$ results in another Wigner matrix, which has been studied widely by many authors, we focus on the case where $\vartheta_{N}\sim 1$ or $N^{-1/2}\ll\vartheta_{N}\ll1$. In other words, $\delta(N)N^{-1/2}\leq \theta_{N}\leq C$ where $C$ is a constant and $\delta_{N}$ tends to infinity with $N$.

Assuming that the empirical spectral distribution(ESD) of $V_{N}$
\beq
\wh{\nu}_{N} \deq \frac{1}{N}\sum_{i=1}^{N}\delta_{v_{i}^{(N)}},\quad V_{N}=\mathrm{diag}(v_{1}^{(N)},\cdots,v_{N}^{(N)})
\eeq
converges weakly (weakly in probability if $V_{N}$ is random) to a deterministic distribution $\nu$, it was proved in \cite{Pastur1972} that for $\theta\sim 1$ the ESD of $W$ 
\beq
	\rho_{N}\deq\frac{1}{N}\sum_{i=1}^{N}\delta_{\lambda_{i}^{(N)}},
\eeq
where $\lambda_{i}^{(N)}$'s are the eigenvalues of $W$, converges weakly in probability to a deterministic measure, $\rho_{fc}$. In the present paper, the limiting distribution is called the \emph{deformed semicircle law}, following \cite{Lee-Schnelli-Stetler-Yau2016}. Even though the distribution $\rho_{fc}$ (or its density function) is hard to describe explicitly in terms of $\nu$, its Stieltjes transform was characterized in \cite{Pastur1972} as the solution of an integral equation concerning $\nu$. Equivalently, the limiting measure $\rho_{fc}$ can be considered as the free additive convolution of the semicircle law and $\nu$. In \cite{Biane1997}, it was shown that $\rho_{fc}$ admits a density, which may be supported on multiple disjoint intervals. For simplicity, we impose some conditions on $\nu$ so that the limiting density is supported on a single compact interval. We also assume $\wh{\nu}_{N}$ is such that there is no outlying eigenvalues of $W$, i.e. the eigenvalues of $W$ stay close to the support of $\rho_{fc}$. As mentioned above, the results of \cite{Pastur1972} implies that whenever we are given a continuous bounded function $\varphi$, we immediately get the convergence
\beq \label{eq:phi_rho}
\bilin{\varphi}{\rho_{N}} \deq\int_{\R}\varphi \,\dd\rho_{N} = \frac{1}{N}\sum_{i=1}^{N}\varphi(\lambda_{i}^{(N)}) \longrightarrow \int_{\R}\varphi \,\dd\rho_{fc}
\eeq
where $\lambda_{i}$ are the eigenvalues and $\rho$ is the empirical spectral distribution of $W$.

Given the convergence of empirical distribution, the fluctuations of various statistics of the deformed Wigner matrices have been studied. The deformed Gaussian unitary ensemble (GUE) for a special case, where the eigenvalues of $V_{N}$ are $\pm a$ with the equal multiplicity, was considered in \cite{Aptekarev-Bleher-Kuijlaars2005,Bleher-Kuijlaars2004,Bleher-Kuijlaars2007} using the Deift/Zhou steepest descent method for the Riemann-Hilbert problem. For a general deformed Gaussian unitary ensemble (GUE), the maximal eigenvalue of $W_{N}$ was studied in \cite{Johansson2007}, where it was proved that the maximal eigenvalue of $W_{N}$ converges weakly to a Tracy-Widom or a Gaussian distribution, depending on the parameter $\vartheta_{N}$. In particular, the transition occurs at $\vartheta_{N}\sim N^{-1/6}$, so that the limiting distribution is Gaussian for $\vartheta_{N}\gg N^{-1/6}$ and the Tracy-Widom for $\vartheta_{N}\ll N^{-1/6}$. Also for non-Gaussian generic Wigner ensemble $W_{N}$, many other statistics has been studied, including the eigenvectors in \cite{Benigni2017,Lee-Schnelli2013,Lee-Schnelli2016}, `four-moment theorem' in \cite{OrourkeVu2014}, extremal eigenvalues in \cite{Lee-Schnelli2015,Lee-Schnelli2016,Lee-Schnelli-Stetler-Yau2016}, and bulk universality in \cite{Lee-Schnelli-Stetler-Yau2016}.

In many cases, the Stieltjes transform $m_{fc}$ of $\rho_{fc}$ and the corresponding Green function $m_{N}$ of $W_{N}$ are extensively used in the analysis of deformed Wigner matrices. In particular, a ``local law'' was first established in \cite{Lee-Schnelli2013}, asserting that the normalized trace of the resolvent of $W_{N}$ is almost of order $N^{-1}$ away from that of the limiting distribution $\rho_{fc}$ and the non-diagonal entries of the resolvent of $W_{N}$ cannot be much larger than $N^{-1/2}$ under the macroscopic scale of $\im z=\eta\sim1$ where $z$ is the spectral parameter for the resolvent. In the same paper, it was also observed that the fluctuation of the Green function of $W_{N}$ can be separated by two parts, one coming from the fluctuation of the Wigner matrix $A_{N}$ and the other from that of the diagonal matrix $V_{N}$. For the diagonal entries of the resolvent, it was proved in \cite{Lee-Schnelli-Stetler-Yau2016} that they also are about $N^{-1/2}$ away from its center, coming from the free additive convolution of the empirical spectral measure of $V_{N}$ and the semicircle distribution. Finally, \cite{Knowles-Yin2016} established an analogue of the local law for the deformed Wigner matrix with a non-diagonal matrix $V_{N}$, by proving a sufficient condition concerning the corresponding local law of deformed Gaussian ensembles whose deformation is the diagonalization of $V_{N}$. The rigidity of eigenvalues and the edge universality is proved under the same condition, which, together with the local law, are the central estimates of this paper. The local law for deformed Wigner matrices was a crucial input in the universality results in \cite{Lee-Schnelli2013,Lee-Schnelli2016} and it was also used in \cite{OrourkeVu2014} to establish a `four-moment theorem' by Lindeberg's replacement strategy. 

As addressed above, we are interested in the fluctuation of $\bilin{\varphi}{\rho_{N}}$ in~\eqref{eq:phi_rho} under proper normalization. Heuristically, the fluctuation of $\bilin{\varphi}{\rho_{N}}$ can be separated into two components, one coming from the fluctuation of Wigner matrix $\frac{1}{\sqrt{N}}A_{N}$, and the other coming from that of the diagonal matrix $V_{N}$. To be specific, as shown in \cite{Lee-Schnelli2013}, the fluctuation of Stieltjes transform of $\rho_{N}$ can be analyzed with respect to two different centers: Stieltjes transform of the limiting distribution $\rho_{fc}$ and that of the free convolution of the semicircle law with the empirical spectral distribution of $V$, i.e. $\wh{\rho}^{(N)}_{fc}\deq \rho_{sc}\boxplus\wh{\nu}_{N}$. Under the macroscopic scale of $\eta=\Im z\sim 1$, where $z$ is the argument of Stieltjes transform, it is shown that the fluctuation with respect to the former center has order $N^{-1}$, and that with respect to the latter center has order $\vartheta_{N} N^{-1/2}$. For analytic test function $\varphi$, we prove the convergence of LES utilizing the Stieltjes transform and analyze the fluctuation of $\bilin{\varphi}{\rho_{N}}$ in the following two settings:
\begin{itemize}
	\item With center $\bilin{\varphi}{\wh{\rho}_{fc}^{(N)}}$, normalized by $N$,
	\item With center $\bilin{\varphi}{\rho_{fc}}$, normalized by $\sqrt{N}$.
\end{itemize}
We also prove that the central limit theorem, i.e. Gaussian convergence of the centered quantity $\bilin{\varphi}{\rho_{N}}-\expct{\bilin{\varphi}{\rho_{N}}}$, can be extended for $C^{2}$ test function, for random $V$ and deterministic $V$.

By adapting the proof in \cite{Bai-Yao2005}, we are able to prove the corresponding Gaussian convergence for $\bilin{\varphi}{\rho_{N}}-\bilin{\varphi}{\wh{\rho}_{fc}^{(N)}}$ for analytic test function $\varphi$ when the entries of $V_{N}$ are deterministic. Along the proof, the main difficulty is the absence of symmetry in $\rho_{fc}$ or $\wh{\rho}_{fc}^{(N)}$, whereas $\rho_{sc}$ has many exploitable features that make explicit calculations possible. For example, while we do not have a simple polynomial equation for the Stieltjes transform of $\rho_{fc}$ in general, the Stieltjes transform $m_{sc}$ of $\rho_{sc}$ satisfies the quadratic equation $m_{sc}(z)^{2}+zm_{sc}(z)+1=0$, solely from which many properties arise(e.g. $\absv{m_{sc}}\leq 1$ and $\absv{m_{sc}(z)+z}\geq 1$). Moreover, even if several applicable properties of $\rho_{fc}$ were obtained, while dealing with $\wh{\rho}_{fc}^{(N)}$, one still needs to extend those properties to $\wh{\rho}_{fc}^{(N)}$. The difficulties are handled using the detailed analysis of the Stieltjes transforms of $\wh{\rho}_{fc}^{(N)}$ and $\rho_{fc}$, given in \cite{Lee-Schnelli-Stetler-Yau2016}. For example, a \emph{stability bound}, which corresponds to the inequality $\absv{m_{sc}(z)}\leq 1$ for the Stieltjes transform of $\rho_{sc}$ and is proved therein, is used widely throughout the paper. To overcome the second difficulty, following \cite{Lee-Schnelli-Stetler-Yau2016}, we introduce an event where the behavior of $\wh{\nu}_{N}$ resembles that of $\nu$ and prove that $\wh{\rho}_{fc}^{(N)}$ also behaves similarly as $\rho_{fc}$ on this event.

The analysis on the behavior of $\bilin{\varphi}{\rho_{N}}-\bilin{\varphi}{\wh{\rho}_{fc}^{(N)}}$ in the proof of the first part then implies that its contribution to the fluctuation of $\bilin{\varphi}{\rho_{N}}-\bilin{\varphi}{\rho_{fc}}$ is negligible. In particular, such inferiority in magnitude holds also for coupling parameters following asymptotic of the form $\vartheta_{N}\sim N^{-1/2}\sqrt{\log N}$, since the proof of the first part enables us to be free of so-called \emph{high-probability bounds}. Therefore, in the second part, the problem reduces to analyzing the fluctuation of $\wh{\rho}_{fc}^{(N)}$ with respect to the center $\rho_{fc}$. As easily seen, we may conceive it as rising from the fluctuation of $\wh{\nu}_{N}$, which results in the classical central limit type behavior. The proof heavily depends on the analysis of the Stieltjes transforms of $\wh{\rho}_{fc}^{(N)}$ and $\rho_{fc}$ using self-comparison.

The paper consists of 5 sections and 5 appendices, including the introduction. Section 2 is dedicated to preliminary materials such as definitions, our model and assumptions on it, and the precise statements of our results. Section 3, where we provide the strategy of our proof that uses mainly probabilistic and complex analytic methods, contains the statements of Propositions~\ref{prop:Gproc} and \ref{prop:Gprocsqrt} that are the central parts of the proof. In the same section we also collect some lemmas to be used in the rest of the paper, including the \emph{local deformed semicircle laws}, whose proofs are given in the \ref{sec:lemmas3.2}. 

Section 4 and 5 are devoted respectively to the proofs of Proposition \ref{prop:Gproc} and Propositions \ref{prop:Gprocsqrttheta} and \ref{prop:Gprocsqrt}, whose ingredients are stated therein and proved in the following sections. In Appendices A and B, we obtain the convergence of the mean and variance of $m_{N}-\wh{m}_{fc}$, respectively. \ref{sec:tight} gives the proof of the tightness of the processes given in the statements of Propositions \ref{prop:Gproc}--\ref{prop:Gprocsqrt}, which also is a part of the proof of the propositions. \ref{sec:lemmarl0igno} is dedicated to the proof of Lemma \ref{lem:rl0igno}, the final ingredient of the proof of the main theorems. Also \ref{sec:lemmarl0igno} contains the proof of Lemma \ref{lem:varbound}, that gives us the bound on $\Var{m_{N}(z)}$ needed to extend the central limit theorem from analytic test functions to $C^{2}$ test functions. Finally, \ref{sec:lemmas3.2} provides the proof of the lemmas given in Section \ref{sec:prelim}.

\begin{notation}
	Throughout the paper, we use $C$ or $c$ to denote a constant that is independent of $N$. Even if the constant is different from one place to another, we may use the same notation $C$ or $c$ as long as it does not depend on $N$ for the convenience of presentation.
\end{notation}

\begin{notation}
	For positive numbers $a\equiv a_{N}$ and $b\equiv b_{N}$ depending on $N$, we write $a_N \ll b_N$ or $a_N=o(b_N)$ to indicate that $\frac{a_N}{b_N}\to 0$ as $N\to\infty$. We write $a_N=O(b_N)$ when there exists a constant $C>1$ independent of $N$ such that $a_N \leq Cb_{N}$ and $a_N\sim b_N$ when $C^{-1}a_N\leq b_N\leq C a_N$.
\end{notation}

\section{Definitions, Assumptions and Main Results}

\subsection{Definitions}

\begin{definition}
	For a probability distribution $\rho$ on $\R$, the \emph{Stieltjes transform} of $\rho$ is defined by 
	\beq
	m_{\rho}(z)\deq\int_{\R}\frac{1}{x-z}\dd\rho(x),\quad z\in\C^{+}.
	\eeq
\end{definition}

\begin{assumption}\label{assump:Wigner}
	Let $\{A_{ij}:i\le j\in\N\}$ be a collection of independent real random variables satisfying the following:
	\begin{enumerate}[(i)]
		\item $\expct{A_{ij}}=0$.
		
		\item For $i<j$, $\expct{A_{ij}^{2}}=1$, $\expct{A_{ij}^{3}}=W_{3}$ and $\expct{A_{ij}^{4}}=W_{4}$ for some constant $W_{3}\in\R$ and $W_{4}>0$.
		
		\item For all $i\in\N$, $\expct{A_{ii}^{2}}=w_{2}$, for some constant $w_{2}\geq0$. 
		
		\item For any $k\geq3$ there is a constant $c_{k}$ such that,
		\beq
		\sup_{1\leq i,j\leq N, N\in\N}\expct{\absv{A_{ij}}^{2}}\leq c_{k}.
		\eeq
	\end{enumerate}
	Define $A_{ji}\deq A_{ij}$ for $i<j$ and let $A=A_{N}=(A_{ij})_{i,j=1}^{N}$ be the random matrix with entries $A_{ij}$. 
\end{assumption}

\begin{definition}
	Let $V=V_{N}=\diag(v_{1}^{(N)},v_{2}^{(N)},\cdots,v_{N}^{(N)})$ be an $N\times N$ real diagonal, random or deterministic matrix, with empirical spectral distribution $\wh{\nu}$, i.e.
	\beq
	\wh{\nu}\equiv\wh{\nu}_{N}=\frac{1}{N}\sum_{i}\delta_{v_{i}^{(N)}}.
	\eeq
	Also, we define $\nu$ to be a deterministic, centered, compactly supported probability measure on $\R$.
\end{definition}

In the following, every notation with hat stands for a quantity depending on $\wh{\nu}$(thus on $N$), rather than $\nu$. And throughout the paper, we assume that the empirical spectral distribution $\wh{\nu}_{N}$ of $V_{N}$ converges weakly to $\nu$ (in probability if $V_{N}$ is random). Moreover, we assume the following assumptions on $\nu$ and $V_{N}$:
\begin{assumption}\label{assump:esdvconv}
	We assume that the entries $(v_{1}^{(N)},\cdots,v_{N}^{(N)})$ of $V_{N}$ satisfy the following:
	\begin{enumerate}[(i)]
		\item If $V_{N}$ is a \textbf{random} matrix, let $\{v_{i}^{(N)}:1\leq i\leq N\}$ be a collection of i.i.d. random variables with law $\nu$, independent of $A$.
		
		\item If $V_{N}$ is a \textbf{deterministic} matrix, we assume
		\beq\label{eq:alpha0}
		\max_{z\in\caD}\biggabsv{ \frac{1}{x-z}\dd\wh{\nu}_{N}(x)-\int\frac{1}{x-z}\dd\nu(x)}=O(N^{-\alpha_{0}}),
		\eeq
		for any fixed compact set $\caD\subset\C^{+}$ with $\caD\cap\supp\nu=\emptyset$, for some $\alpha_{0}>0$.
	\end{enumerate}
\end{assumption}

\begin{assumption}\label{assump:regesdv}
	Let $I_{\nu}$ denote the smallest interval such that $\supp \nu\subset I_{\nu}$. We assume
	\beq
	\inf_{x\in I_{\nu}}\int\frac{1}{(v-x)^{2}}\dd\nu(v)\geq 1+\varpi
	\eeq
	for some $\varpi>0$.
	
	Letting $I_{\wh{\nu}_{N}}$ denote be the smallest interval such that $\supp\wh{\nu}_{N}\subset I_{\wh{\nu}_{N}}$, we also assume the similar condition to $\wh{\nu}_{N}$:
	\begin{enumerate}[(i)]
		\item For \textbf{random} $\{v_{i}\}$, we assume
		\beq
		\biggprob{\inf_{x\in I_{\wh{\nu}_{N}}}\int\frac{1}{(v-x)^{2}}\dd\wh{\nu}_{N}(x)\geq 1+\varpi}\geq 1-N^{-\frt}
		\eeq
		for some $\frt>0$.
		
		\item For \textbf{deterministic} $\{v_{i}\}$, we assume
		\beq
		\inf_{x\in I_{\wh{\nu}_{N}}}\int\frac{1}{(v-x)^{2}}\dd\wh{\nu}_{N}(x)\geq 1+\varpi
		\eeq
		for sufficiently large $N$.
	\end{enumerate}
\end{assumption}

\begin{assumption}\label{assump:coupling}
	Let $\{\vartheta\equiv\vartheta_{N}\}\subset\Theta_{\varpi}\deq[0,1+\varpi]$ be a sequence of parameters such that $\vartheta_{N}\gg N^{-1/2}$ and $\lim_{N\to\infty}\vartheta_{N}=\vartheta_{\infty}\in\Theta_{\varpi}$.
\end{assumption}

\begin{definition}
	Let $A$ and $V_{N}$ satisfy Assumption \ref{assump:Wigner},\ref{assump:esdvconv}, and \ref{assump:regesdv}. We define 
	\beq
	W_{N}=\frac{1}{\sqrt{N}}A_{N}+\vartheta_{N}V_{N}
	\eeq
	to be the deformed Wigner matrix. We denote the resolvent of $W_{N}$ by 
	\beq
		R^{\vartheta_{N}}_{N}(z)\deq (W_{N}-zI_{N})^{-1},
	\eeq
	and its normalized Green function by
	\beq
	m^{\vartheta_{N}}(z)\equiv m_{N}^{\vartheta_{N}}(z)\deq \frac{1}{N}\Tr R_{N}^{\vartheta_{N}}(z)\equiv\tr R_{N}^{\vartheta_{N}}(z).
	\eeq
\end{definition}

\begin{remark}
	For $\vartheta_{N}\sim N^{-1/2}$, $W_{N}$ itself is another Wigner matrix with the variance of its diagonal entries increased. On the other hand for $\vartheta_{N}\ll N^{-1/2}$, the fluctuation is dominated by that of $A_{N}$ henceforth the fluctuation of LES is identical to that of the Wigner matrices, which is thoroughly studied in \cite{Bai-Yao2005}.  Therefore, we here focus on the case $\vartheta_{N}\gg N^{-1/2}$.
\end{remark}

\begin{notation}
	For convenience, we denote the distribution of $\vartheta_{N} v_{1}$ and the empirical spectral distribution of $\vartheta_{N} V_{N}$ by $\nu^{\vartheta_{N}}$ and $\wh{\nu}_{N}^{\vartheta_{N}}$, respectively. Similarly, let $m_{\nu}^{\vartheta_{N}}$ and $m_{\wh{\nu}_{N}}^{\vartheta_{N}}$ be the Stieltjes transforms of $\nu^{\vartheta_{N}}$ and $\wh{\nu}_{N}^{\vartheta_{N}}$, respectively.
\end{notation}

\begin{lemma}[\cite{Pastur1972}]	
	For $\vartheta\in\Theta_{\varpi}$, the following self-consistent equation
	\beq\label{eq:funceq}
	m_{fc}^{\vartheta}(z)=\int_{\R}\frac{1}{x-z-m_{fc}^{\vartheta}(z)}\dd\nu^{\vartheta}(x) =\int_{\R}\frac{1}{\vartheta x-z-m_{fc}^{\vartheta}(z)}\dd\nu(x)
	\eeq
	has unique solution among the class of analytic functions with $\im m_{fc}^{\vartheta}(z)>0$ for $\im z>0$. Also the solution $m_{fc}^{\vartheta}(z)$ is the Stieltjes transform a probability measure, denoted by $\rho_{fc}^{\vartheta}$.
	
	The same holds also for the equation
	\beq\label{eq:funceqhat}
	\wh{m}_{fc}^{\vartheta}(z)=\int \frac{1}{x-z-\wh{m}_{fc}^{\vartheta}(z)}\dd\wh{\nu}_{N}^{\vartheta}(x),\quad \Im \wh{m}_{fc}^{\vartheta}(z)\geq0,\quad\text{for }z\in\C^{+}.
	\eeq
	In this case, we denote the corresponding measure by $\wh{\rho}_{fc}^{\vartheta}$.
\end{lemma}

In particular, $\rho_{fc}^{\vartheta}=\rho_{sc}\boxplus \nu^{\vartheta}$ and $\wh{\rho}_{fc}^{\vartheta}=\rho_{sc}\boxplus\wh{\nu}_{N}^{\vartheta}$, the free additive convolution of the semi-circular distribution $\rho_{sc}$ with $\nu^{\vartheta}$ and $\wh{\nu}^{\vartheta}$, respectively.

\begin{lemma}[Proposition 3 and Corollary 4 of \cite{Biane1997}]
	The free convolution measure $\rho_{fc}^{\vartheta}=\rho_{sc}\boxplus\nu^{\vartheta}$ satisfies $\limsup_{\eta\to0^{+}}\Im m_{fc}^{\vartheta}(E+\ii\eta)<\infty$ for $E\in\R$, hence is absolutely continuous. Its density, which also is denoted by $\rho_{fc}^{\vartheta}$, is analytic on $\{E\in\R:\rho^{\vartheta}_{fc}(E)>0\}$.
\end{lemma}

\begin{theorem}[\cite{Pastur1972}, Deformed semi-circle law]
	The empirical spectral distribution of $W_{N}$ converges in probability to $\rho_{fc}^{\vartheta_{\infty}}$.
\end{theorem}

In \cite{Pastur2006}, it turned out that $\rho_{fc}^{\vartheta}$ being supported on a single interval is crucial for the Gaussian convergence of the LES. Although stated below, we here state the fact here as it is need for the statements of our results.
\begin{lemma}
	Suppose that $\nu$ satisfies Assumption~\ref{assump:regesdv}. Then for any $\vartheta\in\Theta_{\varpi}$, there exists $L_{-}^{\vartheta},L_{+}^{\vartheta}\in\R$ with $L_{-}^{\vartheta}<0<L_{+}^{\vartheta}$ such that $\supp \rho_{fc}^{\vartheta}=[L_{-}^{\vartheta},L_{+}^{\vartheta}]$.
\end{lemma}

\begin{remark}
	For $\vartheta=0$, one immediately gets $\rho_{fc}^{\vartheta}=\rho_{sc}$, $m_{fc}^{\vartheta}=m_{sc}$, and $L^{\vartheta}_{\pm}=\pm2$.
\end{remark}

\subsection{Statement of the result}

\begin{theorem}\label{thm:main}
	Suppose that $W_{N}$ is the deformed Wigner matrix with \textbf{deterministic} $V_{N}$ satisfying Assumption \ref{assump:Wigner}, $\nu$ and $\wh{\nu}_{N}$ satisfy Assumptions \ref{assump:esdvconv} and \ref{assump:regesdv}, and $\vartheta_{N}$ satisfies Assumption~\ref{assump:coupling}. Then for each $\varphi\in C(\R)$ with compact support that is analytic on an open neighborhood of $[L_{-}^{\vartheta_{\infty}},L_{+}^{\vartheta_{\infty}}]$, the random variable 
	\beq\label{eq:TN}
	T_{N}^{\vartheta_{N}}(\varphi) \deq \sum_{i=1}^{N}\varphi(\lambda_{i}^{(N)})-N\int_{\R}\varphi(x)\dd\wh{\rho}_{fc}^{\vartheta_{N}}(x)
	\eeq
	converges in distribution to the Gaussian random variable $T(\varphi)$ with mean $M^{\vartheta_{\infty}}(\varphi)$ and variance $V^{\vartheta_{\infty}}(\varphi)$ given as follows:
	\beq
	M^{\vartheta_{\infty}}(\varphi) =-\frac{1}{2\pi\ii}\oint_{\Gamma}\varphi(z)b^{\vartheta_{\infty}}(z)\dd z,
	\eeq
	\beq
	V^{\vartheta_{\infty}}(\varphi) =\frac{1}{(2\pi\ii)^{2}}\oint_{\Gamma}\oint_{\Gamma}\varphi(z_{1})\varphi(z_{2})\Gamma^{\vartheta_{\infty}}(z_{1},z_{2})\dd z_{1}\dd z_{2},
	\eeq
	where
	\beq
	b^{\vartheta_{\infty}}(z)=\frac{(m_{fc}^{\vartheta_{\infty}})''(z)}{2(1+(m_{fc}^{\vartheta_{\infty}})'(z))^{2}}\bigg[(w_{2}-1)+(m_{fc}^{\vartheta_{\infty}})'(z)+(W_{4}-3)\frac{(m_{fc}^{\vartheta_{\infty}})'(z)}{1+(m_{fc}^{\vartheta_{\infty}})'(z)}\bigg],
	\eeq
	\begin{multline}
	\Gamma^{\vartheta_{\infty}}(z_{1},z_{2})\equiv\frac{\partial^{2}}{\partial z_{1}\partial z_{2}}\wt{\Gamma}^{\vartheta_{\infty}}(z_{1},z_{2})
	=\frac{\partial^{2}}{\partial z_{1}\partial z_{2}}\Big[(w_{2}-2)I+(W_{4}-3)I^{2}-2\log(1-I)\Big]\\
	=(w_{2}-2)\frac{\partial^{2}I}{\partial z_{1} \partial z_{2}} +(W_{4}-3)\bigg(I\frac{\partial^{2}I}{\partial z_{1}\partial z_{2}} +\frac{\partial I}{\partial z_{1}}\frac{\partial I}{\partial z_{2}}\bigg) 
	+\frac{2}{(1-I)^{2}}\bigg(\frac{\partial I}{\partial z_{2}}\frac{\partial I}{\partial z_{1}} +(1-I)\frac{\partial^{2} I}{\partial z_{1}\partial z_{2}}\bigg),
	\end{multline}
	\beq
	I(z_{1},z_{2})\equiv I^{\vartheta_{\infty}}(z_{1},z_{2}) =\int_{\R}\frac{1}{(\vartheta_{\infty} x-z_{1}-m_{fc}^{\vartheta_{\infty}}(z_{1}))(\vartheta_{\infty} x-z_{2}-m_{fc}^{\vartheta_{\infty}}(z_{2}))}\dd\nu(x),
	\eeq
	and $\Gamma$ is a rectangular contour with vertices $(a_{\pm}\pm\ii v_{0})$ so that $\pm(a_{\pm}-L_{\pm}^{\vartheta_{\infty}})>0$ and $\Gamma$ lies within the analytic domain of $\varphi$.
\end{theorem}

\begin{remark}
	The theorem above implies the Gaussian convergence of centered random variable
	\beq
		\sum_{i=1}^{N}\varphi(\lambda_{i}^{(N)})-\expct{\varphi(\lambda_{i}^{(N)})},
	\eeq
	since the last term being subtracted in \eqref{eq:TN} is deterministic, $M^{\varphi_{\infty}}$ comes from the asymptotic difference between the mean and $\bilin{\varphi}{\wh{\rho}_{fc}^{\vartheta_{N}}}$.
\end{remark}
\begin{remark}
	For $\vartheta_{\infty}=0$, we have $m_{fc}^{\vartheta_{\infty}}(z)=m_{sc}(z)$ and $I^{\vartheta_{\infty}}(z_{1},z_{2})=m_{sc}(z_{1})m_{sc}(z_{2})$, so that
	\beq
	b(z)=m_{sc}(z)^{3}(1+m_{sc}'(z))\big((w_{2}-1)+m'_{sc}(z)+(W_{4}-3)m_{sc}(z)^{2}\big),
	\eeq
	\beq
	\Gamma(z_{1},z_{2})=m_{sc}'(z_{1})m_{sc}'(z_{2})\bigg((w_{2}-2)+2(W_{4}-3)m_{sc}(z_{1})m_{sc}(z_{2})+\frac{2}{(1-m_{sc}(z_{1})m_{sc}(z_{2}))^{2}}\bigg),
	\eeq
	which coincides with the limiting formulae given in \cite{Bai-Yao2005} and \cite{Baik-Lee2017}. Therefore $M(\varphi)$ and $V(\varphi)$ are given by
	\beq
	M(\varphi)=\frac{1}{4}(\varphi(2)+\varphi(-2))-\frac{1}{2}\tau_{0}(\varphi)+(w_{2}-2)\tau_{2}(\varphi)+(W_{4}-3)\tau_{4}(\varphi)
	\eeq
	and
	\beq
	V(\varphi)=(w_{2}-2)\tau_{1}(\varphi)^{2}+2(W_{4}-3)\tau_{2}(\varphi)^{2}+2\sum_{\ell=1}^{\infty}\ell\tau_{\ell}(\varphi)^{2},
	\eeq
	where
	\beq
	\tau_{\ell}(\varphi)=\frac{1}{2\pi}\int_{-\pi}^{\pi}\varphi(2\cos\theta)\cos(\ell\theta)\dd\theta.
	\eeq
\end{remark}

\begin{remark}
	For complex Hermitian $A$, with the additional assumption $\expct{A_{ij}^{2}}=0$, the same result holds with $(w_{2}-2)$ and $(W_{4}-3)$ replaced by $(w_{2}-1)$ and $(W_{4}-2)$, respectively.
\end{remark}

\begin{theorem}\label{thm:mainsqrt}
	Suppose that $W_{N}$ is the deformed Wigner matrix with \textbf{random} $V_{N}$, satisfying the assumptions in Theorem~\ref{thm:main}. Then for any test function $\varphi$ satisfying the conditions in Theorem~\ref{thm:main}, the random variable
	\beq
	S_{N}^{\vartheta_{N}}(\varphi)\deq\frac{1}{\sqrt{N}\vartheta_{N}}\sum_{i=1}^{N}\bigg[\varphi(\lambda_{i}^{(N)})-\int_{\R}\varphi(x)\dd\rho_{fc}^{\vartheta_{N}}(x)\bigg]
	\eeq
	converges in distribution to the Gaussian random variable $S(\varphi)$ with mean zero and variance $\wt{V}^{\vartheta_{\infty}}(\varphi)$ given by
	\begin{itemize}
		\item For $\vartheta_{\infty}\neq0$,
		\begin{multline}
		\wt{V}^{\vartheta_{\infty}}(\varphi) =-\frac{1}{4\pi^{2}\vartheta_{\infty}^{2}} \oint_{\Gamma}\oint_{\Gamma}\bigg[\varphi(z_{1})\varphi(z_{2})(1+(m_{fc}^{\vartheta_{\infty}})'(z_{1}))(1+(m_{fc}^{\vartheta_{\infty}})'(z_{2})) \\
		\big(I(z_{1},z_{2})-m_{fc}^{\vartheta_{\infty}}(z_{1})m_{fc}^{\vartheta_{\infty}}(z_{2})\big)\bigg]\dd z_{1}\dd z_{2}
		\end{multline}
		
		\item For $\vartheta_{\infty}=0$,
		\beq
		\wt{V}^{0}=\wt{V}(\varphi)=\Var{v_{1}}\frac{1}{(2\pi \ii)^{2}} \oint_{\Gamma}\oint_{\Gamma}\varphi(z_{1})\varphi(z_{2})m_{sc}'(z_{1})m_{sc}'(z_{2})\dd z_{1}\dd z_{2} =\tau_{1}(\varphi)^{2},
		\eeq
	\end{itemize}
	where $\Gamma$ and $I$ are given as in Theorem~\ref{thm:main}.
\end{theorem}

\begin{remark}
	As in the previous theorem, Theorem \ref{thm:mainsqrt} also implies the weak convergence of the corresponding centered random variable, and the asymptotic difference of mean converges to $0$ in this case.
\end{remark}

\begin{remark}
	The result holds also for complex Wigner matrix $A_{N}$ without any modification, since the fluctuation of $\wh{\rho}_{fc}^{\vartheta_{N}}$ dominates and that of $A_{N}$ is neglected, as shown in  {Section~\ref{sec:Gaussianprocsqrtprf}}.
\end{remark}

\begin{remark}
	In both of the theorems, limiting variances $V^{\vartheta_{\infty}}(\varphi)$ and $\wt{V}^{\vartheta_{\infty}}(\varphi)$ can be expressed in terms of double integral over the interval $[L_{-},L_{+}]$, by deforming the contour $\Gamma$ into $[L_{-},L_{+}]\pm\ii 0$ following \cite{Bai-Yao2005}. The integral involves continuous extension of $m_{fc}^{\vartheta_{\infty}}$ to $\C_{+}\cup \R$. 	
\end{remark}

Given the Gaussian convergence of linear statistics for analytic functions, now we extend the result to $C^{2}$ functions:
\begin{theorem}
	For each $M>\max\{1-L_{-},L_{+}+1\}$, denote by $C^{2}_{M}$ the space of $C^{2}$ real functions supported in $[-M,M]$, equipped with the $C^{2}$ norm defined by
	\beq
	\norm{\varphi}_{C^{2}_{M}}=\sup_{x\in[-M,M]}\big(\absv{\varphi(x)}+\absv{\varphi'(x)}+\absv{\varphi''(x)}\big).
	\eeq
	Then $V^{\vartheta_{\infty}}$ and $\wt{V}^{\vartheta_{\infty}}$ extend to a continuous quadratic functionals on $C^{2}_{M}$ for any $M$. Also, for each $\varphi\in C^{2}_{M}$, we have the following:
	\begin{itemize}
		\item Under the assumptions of Theorem \ref{thm:main}, the centralized random variable
		\beq
			T_{N}^{\vartheta_{N}}-\expct{T_{N}^{\vartheta_{N}}}=\sum_{i=1}^{N}\varphi(\lambda_{i}^{(N)})-\expct{\varphi(\lambda_{i}^{(N)})}
		\eeq
		converges in distribution to the centered Gaussian random variable with variance $V^{\vartheta_{\infty}}(\varphi)$.
		
		\item Under the assumptions of Theorem \ref{thm:mainsqrt}, the centralized random variable
		\beq
			S_{N}^{\vartheta_{N}}-\expct{S_{N}^{\vartheta_{N}}}
		\eeq
		converges in distribution to the centered Gaussian random variable with variance $\wt{V}^{\vartheta_{\infty}}(\varphi)$.
	\end{itemize}
\end{theorem}

\section{Strategy of the Proof}

\begin{notation}
	In the following sections, for simplicity, we omit the subscript $N$ or superscript $\vartheta$ as long as there is no ambiguity. For example, we denote $\vartheta\equiv\vartheta_{N}$, $\wh{\nu}\equiv\wh{\nu}_{N}$, and $m(z)\equiv m_{N}(z)\equiv m_{N}^{\vartheta_{N}}$.
\end{notation}

\subsection{CLT for analytic test function}

For each $x\in[L_{1}^{\vartheta_{\infty}},L_{2}^{\vartheta_{\infty}}]$, the Cauchy integral formula gives 
\beq
\varphi(x)=\frac{1}{2\pi \ii}\oint_{\Gamma}\frac{\varphi(z)}{z-x}dz,
\eeq
where $\Gamma$ is the rectangular contour with vertices given by $(a_{\pm}\pm iv_{0})$ where $a_{+}-L_{1}^{\vartheta_{\infty}}, L_{2}^{\vartheta_{\infty}}-a_{-}$ and $v_{0}$ are small enough(so that $\Gamma$ lies in the analytic domain of $\varphi$) but fixed positive real numbers. Then denoting by $\rho_{N}$ the empirical spectral distribution of $W_{N}$, we have the equality
\beq
T_{N}^{\vartheta}(\varphi)=N\int_{\R}\varphi(x)(\rho_{N}-\wh{\rho}_{fc}^{\vartheta})(\dd x) =N\oint_{\Gamma}\int_{\R}\frac{\varphi(z)}{z-x}(\rho_{N}-\wh{\rho}_{fc}^{\vartheta})(\dd x)dz =-\oint_{\Gamma}\varphi(z)\xi_{N}^{\vartheta}(z)dz
\eeq
and
\beq
S_{N}^{\vartheta}(\varphi)=-\oint_{\Gamma}\varphi(z)\wt{\xi}_{N}^{\vartheta}(z)dz
\eeq
where 
\beq
\xi_{N}^{\vartheta}(z)\deq N(m_{N}^{\vartheta}(z)-\wh{m}_{fc}^{\vartheta}(z)) \quad\text{and}\quad \wt{\xi}^{\vartheta}_{N}(z)\deq \frac{\sqrt{N}}{\theta}(m_{N}^{\vartheta}(z)-m_{fc}^{\vartheta}(z)),
\eeq
whenever the eigenvalues of $\rho_{N}$ and $\wh{\rho}_{fc}^{\vartheta}$ are contained in $[a_{-},a_{+}]$, so that the equality holds with probability greater than $1-cN^{-\frt}$ for sufficiently large $N$. We then decompose the contour $\Gamma$ into $\Gamma_{u}\cup\Gamma_{d}\cup\Gamma_{l}\cup\Gamma_{r}\cup\Gamma_{0}$, where
\beq
\begin{array}{l}
	\Gamma_{u}\deq \{z=x+\ii v_{0}:x\in[a_{-},a_{+}]\}, \\
	\Gamma_{d}\deq \{z=x-\ii v_{0}:x\in[a_{-},a_{+}]\}, \\
	\Gamma_{l}\deq \{z=a_{-}+\ii v:N^{-\delta}\le \absv{v}\le v_{0}\}, \\
	\Gamma_{r}\deq \{z=a_{+}+\ii v:N^{-\delta}\leq \absv{v}\le v_{0}\}, \\
	\Gamma_{0}\deq \{z=a_{\pm}+\ii v:\absv{v}\le N^{-\delta}\}. \\
\end{array}
\eeq
Each path is given the linear parametrization $[0,1]\to\Gamma_{\#}$, which is also denoted by $\Gamma_{\#}$.

In sections below, we prove the following propositions and lemmas:
\begin{proposition}\label{prop:Gproc}
	Suppose that $V$ is \textbf{deterministic}. For a fixed constant $c>0$ and a path $\caK\subset\{\Im z>c\}$, the process $\{\xi_{N}^{\vartheta}(z):z\in\caK\}$ converges weakly to the Gaussian process $\{\xi^{\vartheta_{\infty}}(z):z\in\caK\}$ with the mean $b^{\vartheta_{\infty}}(z)$ given by
	\beq
	b^{\vartheta_{\infty}}(z)=\frac{1}{2}\frac{(m_{fc}^{\vartheta_{\infty}})''(z)}{(1+(m_{fc}^{\vartheta_{\infty}})'(z))^{2}}\bigg[(w_{2}-1)+(m_{fc}^{\vartheta_{\infty}})'(z)+(W_{4}-3)\frac{(m_{fc}^{\vartheta_{\infty}})'(z)}{1+(m_{fc}^{\vartheta_{\infty}})'(z)}\bigg]
	\eeq
	and the covariance $\Gamma^{\vartheta_{\infty}}(z_{1},z_{2})$ defined by
	\begin{equation}
	(w_{2}-2)\frac{\partial^{2}I}{\partial z_{1} \partial z_{2}} +(W_{4}-3)\bigg(I\frac{\partial^{2}I}{\partial z_{1}\partial z_{2}} +\frac{\partial I}{\partial z_{1}}\frac{\partial I}{\partial z_{2}}\bigg)
	+\frac{2}{(I-1)^{2}}\bigg(\frac{\partial I}{\partial z_{2}}\frac{\partial I}{\partial z_{1}}+(1-I)\frac{\partial^{2} I}{\partial z_{1}\partial z_{2}}\bigg),
	\end{equation}
	where 
	\beq
	I(z_{1},z_{2})\equiv I^{\vartheta_{\infty}}(z_{1},z_{2}):=\int_{\R}\frac{1}{(\vartheta_{\infty} x-z_{1}-m_{fc}^{\vartheta_{\infty}}(z_{1}))(\vartheta_{\infty} x-z_{2}-m_{fc}^{\vartheta_{\infty}}(z_{2}))}\dd\nu(x).
	\eeq
\end{proposition}

\begin{proposition}\label{prop:Gprocsqrttheta}
	Suppose that $V$ is \textbf{random} and $\vartheta_{\infty}>0$. For a fixed constant $c>0$ and a path $\caK\subset\{\Im z>c\}$, the process $\{\wt{\xi}^{\vartheta}_{N}(z):z\in\caK\}$ converges weakly to a Gaussian process $\{\wt{\xi}^{\vartheta_{\infty}}(z):z\in\caK\}$ with zero mean and the covariance 
	\beq
	\vartheta_{\infty}^{-2}(1+(m_{fc}^{\vartheta_{\infty}})'(z_{1}))(1+(m_{fc}^{\vartheta_{\infty}})'(z_{2}))\big[I(z_{1},z_{2})-m_{fc}^{\vartheta_{\infty}}(z_{1})m_{fc}^{\vartheta_{\infty}}(z_{2})\big],
	\eeq
	where $I(z_{1},z_{2})$ is given above.
\end{proposition}

\begin{proposition}\label{prop:Gprocsqrt}
	Suppose that $V$ is \textbf{random} and $\vartheta_{\infty}=0$. For a fixed constant $c>0$ and a path $\caK\subset\{\Im z>c\}$, the process $\{\wt{\xi}^{\vartheta}_{N}(z):z\in\caK\}$ converges weakly to a Gaussian process $\{\wt{\xi}(z):z\in\caK\}$ with zero mean and the covariance $\Var{v_{1}}m_{sc}'(z_{1})m_{sc}'(z_{2})$.
\end{proposition}

\begin{notation}
	For notational simplicity, we denote $\wt{\xi}^{0}(z)$ by $\wt{\xi}(z)$.
\end{notation}

\begin{lemma}\label{lem:rl0igno}
	For any sufficiently small $\delta>0$ and a sequence of events $\{\Omega_{N}\}_{n\in\N}$ with $\prob{\Omega_{N}} \to 1$,
	\beq\label{eq:sharpN}
	\lim_{v_{0}\to 0^{+}}\limsup_{N\to\infty}\int_{0}^{1}\expct{\absv{\xi_{N}^{\vartheta}(\Gamma_{\#}(t))}^{2}\lone_{\Omega_{N}}}\absv{\Gamma_{\#}'(t)}dt=0,
	\eeq
	and
	\beq\label{eq:sharp}
	\lim_{v_{0}\to 0^{+}}\int_{0}^{1}\expct{\absv{\xi^{\vartheta_{\infty}}(\Gamma_{\#}(t))}^{2}}\absv{\Gamma'_{\#}(t)}dt=0
	\eeq
	where $\xi$ equals either $\xi_{N}^{\vartheta}$ or $\wt{\xi}_{N}^{\vartheta}$ and $\xi^{\infty}$ equals $\xi^{\vartheta_{\infty}}$ or $\wt{\xi}^{\vartheta_{\infty}}$, and $\Gamma_{\#}$ can be $\Gamma_{l},\Gamma_{r}$ or $\Gamma_{0}$.
\end{lemma}

Given these results, we deduce that
\begin{multline}
\biggexpct{\biggabsv{\lone_{\Omega_{N}}\oint_{\Gamma_{\#}}\xi_{N}(z)\varphi(z)dz}^{2}} 
\le \bigg(\sup_{\substack{\Re z=a_{\pm},\\ \absv{\Im z}\leq v_{0}}}\absv{\varphi(z)}^{2}\bigg) \biggexpct{\lone_{\Omega_{N}}\biggabsv{\int_{0}^{1}\absv{\xi_{N}(\Gamma_{\#}(t))}\absv{\Gamma_{\#}'(t)}dt}^{2}} \\
\le C\biggexpct{\lone_{\Omega_{N}}\int_{0}^{1}\absv{\xi_{N}(\Gamma_{\#}(t))}^{2}\absv{\Gamma_{\#}'(t)}^{2}dt} \leq C\int_{0}^{1}\expct{\lone_{\Omega_{N}}\absv{\xi_{N}(\Gamma_{\#}(t))}^{2}}\absv{\Gamma_{\#}'(t)}^{2}dt,
\end{multline}
which, together with the assumption $\prob{\Omega_{N}}\to 1$, imply that $\oint_{\Gamma_{\#}}\xi_{N}(z)\varphi(z)\dd z$ converges to 0 in probability. Similarly, $\oint_{\Gamma_{\#}}\xi^{\vartheta_{\infty}}(z)\varphi(z)\dd z$, $\oint_{\Gamma_{\#}}\wt{\xi}^{\vartheta_{\infty}}(z)\varphi(z)\dd z$ and $\oint_{\Gamma_{\#}}\wt{\xi}_{N}^{\theta}(z)\varphi(z)\dd z$ also converge in probability to $0$ as $v_{0}\to 0^{+}$, for $\Gamma_{\#}=\Gamma_{l},\Gamma_{r},\Gamma_{0}$. Since $T_{N}(\varphi)$ and $S_{N}(\varphi)$ do not depend on $v_{0}$ as long as it is strictly positive, combining Proposition~\ref{prop:Gproc} and Lemma~\ref{lem:rl0igno}, we get the convergence
\beq
T_{N}^{\vartheta}(\varphi)\Rightarrow T(\varphi) \quad\text{and}\quad S_{N}^{\vartheta}(\varphi)\Rightarrow S(\varphi)
\eeq
in distribution, where $T(\varphi)$ and $S(\varphi)$ are the Gaussian random variables defined in Theorems~\ref{thm:main} and \ref{thm:mainsqrt}.

\begin{remark}
	In \cite{Su2013}, where the deformed Gaussian orthogonal ensembles were analyzed in depth, it was proved that for $W=N^{-1/2}A+N^{-\alpha/2}V$ where $\alpha\in(0,1)$, $A$ is a GOE matrix, and $V$ is a random diagonal matrix as in Assumption~\ref{assump:esdvconv}, the mean and variance of $m_{N}(z)=\tr(W-zI)^{-1}$ are given by
	\beq
	\expct{m_{N}(z)}=m_{sc}(z)+\frac{m_{sc}(z)^{3}}{1-m_{sc}(z)^{2}}\cdot\frac{\Var{v_{1}}}{N^{\alpha}}+O(N^{-\frac{\min(3\alpha,2)}{2}}),
	\eeq
	\beq
	\Var{m_{N}(z)}=\biggabsv{\frac{m_{sc}(z)^{2}}{1-m_{sc}(z)^{2}}}^{2}\cdot\frac{\Var{v_{1}}}{N^{1+\alpha}}+O(N^{-\frac{\min(2+3\alpha,3+\alpha)}{2}}),
	\eeq
	which coincide with our results given in Proposition~\ref{prop:Gprocsqrt}.
\end{remark}
\begin{remark}
	For the case where $\vartheta=\sigma N^{-1/2}$ for some constant $\sigma>0$, $W$ itself is another Wigner matrix with $w_{2}$ replaced by $w_{2}+\sigma^{2}\Var{v_{1}}$, and hence it is known (see \cite{Bai-Yao2005}, for instance) that $\{N(m_{N}(z)-m_{sc}(z)):z\in\caK\}$ converges in distribution to the Gaussian process with mean
	\beq
	m_{sc}(z)^{3}(1+m_{sc}'(z))((w_{2}+\sigma^{2}\Var{v_{1}}-2)+m_{sc}'(z)+(W_{4}-3)m_{sc}(z)^{2})
	\eeq
	and covariance
	\beq
	m_{sc}'(z_{1})m_{sc}'(z_{2})\bigg((w_{2}+\sigma^{2}\Var{v_{1}}-2)+2(W_{4}-3)m_{sc}(z_{1})m_{sc}(z_{2})+\frac{2}{(1-m_{sc}(z_{1})m_{sc}(z_{2}))^{2}}\bigg).
	\eeq
	After normalizing by $N^{\frac{1}{2}}\vartheta^{-1}$, the mean is
	\beq\label{eq:baiyaom}
	\frac{1}{\sigma}\expct{m_{N}(z)-m_{sc}(z)} =\sigma m_{sc}(z)^{3}(1+m_{sc}'(z))\Var{v_{1}} +O(\sigma^{-1})
	\eeq
	and the covariance is
	\beq
	m'_{sc}(z_{1})m_{sc}'(z_{2})\Var{v_{1}} +O(\sigma^{-2}).
	\eeq
	The difference between our results stems from the deterministic factor $m_{fc}-m_{sc}$. Considering the self-consistent equation \eqref{eq:funceq}, for $\vartheta=\sigma N^{-1/2}$, $\nu$ being centered implies
	\begin{multline}
	\Lambda \deq \frac{m_{fc}(z)-m_{sc}(z)}{\vartheta^{2}} =\vartheta^{2}m_{sc}(z)\Lambda^{2} +m_{sc}(z)^{2}\Lambda-\vartheta^{-1} m_{sc}(z)\int_{\R}\frac{x}{\vartheta x-z-m_{fc}^{\vartheta}(z)}\dd\nu(x) \\
	=m_{sc}(z)^{2}\Lambda+\vartheta^{-1}m_{sc}(z)\int_{\R}\bigg(\frac{x}{-z-m_{sc}(z)}-\frac{x}{\vartheta x-z-m_{fc}^{\vartheta}(z)}\bigg)\dd\nu(x)+O(\vartheta^{2}\Lambda^{2}) \\
	=m_{sc}(z)^{2}\Lambda+m_{sc}(z)\int_{\R}\frac{x^{2}}{(\vartheta x-z-m_{fc}^{\vartheta}(z))(-z-m_{sc}(z))}\dd\nu(x)+O(\vartheta\Lambda+\vartheta^{2}\Lambda^{2}) \\
	=m_{sc}(z)^{2}\Lambda+m_{sc}(z)^{3}\Var{v_{1}}+O(\vartheta+\vartheta^{2}\Lambda+\vartheta\Lambda+\vartheta^{2}\Lambda^{2}),
	\end{multline}
	where we add an auxiliary factor 
	\beq
	\vartheta^{-1}m_{sc}(z)^{2}\int_{\R}x\dd\nu(x) =\vartheta^{-1}m_{sc}(z)\int_{\R}\frac{x}{-z-m_{sc}(z)}\dd\nu(x) =0
	\eeq
	in the second equality.
	Then by assuming $\Lambda=O(1)$ we get 
	\beq
	\Lambda\to \frac{m_{sc}(z)^{3}}{1-m_{sc}(z)^{2}}\Var{v_{1}} =m_{sc}(z)^{3}(1+m_{sc}'(z))\Var{v_{1}},
	\eeq
	and indeed
	\beq
	\frac{N}{C}(m_{fc}-m_{sc}) =C\frac{m_{sc}(z)^{3}}{1-m_{sc}(z)^{2}}\Var{v_{1}}+O(C^{2}N^{-\frac{1}{2}}).
	\eeq
	It can be easily seen that this term precisely compensates $Cm_{sc}(z)^{3}(1+m_{sc}'(z))\Var{v_{1}}$ in \eqref{eq:baiyaom}, and the same holds for the covariance.
\end{remark}

\subsection{Extension to $C^{2}$ test functions}

In order to use Cauchy integral formula to pass from the Gaussian fluctuation of resolvent to that of linear statistics, we restricted our test function to be analytic in the previous section. In this section, we use the density argument proposed in \cite{Shcherbina2011a} to extend the result for $C^{2}$ functions. The following lemma which enables the density argument was proved in the same paper:
\begin{lemma}[Proposition 3 of \cite{Shcherbina2011a}]\label{lem:Shprop}
	Let $\{\xi_{\ell}^{(n)}\}_{\ell=1}^{n}$ be a triangular array of random variables, $\caN_{n}[\varphi]\deq\sum_{\ell=1}^{n}\varphi(\xi_{\ell}^{(n)})$ corresponding to a test function $\varphi:\R\to\R$, and $V_{n}[\varphi]\deq\Var{\caN_{n}[\varphi]}$ be the variance of $\caN_{n}[\varphi]$. Assume the following:
	\begin{itemize}
		\item There exists a vector space $\caL$ endowed with a norm $\norm{\cdot}$, on which $V_{n}$ is defined and admits the bound
		\beq
		V_{n}[\varphi]\leq C\norm{\varphi}^{2},\quad\forall\varphi\in\caL.
		\eeq
		
		\item There exists a dense linear subspace $\caL_{1}\subset\caL$ such that the CLT is valid for $\caN_{n}[\varphi]$, $\varphi\in\caL_{1}$. That is, we have the weak convergence
		\beq\label{eq:dense}
		\caN_{n}[\varphi]-\expct{\caN_{n}[\varphi]}\to \caN(0,V[\varphi])
		\eeq
		where $V:\caL_{1}\to\R_{+}$ is a continuous quadratic functional.
	\end{itemize}
	Then $V$ admits a continuous extension to $\caL$ and \eqref{eq:dense} holds for $\varphi\in\caL$.
\end{lemma}
The proof is based on the uniform convergence of characteristic functions, and we see that the lemma extends to different normalizations without any modification in the proof. In particular for random $V$, our choice of normalization will be $N^{-1/2}\vartheta$, so that we replace $\caN_{n}[\varphi]$ by $N^{-1/2}\vartheta^{-1}\sum_{\ell}\varphi(\xi_{\ell}^{(n)})$. 

As addressed above, our choice of dense subset will be the analytic functions. The following lemma is a direct application of Weierstrass approximation theorem, which proves the density of analytic functions in $C^{2}$:
\begin{lemma}
	Let $\caL_{1}$ be the space real-valued functions on $\R$ with compact support which is analytic in a neighborhood of $[L_{-},L_{+}]$. Then $\caL_{1}\cap C^{2}_{M}$ is dense in $C^{2}_{M}$.
\end{lemma}

Given the density of analytic functions in $C^{2}_{M}$, we proceed to the first condition of \ref{lem:Shprop}:
\begin{lemma}[Proposition 1 of \cite{Shcherbina2011a}]
	For any $s>0$, any $(N\times N)$ Hermitian or real symmetric random matrix $M$, and for any $\varphi:\R\to\R$, we have
	\beq
	\Var{\sum_{i=1}^{N}\varphi(\lambda_{i})}\leq C_{s}\norm{\varphi}^{2}_{s}\int_{0}^{\infty}\e{-y}y^{2s-1}\int_{-\infty}^{\infty}\Var{\Tr R(x+\ii y)}\dd x\dd y,
	\eeq
	where $R(x+\ii y)$ is the resolvent of $M$ and
	\beq
	\norm{\varphi}_{s}^{2}\deq\int(1+2\absv{k})^{2s}\absv{\wh{\varphi}(k)}^{2}\dd k,\quad \wh{\varphi}(k)\deq\frac{1}{2\pi}\int \e{\ii kx}\dd x.
	\eeq
\end{lemma}
The proposition above enables us to pass from bound on the variance of $R$ into the norm $\norm{\varphi}_{s}$. In particular, our choice of $s$ will be $\frac{3}{2}+\epsilon$, so that $\norm{\varphi}_{s}$ is bounded by $\norm{\varphi}_{C^{2}_{M}}$.

The following lemma is stated and proved in \cite{Shcherbina2011a} for Wigner matrices, and the same proof works also for our $W$. The proof is given in \ref{sec:lemmarl0igno} as its proof is similar to that of Lemma \ref{lem:rl0igno}, for both of lemmas are concerning bounds for $\Var{m_{N}}$.

\begin{lemma}\label{lem:varbound}
	Let $z=E+\ii\eta\in\C_{+}$. Then we have the following bounds of variances for any $\epsilon>0$:
	\beq
	N^{2}\expct{\absv{m_{N}(z)-\cexpct{m_{N}(z)}{V}}^{2}}\leq C\eta^{-3-\epsilon}\frac{1}{N}\sum_{k}\expct{\absv{R_{kk}}^{1+\epsilon}},
	\eeq
	\beq
	\frac{N}{\vartheta^{2}}\expct{\absv{m_{N}(z)-\expct{m_{N}(z)}}^{2}}\leq C\Var{v_{1}}\eta^{-3-\epsilon}\frac{1}{N}\sum_{k}\expct{\absv{R_{kk}}^{1+\epsilon}}.
	\eeq
\end{lemma}

Combining two lemmas above, following the proof of Lemma 2 in \cite{Shcherbina2011a}, we obtain the bound for variance of linear statistics when $\norm{\varphi}_{3/2+\epsilon}$:
\begin{lemma}
	Suppose that $\varphi:\R\to\R$ satisfies $\norm{\varphi}_{3/2+\epsilon}<\infty$ for some $\epsilon>0$.
	\begin{itemize}
		\item If $V$ is \textbf{deterministic}, 
		\beq
		\Var{\sum_{i}\varphi(\lambda_{i})}\leq C\norm{\varphi}^{2}_{3/2+\epsilon}.
		\eeq
		\item If $V$ is \textbf{random},
		\beq
		\Var{\frac{\sqrt{N}}{\vartheta}\sum_{i}\varphi(\lambda_{i})}\leq C\norm{\varphi}^{2}_{3/2+\epsilon}.
		\eeq
	\end{itemize}
\end{lemma}

Combining the results, we get the CLT for centralized random variables, with compactly supported $C^{2}$ test functions, as $\norm{\varphi}_{3/2+\epsilon}\leq C_{M}\norm{\varphi}_{C^{2}_{M}}$ for some constant $C_{M}>0$ whenever $\varphi\in C^{2}_{M}$.

\begin{remark}
	In \cite{Najim-Yao2016}, the authors proved analogous result for non-white sample covariance matrices. In particular, they proved that the centered linear spectral statistics converges to a Gaussian if the test function is $C^{3}$. Also the "bias" defined by
	\beq
		\expct{\sum_{i=1}^{N}\varphi(\lambda_{i})}-\int \varphi(x)\dd(\mu_{MP}\boxtimes \wh{\nu})(x),
	\eeq
	where $\wh{\nu}$ is the spectral distribution of population matrix and $\mu_{MP}$ is the Marchenko-Pastur law, was analyzed. They proved that the bias can be asymptotically represented as an integral concerning the Green function, provided that the test function is $C^{18}$. Their proof uses Helffer-Sj\"{o}strand  formula to pass from the convergence of Green function to that of linear statistics, and the main reason of such restriction is that difference of expected Green function and Stieltjes transform $\mu_{MP}\boxtimes\wh{\nu}$ is of order $\frac{1}{N}\eta^{-17}$ (see Remark 4.4 of \cite{Najim-Yao2016} for detailed discussion).
	
	As seen in the definition, the deterministic quantity defined above corresponds to the expectations of $S_{N}^{\vartheta_{N}}$ and $T_{N}^{\vartheta_{N}}$. If we can obtain the bounds of $(\expct{m_{N}^{\vartheta_{N}}}(z)-\wh{m}_{fc}(z))$ of the form $\frac{1}{N}\eta^{-k}$ for some $k$, then we can also prove convergence of corresponding bias for $C^{k+1}$ test functions by Proposition 6.2 of \cite{Najim-Yao2016}. In fact, recently the authors of \cite{Dallaporta-Fevrier2019} established the bound of the form $\frac{1}{N}\eta^{-13/2}$ and extended the convergence of bias to $\caH_{s}$ functions with $s>13/2$, and hence to $C^{7}_{c}$ functions.
\end{remark}

\subsection{Preliminary results}\label{sec:prelim}
Here we collect some preliminary results concerning the behavior of $m_{fc}^{\vartheta}$ and $\wh{m}_{fc}^{\vartheta}$. Some of the lemmas are not cited, and their proofs are addressed in the \ref{sec:lemmas3.2}.

\subsubsection{Deformed semicircle laws}
\begin{notation}
	For random variables $X\equiv X_{N}$ and $Y\equiv Y_{N}$ depending on $N$, we use the notations $X\prec Y$ and $X=\caO(Y)$ to indicate that for any $\epsilon, D>0$,
	\beq
	\prob{\absv{X}>N^{\epsilon}\absv{Y}}\leq N^{-D}
	\eeq
	for any sufficiently large $N$. Similarly, for a given event $\Omega\equiv\Omega_{N}$, we write $X\prec Y$ \emph{on $\Omega$} to indicate that for any $\epsilon,D>0$, 
	\beq
	\prob{[\absv{X}>N^{\epsilon}\absv{Y}]\cap\Omega}\leq N^{-D}.
	\eeq
	Also we write $X=\caO_{p}(Y)$ when $X$ is bounded by $\absv{Y}N^{\epsilon}$ in probability, for any $\epsilon>0$.
\end{notation}

\begin{notation}
	Let $p\neq q\in\{1,\cdots,N\}$. Then we define $W^{(p)}$ to be the $(N-1)\times(N-1)$ minor of $W$, obtained by removing all columns and rows of $W$ with index $p$. Note that we set the rows and columns of $W^{(p)}$ to be indexed by $\{1,\cdots,N\}\setminus\{p\}$, so that the indices of $W$ remain intact while defining $W^{(p)}$. In a similar fashion, we also define $W^{(p,q)}$ to be the $(N-2)\times (N-2)$ minor of $W$, obtained by removing all columns and rows of $W$ with index $p$ or $q$.
	
	Also, we denote the resolvents of $W^{(p)}$ and $W^{(p,q)}$ by $R^{(p)}$ and $R^{(p,q)}$ respectively, and denote
	\beq
		\frac{1}{N}\Tr R^{(p)}(z)=m_{N}^{(p)}(z).
	\eeq
	Finally, we use the following shorthand notation:
	\beq
		\sum_{i}^{(p)}\deq\sum_{\substack{i=1 \\ i\neq p}}^{N},\quad \sum_{i}^{(p,q)}\deq\sum_{\substack{i=1 \\i\neq p,q}}^{N}.
	\eeq
	The same notation with multiple indices in the sum is similarly defined: e.g. 
	\beq
		\sum_{i,j}^{(p)}\deq\sum_{\substack{i,j=1 \\i,j\neq p}}^{N}
	\eeq
	
\end{notation}

\begin{definition}
	Let $\Omega\equiv\Omega_{N}(\alpha)$ be the event on which the following holds:
	\begin{enumerate}[(i)]
		\item We have
		\beq\label{eq:regesdveq}
		\inf_{x\in I_{\wh{\nu}}}\int\frac{1}{(v-x)^{2}}\dd\wh{\nu}(x)\geq 1+\varpi.
		\eeq
		
		\item For any fixed compact set $\caD\subset\C^{+}$ with $\caD\cap\supp\nu=\emptyset$, there exists a constant $C>0$ such that for any sufficiently large $N$,
		\beq\label{eq:esdconveq}
		\sup_{z\in\caD}\absv{m_{\wh{\nu}}(z)-m_{\nu}(z)}\leq C N^{-\alpha}
		\eeq
		
		\item If $V$ is \textbf{random}, we impose another condition: for any fixed compact set $\caD\subset \Theta_{\varpi}\times\C^{+}$ satisfying
		\beq
		\inf_{(\vartheta,z)\in\caD,x\in I_{\nu}}\absv{\vartheta x-z}>0,
		\eeq
		there exists a constant $C>0$ such that for any sufficiently large $N$,
		\beq\label{eq:esdconveqtheta}
		\sup_{(\vartheta,z)\in\caD,\vartheta\neq0}\vartheta^{-1}\absv{m_{\wh{\nu}}^{\vartheta}(z)-m_{\nu}^{\vartheta}(z)}\leq CN^{-\alpha}
		\eeq
	\end{enumerate}
\end{definition}

\begin{remark}
	We remark that for random $V$, our conditions of $\Omega_{N}$ is stronger than the conditions in Definition 3.1 of \cite{Lee-Schnelli-Stetler-Yau2016}. Therefore the bounds on $\Omega_{N}$ given in \cite{Lee-Schnelli-Stetler-Yau2016} applies without modification.
\end{remark}

The following lemma, which is proved in \ref{sec:lemmas3.2}, controls the probability of the complementary event $\Omega_{N}^{c}$, allowing us to focus on the analysis within $\Omega_{N}$:
\begin{lemma}\label{lem:chernoff}
	For \textbf{random} $V$  and any fixed $\epsilon_{0}>0$, there exists $c>0$ such that
	\beq
	\Bigprob{\Omega_{N}\Big(\frac{1}{2}-\epsilon_{0}\Big)}\geq 1-cN^{-\frt},
	\eeq
	where $\frt$ is given in Assumption~\ref{assump:regesdv}.
\end{lemma}

\begin{remark}\label{rem:omegaprob}
	If $V$ is deterministic, $\Omega_{N}(\alpha_{0})$ holds with probability $1$ for $\alpha_{0}$ given in \eqref{eq:alpha0} and $N$ sufficiently large, by Assumptions~\ref{assump:esdvconv} and \ref{assump:regesdv}. Thus, in Appendices~\ref{sec:meanb}--\ref{sec:lemmarl0igno}, whenever we assume $V$ is deterministic, we let $N$ be large enough so that the assumptions on $\Omega_{N}(\alpha_{0})$ holds with probability $1$.
\end{remark}

\begin{notation}
	We write $\Omega_{N}\equiv\Omega_{N}(\alpha_{0})$ if $V$ is \textbf{deterministic}, and $\Omega_{N}\equiv\Omega_{N}(\frac{1}{2}-\epsilon_{0})$ if $V$ is \textbf{random}, where $\epsilon_{0}$ is small enough but fixed positive real number.
\end{notation}

\begin{lemma}[Lemma 3.2 of \cite{Lee-Schnelli-Stetler-Yau2016}]\label{lem:whrhofc}
	Suppose that $\wh{\nu}$ and $\nu$ satisfy Assumptions~\ref{assump:esdvconv} and \ref{assump:regesdv}. Then for any $N\in\N$ and $\vartheta\in\Theta_{\varpi}$, the inversion formulae
	\beq
	\rho_{fc}^{\vartheta}\deq\lim_{\eta\to 0^{+}}\frac{1}{\pi}\Im m_{fc}^{\vartheta}(E+\ii\eta),\quad E\in\R
	\eeq
	and
	\beq
	\wh{\rho}_{fc}^{\vartheta}\deq\lim_{\eta\to 0^{+}}\frac{1}{\pi}\Im \wh{m}_{fc}^{\vartheta}(E+\ii\eta),\quad E\in\R
	\eeq
	define absolutely continuous measures $\rho_{fc}^{\vartheta}$ and $\wh{\rho}_{fc}^{\vartheta}$. Moreover,  $\rho_{fc}^{\vartheta}$ is supported on a single interval with strictly positive density inside the interval and the same conclusion holds for $\wh{\rho}_{fc}^{\vartheta}$  on $\Omega_{N}$ for any sufficiently large $N$.
\end{lemma}

\begin{notation}
	For simplicity, densities of $\rho_{fc}^{\vartheta}$ and $\wh{\rho}_{fc}^{\vartheta}$ are also denoted by the same symbol.
\end{notation}

\begin{definition}
	We denote the supporting interval of $\wh{\rho}^{\vartheta}_{fc}$($\rho_{fc}^{\vartheta}$, resp.) by $[\wh{L}^{\vartheta}_{-},\wh{L}^{\vartheta}_{+}]$($[L^{\vartheta}_{-},L^{\vartheta}_{+}]$, resp.) and let $E_{0}\geq 1+\max\{\absv{L^{1}_{+}},\absv{L^{1}_{-}}\}$. We also define domains
	\beq
	\caD:=\{z\in E+\ii\eta\in\C^{+}:\absv{E}\leq E_{0}, N^{-\delta} \leq\eta \leq 3\}
	\eeq
	and
	\beq
	\caD':=\{z\in E+\ii\eta\in\C^{+}:\absv{E}\leq E_{0},0<\eta<3\},
	\eeq
	where $0<\delta<1$ is a sufficiently small but fixed constant (independent of $N$), to be determined.
\end{definition}

\begin{lemma}[Theorem 3.3 of \cite{Lee-Schnelli-Stetler-Yau2016}, Strong local deformed semicircle law]\label{lem:locallaw}
	Under Assumptions \ref{assump:Wigner}, \ref{assump:esdvconv} and \ref{assump:regesdv}, the following hold on $\Omega_{N}$:
	
	For any $z\in \caD$ and $\vartheta\in\Theta_{\varpi}$ (both of which possibly vary with $N$),
	\beq
	\absv{m_{N}^{\vartheta}(z)-\wh{m}_{fc}^{\vartheta}(z)}\prec\frac{1}{N\eta}
	\eeq
	and
	\beq
	\absv{R_{ij}^{\vartheta}(z)-\delta_{ij}\wh{g}_{i}^{\vartheta}(z)}\prec\sqrt{\frac{\Im \wh{m}_{fc}^{\vartheta}(z)}{N\eta}}+\frac{1}{N\eta}, \quad\text{for any }1\leq i,j\leq N,
	\eeq
	where we defined 
	\beq
	\wh{g}_{i}^{\vartheta}(z):=\frac{1}{\vartheta v_{i}-z-\wh{m}_{fc}^{\vartheta}(z)}.
	\eeq
\end{lemma}

\begin{definition}
	The eigenvalues of $W$ are denoted by $\lambda_{1}^{\vartheta}\leq\lambda_{2}^{\vartheta}\leq\cdots\leq\lambda_{N}^{\vartheta}$, and we use $\wh{\gamma}_{i}^{\vartheta}$ and $\gamma_{i}^{\vartheta}$ to denote the respective \emph{classical locations} of laws $\wh{\rho}_{fc}^{\vartheta}$ and $\rho_{fc}^{\vartheta}$, i.e.,
	\beq
	\int_{-\infty}^{\wh{\gamma}_{i}^{\vartheta}}\wh{\rho}_{fc}^{\vartheta}(x)\dd x=\frac{i-\frac{1}{2}}{N}\quad\text{and}\quad \int_{-\infty}^{\gamma_{i}^{\vartheta}}\rho_{fc}^{\vartheta}(x)\dd x = \frac{i-\frac{1}{2}}{N}\quad\text{for }1\leq i\leq N.	
	\eeq
\end{definition}

\begin{notation}
	We also define $\wh{\gamma}_{0}\deq \wh{L}_{-}$ and $\gamma_{0}\deq L_{-}$.
\end{notation}

\begin{lemma}[Corollary 3.4 of \cite{Lee-Schnelli-Stetler-Yau2016}, Rigidity estimates]\label{lem:rigid}
	For a deterministic $V$, under Assumptions~\ref{assump:Wigner}, \ref{assump:esdvconv} and \ref{assump:regesdv}, the following hold on $\Omega_{N}$:
	\beq
	\absv{\lambda_{i}^{\vartheta}-\wh{\gamma}_{i}^{\vartheta}}\prec N^{-\frac{2}{3}}\check{\alpha}_{i}^{-\frac{1}{3}}\quad\text{for }1\leq i\leq N \quad\text{and}\quad \sum_{i=1}^{N}\absv{\lambda_{i}^{\vartheta}-\wh{\gamma}_{i}^{\vartheta}}^{2}\prec \frac{1}{N}
	\eeq
	uniformly for $\vartheta\in\Theta_{\varpi}$,  where $\check{\alpha}_{i}:=\min\{i,N-i+1\}$.
\end{lemma}

\begin{lemma}[Theorem 2.22 of \cite{Lee-Schnelli2013}, Rigidity estimates]
	For a random $V$, under Assumptions \ref{assump:Wigner}, \ref{assump:esdvconv} and \ref{assump:regesdv}, the following hold on $\Omega_{N}$:
	\beq\label{eq:rigidncon}
	\absv{\lambda_{i}^{\vartheta}-\gamma_{i}^{\vartheta}}\prec N^{-\frac{2}{3}}\bigg(\check{\alpha}_{i}^{-\frac{1}{3}}+\lone_{\big[\check{\alpha}_{i}\leq N^{\epsilon}(1+\vartheta^{\frac{3}{2}}N^{\frac{1}{4}})\big]}\bigg)+\vartheta^{2}N^{-\frac{1}{3}}\check{\alpha}_{i}^{-\frac{2}{3}}+\vartheta N^{-\frac{1}{2}},
	\eeq
	uniformly for $\vartheta\in\Theta_{\varpi}$, where $\check{\alpha}_{i}:=\min\{i,N-i+1\}$.
\end{lemma}

In order to control $m_{N}(z)-m_{fc}(z)$ for $z\in\Gamma_{0}$ and sample path outside of $\Omega$, we propose a similar bound following from \eqref{eq:rigidncon}:
\begin{corollary}\label{cor:olxiorder}
	For any $z\in\Gamma$ and $\vartheta\in\Theta_{\varpi}$, we have
	\beq
	\absv{m_{N}^{\vartheta}(z)-m_{fc}^{\vartheta}(z)}\prec \frac{\vartheta}{\sqrt{N}}.
	\eeq
\end{corollary}

\begin{lemma}[Lemmas 3.5, 3.6, and A.1 of \cite{Lee-Schnelli-Stetler-Yau2016}, Square-root behavior]\label{lem:sqrtbehav}
	Let $\nu$ and $\wh{\nu}$ satisfy Assumptions \ref{assump:esdvconv} and \ref{assump:regesdv}, $\kappa_{E}=\min\{\absv{E-L_{-}},\absv{E-L_{+}}\}$. Then the following hold for any $\vartheta\in\Theta_{\varpi}$:
	\begin{enumerate}[(i)]
		\item For any $z=E+\ii\eta\in\caD'$,
		\beq
		\Im m_{fc}^{\vartheta}(z)\sim \Bigg\{
		\begin{array}{lc}\label{eq:sqrtim}
			\sqrt{\kappa_{E}+\eta}, & E\in[L_{-},L_{+}], \\
			\dfrac{\eta}{\sqrt{\kappa_{E}+\eta}}, & E\notin[L_{-},L_{+}],
		\end{array}
		\eeq
		where
		\beq
		\kappa_{E} \deq \min\{\absv{E-L_{-}^{\vartheta}},\absv{E-L_{+}^{\vartheta}}\}. 
		\eeq
		\item There exists a constant $C>1$ such that for any $z\in\caD'$ and $x\in I_{\nu}$,
		\beq\label{eq:mfczmaps}
		C^{-1} \leq\absv{\vartheta x-z-m_{fc}^{\vartheta}(z)} \leq C.
		\eeq
		
		\item There exists a constant $C>1$ such that for any $z=E+\ii\eta\in\caD'$,
		\beq\label{eq:1-s1s2}
		C^{-1}\sqrt{\kappa_{E}+\eta}\leq\biggabsv{1-\int_{\R}\frac{1}{(\vartheta x-z-m_{fc}^{\vartheta}(z))^{2}}\dd\nu(x)}\leq C\sqrt{\kappa_{E}+\eta}.
		\eeq
	\end{enumerate}
	The constants in \eqref{eq:sqrtim}, \eqref{eq:mfczmaps}, \eqref{eq:1-s1s2} can be chosen uniformly in $\vartheta\in\Theta_{\varpi}$.
	
	Furthermore, on $\Omega_{N}$, the following hold for sufficiently large $N$:
	\begin{enumerate}[(i)]
		\item There exists constant $c>0$ such that for any $z\in\caD'$,
		\beq
		\absv{\wh{m}_{fc}^{\vartheta}(z)-m_{fc}^{\vartheta}(z)}\leq N^{-\frac{c\alpha}{2}}\quad\text{and}\quad \absv{\wh{L}_{\pm}^{\vartheta}-L_{\pm}^{\vartheta}}\leq N^{-c\alpha}
		\eeq
		for sufficiently large $N$.
		
		\item For any $z=E+\ii\eta\in\caD'$, 
		\beq\label{eq:sqrtimwh}
		\Im \wh{m}_{fc}^{\vartheta}(z)\sim \Bigg\{
		\begin{array}{lc}
			\sqrt{\wh{\kappa}_{E}+\eta}, & E\in[\wh{L}_{-}^{\vartheta},\wh{L}_{+}^{\vartheta}], \\
			\dfrac{\eta}{\sqrt{\wh{\kappa}_{E}+\eta}}, & E\notin[\wh{L}_{-}^{\vartheta},\wh{L}_{+}^{\vartheta}],
		\end{array}
		\eeq
		
		\item There exists constant $C>1$ such that for any $z\in\caD'$ and $x\in I_{\wh{\nu}}$,
		\beq\label{eq:mfczmapswh}
		C^{-1} \leq\absv{\vartheta x-z-\wh{m}_{fc}^{\vartheta}(z)} \leq C.
		\eeq
	\end{enumerate} 
	The constants in \eqref{eq:sqrtimwh} and \eqref{eq:mfczmapswh} can be chosen uniformly in $\vartheta\in\Theta_{\varpi}$ and $N\in\N$ for $N$ sufficiently large.
\end{lemma}

The square-root behavior of $\Im m_{fc}^{\vartheta}$ and $\Im \wh{m}_{fc}^{\vartheta}$ implies the following fact, to be used later on.
\begin{corollary}\label{cor:prepath}
	If $\caD\subset\C$ is a compact subset with a constant $c>0$ satisfying $\inf_{z\in\caD}\sqrt{\kappa_{E}+\eta}>c$ and
	\beq
	\caD\cap\{z=E+\ii\eta\in\C: E\in[L_{-}^{\vartheta},L_{+}^{\vartheta}]\} \subset\{z\in\C:\absv{\eta} >c\},
	\eeq
	then there exists a constant $c'>0$ such that $\Im m_{fc}^{\vartheta}(z)>c'\eta$ and also
	$\Im \wh{m}_{fc}^{\vartheta}(z)>c'\eta$ on $\Omega_{N}$ for sufficiently large $N$.
	The constant $c'$ can be chosen uniformly in $\vartheta\in\Theta_{\varpi}$ and $N\in\N$ for $N$ sufficiently large.
\end{corollary}

Using the rigidity estimates, we can bound $\absv{m_{fc}-\wh{m}_{fc}}$ for $\Im z\ll 1$, in particular, for $z\notin\Gamma_{u}$. Also, the square-root behavior enables us to enlarge the domain of $z$ in the entrywise local law in Lemma~\ref{lem:locallaw}.
\begin{corollary}\label{cor:xiorder}
	For a fixed constant $c>0$, define $\caD_{c}\deq\{z=E+i\eta:\eta\in(c,3), \absv{E}\leq E_{0}\}$. Then on $\Omega_{N}$, for any $z\in\caD_{c}$ and $\vartheta\in\Theta_{\varpi}$,
	\beq\label{eq:stieltjesconcen}
	\absv{m_{N}^{\vartheta}(z)-\wh{m}_{fc}^{\vartheta}(z)}\prec\frac{1}{N}
	\eeq
	and
	\beq\label{eq:resolventconcen}
	\absv{R_{ij}^{\vartheta}(z)-\delta_{ij}\wh{g}_{i}^{\vartheta}(z)}\prec\frac{1}{\sqrt{N}}.
	\eeq
	Moreover, the estimate~\eqref{eq:stieltjesconcen} holds on $\Omega_{N}$ for $z\in\Gamma_{0}\cup\Gamma_{l}\cup\Gamma_{r}$ and \eqref{eq:resolventconcen} holds on $\Omega_{N}$ for $z\in\Gamma_{l}\cup\Gamma_{r}$.
\end{corollary}
In fact, the bound of $\wh{m}_{fc}^{\vartheta}(z)-m_{fc}^{\vartheta}(z)$ can be improved with a stronger assumption on the domain, following the proof of Lemma~3.6 of \cite{Lee-Schnelli-Stetler-Yau2016}:
\begin{lemma}\label{lem:apriori}
	For \textbf{random} V, if $\caD_{0}\subset\caD'$ is a compact subset with $\inf\{\absv{\kappa_{E}+\eta}:z=E+\ii\eta\in\caD_{0}\}\sim 1$, there exists a constant $C>0$ such that
	\beq
	\absv{\wh{m}_{fc}^{\vartheta}(z)-m_{fc}^{\vartheta}(z)}\leq C\vartheta N^{-\frac{1}{2}+\epsilon_{0}}
	\eeq
	for all $z\in\caD_{0}$ on $\Omega_{N}$, and the constant $C$ can be chosen uniformly in $\vartheta\in\Theta_{\varpi}$, $z\in\caD_{0}$, $N\in\N$ for sufficiently large $N$.
\end{lemma}

\begin{corollary}\label{cor:path}
	Let $\nu$ and $\wh{\nu}$ satisfy Assumptions \ref{assump:esdvconv} and \ref{assump:regesdv}, $\kappa_{E}=\min\{\absv{E-L_{-}},\absv{E-L_{+}}\}$ . Define
	\beq
	I^{\vartheta}(z_{1},z_{2}) \deq\int_{\R}\frac{1}{(\vartheta x-z_{1}-m_{fc}^{\vartheta}(z_{1}))(\vartheta x-z_{2}-m_{fc}^{\vartheta}(z_{2}))}\dd\nu(x),
	\eeq
	and
	\beq
	\wh{I}_{k}^{\vartheta}(z_{1},z_{2})\deq \frac{1}{N}\sum_{p>k}\wh{g}_{p}^{\vartheta}(z_{1})\wh{g}_{p}^{\vartheta}(z_{2})
	\eeq
	for $0\leq k \leq N-1$. Then for each fixed compact subset $\caD\subset\C$ with a constant $c>0$ satisfying $\inf_{z\in\caD}\sqrt{\kappa_{E}+\eta}>c$ and
	\beq
	\caD\cap\{z=E+\ii\eta\in\C: E\in[L_{-}^{\vartheta},L_{+}^{\vartheta}]\} \subset\{z\in\C:\absv{\eta} >c\},
	\eeq
	there exists a constant $r\in(0,1)$ such that
	\beq
	\sup\big\{\absv{I^{\vartheta}(z_{1},z_{2})}: z_{1},z_{2}\in\caD_{c}, \vartheta\in\Theta_{\varpi}\big\}<r
	\eeq
	and on $\Omega_{N}$,
	\beq
	\sup\Big\{\absv{\wh{I}^{\vartheta}_{k}(z_{1},z_{2})}: z_{1},z_{2}\in\caD_{c}, 0\leq k\leq N, \vartheta\in\Theta_{\varpi}\Big\}<r,
	\eeq
	for any sufficiently large $N$.
\end{corollary}

Using the bound for $\absv{\wh{m}_{fc}^{\vartheta}(z)-m_{fc}^{\vartheta}(z)}$ in Lemma~\ref{lem:apriori}, we prove another estimate concerning the covariance of $(\vartheta v_{i}-z-\wh{m}_{fc}(z))^{-1}$. To this end, we prove another lemma used along its proof.

\begin{lemma} \label{lem:G_diff}
	Let $G^{\vartheta}(z)=z+m_{fc}^{\vartheta}(z)$ on $\C^{+}$ and $c\in(0,3)$ be a constant.  Then for each compact subset $\caD\subset\C$ satisfying the assumptions of Corollary~\ref{cor:path}, there exists a constant $d>0$ such that for any $z_{1},z_{2}\in\caD$,
	\beq
	\absv{G^{\vartheta}(z_{1})-G^{\vartheta}(z_{2})} \geq d\absv{z_{1}-z_{2}}.
	\eeq
	The constant $d$ can be chosen uniformly in $z_{1},z_{2}\in\caD_{c}$ and $\vartheta\in\Theta_{\varpi}$.
\end{lemma}
\begin{proof}
	The lemma directly follows from \eqref{eq:funceq}:
	\begin{equation}
	\biggabsv{\frac{G^{\vartheta}(z_{1})-G^{\vartheta}(z_{2})}{z_{1}-z_{2}}}=\biggabsv{1-\int_{\R}\frac{1}{(\vartheta x-z_{1}-m_{fc}(z_{1}))(\vartheta x-z_{2}-m_{fc}(z_{2}))}\dd\nu(x)}
	\geq 1-r,
	\end{equation}
	where $r$ is given in Corollary~\ref{cor:path}.
\end{proof}

Now with the help of Lemma~\ref{lem:G_diff}, we state and prove the desired result.
\begin{corollary}\label{cor:covwh}
	Suppose that $V$ is \textbf{random}. Let $c>0$ be given, $\caD$ satisfy the assumptions of Corollary \ref{cor:prepath}, and $I^{\vartheta}(z_{1},z_{2})$ and $\wh{I}_{k}^{\vartheta}(z_{1},z_{2})$ be defined as in Corollary~\ref{cor:path}. For any compact subset $\caD_{1}\subset\caD$, there exists a constant $C>0$ such that
	\beq
	\absv{I^{\vartheta}(z_{1},z_{2})-\wh{I}_{0}^{\vartheta}(z_{1},z_{2})}\leq C\vartheta N^{-\frac{1}{2}+\epsilon_{0}}
	\eeq
	for $z_{1},z_{2}\in\caD_{1}$, on $\Omega_{N}$. The constant $C$ can be chosen uniformly in $z_{1},z_{2}\in\caD_{1}$, $\vartheta\in\Theta_{\varpi}$, and $N\in\N$ for sufficiently large $N$.
\end{corollary}

Since we are assuming that $\vartheta$ varies with $N$, with the limit $\vartheta_{\infty}$, the deterministic function $m_{fc}^{\vartheta}(z)$ converges to $m_{fc}^{\vartheta_{\infty}}(z)$, which is the Stieltjes transform of the semicircle law if $\vartheta_{\infty}=0$, with rate $O(\absv{\vartheta-\vartheta_{\infty}})$. Even though this fact has been addressed by many authors previously, we here propose another proof following a method that will be used frequently throughout this paper: the self-comparison method(see \ref{sec:lemmas3.2}). Namely, we start from the self-consistent equations satisfied by the Stieltjes transform, say $m$, and track the leading order terms of $m$ in the integral equation. The following lemma, which describes the behavior of $m_{sc}$, is used along the proof for the case where $\vartheta_{\infty}=0$.
\begin{lemma}[Lemma 4.2 of \cite{Erdos-Yau-Yin2011}, Square-root behavior of $m_{sc}(z)$]\label{lem:sqrtbehavsc}
	Suppose that $z=E+\ii\eta\in\C^{+}$ with $\absv{E}\leq 5$. Then
	\beq
	\absv{m_{sc}(z)}=\absv{m_{sc}(z)^{-1}+z}^{-1}\leq 1.
	\eeq
	If in addition we have $\eta\leq 10$, then
	\beq
	\absv{m_{sc}(z)}\sim 1,\quad \absv{1-m_{sc}(z)}\sim\sqrt{\kappa+\eta},
	\eeq
	where $\kappa\equiv\kappa_{E}\deq \min\{\absv{E-2},\absv{E+2}\}$.
\end{lemma}

Using the lemma, we now prove the required estimate for $\absv{m_{fc}^{\vartheta}(z)-m_{fc}^{\vartheta_{\infty}}(z)}$ using the self-comparison.
\begin{lemma}\label{lem:mfc-msc}
	For any fixed compact set $\caD\in\C$ with $\dist(\caD, [L_{-}^{\vartheta},L_{+}^{\vartheta}])\sim 1$, there exists a constant $C>0$ such that
	\beq
	\sup_{z\in\caD}\absv{m_{fc}^{\vartheta}(z)-m_{fc}^{\vartheta_{\infty}}(z)}\leq C\absv{\vartheta-\vartheta_{\infty}} 
	\eeq
	for any sufficiently large $N\in\N$.
\end{lemma}

Given the bound of $m_{fc}^{\vartheta}(z)-m_{sc}(z)$, that of the covariance term can be deduced easily.
\begin{corollary}\label{cor:I-s1s2}
	Under the assumptions of Lemma~\ref{lem:mfc-msc}, there exists a constant $C>0$ such that
	\beq
	\sup_{z_{1},z_{2}\in\caD}\absv{I^{\vartheta}(z_{1},z_{2})-I^{\vartheta_{\infty}}(z_{1},z_{2})}\leq C\absv{\vartheta-\vartheta_{\infty}}.
	\eeq
\end{corollary}
\begin{proof}
	Given the Lemma~\ref{lem:mfc-msc}, the lemma is a direct consequence of the self-consistent equation \eqref{eq:funceq} and the stability bound \eqref{eq:mfczmaps}:
	\beq
	I^{\vartheta}(z_{1},z_{2})-I^{\vartheta_{\infty}}(z_{1},z_{2})=\int_{\R} \big[f^{\vartheta}_{z_{1}}(x)f^{\vartheta}_{z_{2}}(x)-f^{\vartheta_{\infty}}_{z_{1}}(x)f^{\vartheta_{\infty}}_{z_{2}}(x)\big] \dd\nu(x)=O(\absv{\vartheta-\vartheta_{\infty}}),
	\eeq
	where we abbreviated $f^{\vartheta}_{z}(x)\deq\frac{1}{\vartheta x-z-m_{fc}^{\vartheta}(z)}$. Note that the uniformity follows from \eqref{eq:mfczmaps} and Lemma~\ref{lem:mfc-msc}.
\end{proof}
\begin{remark}
	Note that for $\vartheta_{\infty}=0$, we have
	\beq
	I^{\vartheta_{\infty}}(z_{1},z_{2})=\frac{1}{(-z-m_{sc}(z_{1}))(-z_{2}-m_{sc}(z_{2}))}=m_{sc}(z_{1})m_{sc}(z_{2}).
	\eeq
\end{remark}

Recalling the definition of $W$, we observe that the off-diagonal terms are identical to that of Lemma 5.3 in \cite{Baik-Lee2017} with $J=0$, so that we have the following lemma holds:
\begin{lemma}[Lemma 5.3 of \cite{Baik-Lee2017}, Large deviation estimates]\label{lem:lde}
	Let $S$ be an $(N-1)\times (N-1)$ matrix independent of $\{W_{ia}:1\leq a\leq N,a\neq i\}$ with operator norm $\norm{S}$. Then for $n=1,2$, there exists a constant $C_{n}$ depending only on $W_{4}$ in Assumption~\ref{assump:Wigner} such that
	\beq
	\biggexpct{\biggabsv{\sum_{p,q}^{(i)}W_{ip}S_{pq}W_{qi}-\frac{1}{N}\sum_{p}^{(i)}S_{pp}}^{2n}}\leq \frac{C_{n}}{N^{n+1}}\Tr\absv{S}^{2n}\leq C_{n}\frac{\norm{S}^{2n}}{N^{n}}.
	\eeq 
	
	Moreover,
	\beq
	\biggabsv{\sum_{p,q}^{(i)}W_{ip}S_{pq}W_{qi}-\frac{1}{N}\sum_{p}^{(i)}S_{pp}}\prec\frac{\norm{S}}{\sqrt{N}}.
	\eeq
\end{lemma}

\subsubsection{Matrix identities}

\begin{lemma}[Lemma 3.1 of \cite{Lee-Schnelli2015}, Matrix identities]
	Let $X$ be an $N\times N$, symmetric matrix and $R(z)\deq (X-zI)^{-1}$, $z\in\C$. Then for $i,j,k\in\{1,\cdots,N\}$, the following identities hold:
	\begin{itemize} \label{eq:mat id 3}
		\item[--] Schur complement formula:
		\beq\label{eq:Schur}
		R_{ii}=\bigg(X_{ii}-z-\sum_{m,n}^{(i)}X_{im}R^{(i)}_{mn}X_{ni}\bigg)^{-1}.
		\eeq
		
		\item[--] For $i\neq j$,
		\beq
		R_{ij}=-R_{ii}\sum_{m}^{(i)}X_{im}R_{mj}^{(i)}=-R_{ii}R_{jj}^{(i)}\Big(X_{ij}-\sum_{m,n}^{(i,j)}X_{im}R^{(i,j)}_{mn}X_{nj}\Big).
		\eeq
		
		\item[--] For $i,j\neq k$,
		\beq
		R_{ij}=R_{ij}^{(k)}+\frac{R_{ik}R_{kj}}{R_{kk}}.
		\eeq
	\end{itemize}
\end{lemma}

\section{Proof of Proposition \ref{prop:Gproc}}\label{sec:Gaussianprocprf}

In this section we assume that $V$ is \textbf{deterministic}, so that conditions \eqref{eq:regesdveq} and \eqref{eq:esdconveq} hold true.

As both of propositions asserts convergence of processes, according to Theorem 8.1 of \cite{Billingsley1968} it suffices to prove the finite-dimensional convergence and the tightness of processes. For the finite-dimensional convergence, following \cite{Bai-Yao2005} and \cite{Baik-Lee2017}, we express $\xi_{N}^{\vartheta}(z)$ as a martingale and use the following theorem concerning the central limit theorem of martingales:
\begin{lemma}[Theorem 35.12 of \cite{Billingsley1995}]\label{lem:martclt}
	Suppose that for each fixed $n\in\N$, $\{X_{n,k}\}_{k\in\N}$ is a martingale with respect to a filtration $\caF_{n,1}\subset\caF_{n,2}\subset\cdots$. Let $Y_{n,k}\deq X_{n,k}-X_{n,k-1}$ where $X_{n,0}=0$, and suppose that $\sum_{k}Y_{n,k}$ converges a.s. and $Y_{n,k}\in L^{2}$ for each $k$. Denote $\sigma^{2}_{n,k}\deq \cexpct{Y_{n,k}^{2}}{\caF_{n,k-1}}$ where $\caF_{n,0}$ is the trivial $\sigma-$algebra. If
	\beq\label{eq:martclt1}
	\sum_{k}\sigma_{n,k}^{2}\longrightarrow\sigma^{2}
	\eeq
	in probability where $\sigma\in\R_{+}$ is a constant and
	\beq\label{eq:martclt2}
	\sum_{k}\expct{Y_{n,k}^{2}\lone_{[\absv{Y_{n,k}}\geq\epsilon]}}\to 0
	\eeq
	for each $\epsilon>0$, then $\sum_{k=1}^{\infty}Y_{n,k}$ converges weakly to the normal distribution with zero mean and variance $\sigma^{2}$.
\end{lemma}

Lemma \ref{lem:martclt} is used to conclude that the finite dimensional distribution of $\xi_{N}^{\vartheta}(z)-\expct{\xi_{N}^{\vartheta}(z)}$ converges weakly to the centered Gaussian distribution with designated covariance, and the convergence of mean is dealt separately. In particular, the filtration to which our martingale is adapted is defined as follows:
\begin{definition}\label{def:filt}
	We define the (decreasing) filtration $\{\caF_{k}:0\leq k\leq N\}$ as
	\beq
	\caF_{k}:=\sigma(W_{ij}:k<i,j\leq N),\quad k=0,1,\cdots,N
	\eeq
	and denote the conditional expectation $\cexpct{\cdot}{\caF_{k}}$ by $\expctk{\cdot}{k}$. We also define \beq
	\caG_{k}\deq\sigma(W_{ij}:k<i,j\leq N)\vee\sigma(v_{i}:i>k),\quad k=0,1,\cdots,N.
	\eeq
\end{definition}
\begin{notation}
	For \textbf{random} $V$, we also define
	\beq
	\caF_{k}\deq\sigma(W_{i,j}:k<i,j\leq N) \vee \sigma(v_{m}:1\leq m\leq N),\quad k=0,1,\cdots,N
	\eeq
	and $\E_{k}$ similarly. Note that for \textbf{deterministic} $V$, $\sigma(v_{m}:1\leq m\leq N)$ is the trivial $\sigma$-algebra, so that the definition is still consistent.
\end{notation}

The convergence of mean $\expct{\xi_{N}^{\vartheta}(z)}$ is contained in the following lemma, which is proved in \ref{sec:meanb}.
\begin{lemma}
	Define
	\beq
	b_{N}^{\vartheta}(z)\deq \expct{\xi_{N}^{\vartheta}(z)}=\expctk{\xi_{N}^{\vartheta}(z)}{N}=N\expct{m_{N}(z)-\wh{m}_{fc}^{\vartheta}(z)}.
	\eeq
	For $z\in\caD_{c}$ or $z\in\Gamma_{r}\cup\Gamma_{l}$, 
	\begin{multline}
	b_{N}^{\vartheta}(z)=-\frac{1}{2}\frac{(m_{fc}^{\vartheta_{\infty}})''(z)}{(1+(m_{fc}^{\vartheta_{\infty}})'(z))^{2}}\bigg[(w_{2}-1)+(m_{fc}^{\vartheta_{\infty}})'(z)+(W_{4}-3)\frac{(m_{fc}^{\vartheta_{\infty}})'(z)}{1+(m_{fc}^{\vartheta_{\infty}})'(z)}\bigg] \\
	+O(\vartheta N^{-\alpha_{0}}+N^{-\frac{1}{2}+\epsilon})
	\end{multline}
	if $\vartheta_{\infty}>0$, and
	\beq
	b_{N}^{\vartheta}(z) =m_{sc}(z)^{3}(1+m_{sc}'(z))((w_{2}-1)+m'_{sc}(z)+(W_{4}-3)m_{sc}(z)^{2})+O(\vartheta+N^{-\frac{1}{2}+\epsilon})
	\eeq
	if $\vartheta_{\infty}=0$.
\end{lemma}
\begin{remark}\label{rem:preduc1}
	For \textbf{random} $V$, the same proof with $\expct{\cdot}$ replaced by $\expctk{\cdot}{N}$ gives us the absolute bound
	\beq
	\absv{\expctk{m_{N}(z)}{N}-\wh{m}_{fc}(z)}\lone_{\Omega_{N}}=O(1/N).
	\eeq 
\end{remark}

Now given the convergence of means, to use  {Lemma \ref{lem:martclt}} as addressed above, we express $\xi_{N}^{\vartheta}(z)-\expct{\xi_{N}^{\vartheta}}(z)$ as a martingale. Letting
\beq
	\zeta_{N}\deq\xi_{N}-\expct{\xi_{N}}=\Tr R-\expct{\Tr R},
\eeq
one can rewrite $\zeta_{N}$ as sum of a martingale difference sequence as follows:
\begin{equation}
	\zeta_{N} =\sum_{k=1}^{N}\expctk{\Tr R}{k-1}-\expctk{\Tr R}{k} \\
=\sum_{k=1}^{N}(\E_{k-1}-\E_{k})\Tr R =\sum_{k=1}^{N}(\E_{k-1}-\E_{k})(\Tr R-\Tr R^{(k)}).
\end{equation}
In \ref{sec:covariance}, we further simplify the martingale decomposition to get
\beq
	\zeta_{N}=\sum_{k=1}^{N}\expctk{\phi_{k}^{\vartheta}}{k-1}+\caO_{p}(N^{-\frac{1}{2}})
\eeq
where
\begin{multline}
\phi_{k}^{\vartheta}=\wh{g}_{k}^{\vartheta}\Big(\sum_{p,q}^{(k)}W_{kp}(R^{(k)})^{2}_{pq}W_{qk}-(\wh{m}_{fc}^{\vartheta})'\Big) \\
+(\wh{g}_{k}^{\vartheta})^{2}\Big(-W_{kk}+\vartheta v_{k}+\sum_{p,q}^{(k)}W_{kp}R^{(k)}_{pq}W_{qk}-\wh{m}_{fc}^{\vartheta}\Big)(1+(\wh{m}_{fc}^{\vartheta})').
\end{multline}

In \ref{sec:covariance}, we also prove the conditions \eqref{eq:martclt1} and \eqref{eq:martclt2} of  {Lemma \ref{lem:martclt}}:
\begin{lemma}\label{lem:martclt11}
	For distinct points $z_{1},z_{2}\in\caK$, we let
	\beq
	\Gamma_{N}^{\vartheta}(z_{1},z_{2})=\sum_{k=1}^{N}\expctk{\expctk{\phi_{k}^{\vartheta}(z_{1})}{k-1}\cdot\expctk{\phi_{k}^{\vartheta}(z_{2})}{k-1}}{k}.
	\eeq
	Then
	\begin{multline}
	\Gamma^{\vartheta}_{N}(z_{1},z_{2}) =(w_{2}-2)\frac{\partial^{2}I}{\partial z_{1} \partial z_{2}} +(W_{4}-3)\bigg(I\frac{\partial^{2}I}{\partial z_{1}\partial z_{2}} +\frac{\partial I}{\partial z_{1}}\frac{\partial I}{\partial z_{2}}\bigg) \\
	+\frac{2}{(1-I)^{2}}\bigg(\frac{\partial I}{\partial z_{2}}\frac{\partial I}{\partial z_{1}}+(1-I)\frac{\partial^{2} I}{\partial z_{1}\partial z_{2}}\bigg)+\caO(N^{-\frac{1}{2}})+O(\vartheta N^{-\alpha_{0}})
	\end{multline}
	if $\vartheta_{\infty}>0$ where
	\beq
	I(z_{1},z_{2})\equiv I^{\vartheta_{\infty}}(z_{1},z_{2})\deq\int_{\R}\frac{1}{(\vartheta_{\infty}x-z_{1}-m_{fc}^{\vartheta_{\infty}}(z_{1}))(\vartheta_{\infty}x-z_{2}-m_{fc}^{\vartheta_{\infty}}(z_{2}))}\dd\nu(x),
	\eeq
	and
	\begin{multline}
	\Gamma_{N}(z_{1},z_{2})=m_{sc}'(z_{1})m_{sc}'(z_{2})\bigg((w_{2}-2)+2(W_{4}-3)m_{sc}(z_{1})m_{sc}(z_{2}) \\
	+\frac{2}{(1-m_{sc}(z_{1})m_{sc}(z_{2}))^{2}}\bigg)+O(\vartheta)+\caO(N^{-\frac{1}{2}})
	\end{multline}
	if $\vartheta_{\infty}=0$.
\end{lemma}

\begin{lemma}\label{lem:martclt22}
	For any $z\in\caK$ and $\epsilon>0$,
	\beq
	\sum_{k}\expct{\absv{\expctk{\phi_{k}^{\vartheta}}{k-1}}^{2}\lone_{[\absv{\expctk{\phi_{k}^{\vartheta}}{k-1}}\geq\epsilon]}}\to 0.
	\eeq
\end{lemma}

\begin{remark}\label{rem:preduc2}
	As in  {Remark \ref{rem:preduc1}}, for \textbf{random} $V$, the proofs of  {Lemma \ref{lem:martclt11}} and  {Lemma \ref{lem:martclt22}} imply the fact that
	\beq
	\expct{\absv{m_{N}(z)-\expctk{m_{N}(z)}{N}}^{2}\lone_{\Omega_{N}}}=O(N^{-2+\epsilon}).
	\eeq
\end{remark}

Now given the finite-dimensional convergence, it remains to prove the tightness. Since the mean $b_{N}(z)$ converges, the tightness of $\zeta_{N}^{\vartheta}(z)$ implies that of $\xi_{N}^{\vartheta}(z)$. Following \cite{Bai-Yao2005}, by Theorem 12.3 of \cite{Billingsley1968}, it suffices to check the tightness for a fixed $z\in\caK$ and prove a H\"{o}lder condition given below. The tightness for fixed $z\in\caK$ follows directly from the finite-dimensional convergence and hence the tightness reduces to the following H\"{o}lder condition:
\beq
\expct{\absv{\zeta_{N}^{\vartheta}(z_{1})-\zeta_{N}^{\vartheta}(z_{2})}^{2}}\leq K\absv{z_{1}-z_{2}}^{2}, \quad^{\forall}z_{1},z_{2}\in\caK,
\eeq
for some constant $K$ independent of $N\in\N$ and $z_{1},z_{2}\in\caD_{c}$.

The proof starts with an application of the resolvent equation $R(z_{1})-R(z_{2})=(z_{1}-z_{2})R(z_{1})R(z_{2})$, to get
\begin{multline}
\expct{\absv{\zeta_{N}^{\vartheta}(z_{1})-\zeta_{N}^{\vartheta}(z_{2})}^{2}}=\expct{\absv{(\Tr R(z_{1})-\expct{\Tr R(z_{2})})-(\Tr R(z_{2})-\expct{\Tr R(z_{2})})}^{2}} \\
=\absv{z_{1}-z_{2}}^{2}\expct{\absv{\Tr R(z_{1})R(z_{2})-\expct{\Tr R(z_{1})R(z_{2})}}^{2}}.
\end{multline}

Therefore the following lemma completes the proof of  {Proposition \ref{prop:Gproc}}, which is proved in \ref{sec:tight}:
\begin{lemma}\label{lem:tightproc}
	For $z_{1},z_{2}\in\caK$ and sufficiently large $N\in\N$, we have
	\beq
	\expct{\absv{\Tr R(z_{1})R(z_{2})-\expct{\Tr R(z_{1})R(z_{2})}}^{2}}\leq K,
	\eeq
	where $K$ is a constant independent of $z_{1},z_{2}$ and $N\in\N$.
\end{lemma}

\section{Proof of Propositions \ref{prop:Gprocsqrttheta} and \ref{prop:Gprocsqrt}}\label{sec:Gaussianprocsqrtprf}
In this section we assume that $V$ is \textbf{random}. Since $\prob{\Omega_{N}}\to 1$, assuming $\Omega_{N}$ for any sample paths below will do no harm to our proof.

\subsection{Primary reduction}
Since $\{\wt{\xi}^{\vartheta}_{N}(z):z\in\caK\}$ defines a continuous process for each $n\in\N$, we again use  {Theorem 8.1} of \cite{Billingsley1968} to prove  {Proposition \ref{prop:Gprocsqrt}}, hence it suffices to prove the finite-dimensional convergence and the tightness.
To this end, we start the proof by reducing $\wt{\xi}_{N}(z)$ into a simplified form which looks pleasant to apply the classical central limit theorem.

In this section, we assume $z\in\caK$.
Note first that 
\beq
\wt{\xi}_{N}^{\vartheta}(z) =\frac{\sqrt{N}}{\vartheta}(m_{N}^{\vartheta}(z)-m_{fc}^{\vartheta}(z)) =\frac{1}{\sqrt{N}\vartheta}\xi_{N}^{\vartheta}(z)  +\frac{\sqrt{N}}{\vartheta}
(\wh{m}_{fc}^{\vartheta}(z)-m_{fc}^{\vartheta}(z)).
\eeq
Considering the first term, we decompose it as
\beq
\frac{1}{\sqrt{N}\vartheta}\xi_{N}^{\vartheta}(z) =\frac{\sqrt{N}}{\vartheta}(m_{N}(z)-\expctk{m_{N}(z)}{N})+\frac{\sqrt{N}}{\vartheta}(\expctk{m_{N}(z)}{N}-\wh{m}_{fc}(z)).
\eeq

Then by  {Remarks \ref{rem:preduc1}} and  {\ref{rem:preduc2}}, 
\beq
\biggexpct{\biggabsv{\frac{1}{\sqrt{N}\vartheta}\xi_{N}^{\vartheta}(z)}^{2}\lone_{\Omega_{N}}}=O(N^{-1}\vartheta^{-2}),
\eeq
which, together with the fact that $\prob{\Omega_{N}}\to 1$, implies the in probability convergence $N^{-1/2}\vartheta^{-1}\xi_{N}^{\vartheta}(z)\to0$.
In this sense, we let $\wt{\zeta}^{\vartheta}_{N}(z)\deq\vartheta^{-1}\sqrt{N}(\wh{m}_{fc}^{\vartheta}(z)-m_{fc}^{\vartheta}(z))$ and try to estimate it.

Given the estimate above, we first rewrite $\wt{\zeta}^{\vartheta}_{N}(z)$ as follows:
\begin{multline}\label{eq:above}
	\wt{\zeta}^{\vartheta}_{N}(z) =\frac{\sqrt{N}}{\vartheta}\bigg(\frac{1}{N}\sum_{i=1}^{N}\wh{g}^{\vartheta}_{i}(z)-m_{fc}^{\vartheta}(z)\bigg)
 	=\frac{1}{\sqrt{N}\vartheta} \sum_{i=1}^{N}\bigg[\frac{1}{\vartheta v_{i}-z-\wh{m}_{fc}^{\vartheta}(z)}- \int\frac{1}{\vartheta x-z-m_{fc}^{\vartheta}(z)}\dd\nu(x)\bigg] \\
	=\frac{1}{\sqrt{N}\vartheta}\sum_{i=1}^{N}\bigg[\frac{1}{\vartheta v_{i}-z-m_{fc}^{\vartheta}(z)} -\expct{\frac{1}{\vartheta v_{i}-z-m_{fc}^{\vartheta}(z)}}\bigg] \\
	+\wt{\zeta}^{\vartheta}_{N}(z)\frac{1}{N}\sum_{i=1}^{N}\frac{1}{(\vartheta v_{i}-z-\wh{m}_{fc}^{\vartheta}(z))(\vartheta v_{i}-z-m_{fc}^{\vartheta}(z))}.
\end{multline}

To estimate the second summand in the right-hand side of \eqref{eq:above}, we expand it in terms of $\wh{m}_{fc}(z)-m_{fc}(z)$:
\begin{multline}
\frac{1}{N}\sum_{i=1}^{N}\frac{1}{(\vartheta v_{i}-z-\wh{m}_{fc}^{\vartheta}(z))(\vartheta v_{i}-z-m_{fc}^{\vartheta}(z))} \\
=\int\bigg[\frac{1}{(\vartheta x-z-m_{fc}^{\vartheta}(z))^{2}}+\frac{\wh{m}_{fc}^{\vartheta}(z)-m_{fc}^{\vartheta}(z)}{(\vartheta x-z-m_{fc}^{\vartheta}(z))^{2}(\vartheta x-z-\wh{m}_{fc}^{\vartheta}(z))}\bigg]\dd\wh{\nu}(x),
\end{multline}
and we extract the leading term of above as
\begin{equation}
	\int\frac{1}{(\vartheta x-z-m_{fc}^{\vartheta}(z))^{2}}\dd\wh{\nu}(x) +\frac{\vartheta\wt{\zeta}_{N}^{\vartheta}(z)}{\sqrt{N}}\int\frac{1}{(\vartheta x-z-m_{fc}^{\vartheta}(z))^{2}(\vartheta x-z-\wh{m}_{fc}^{\vartheta}(z))}\dd\wh{\nu}(x).
\end{equation}
Substituting, we get
\begin{multline}\label{eq:barxiexpand}
	\bigg(1-\int\frac{1}{(\vartheta x-z-m_{fc}^{\vartheta}(z))^{2}}\dd\nu(x)\bigg)\wt{\zeta}_{N}(z) \\
	=\frac{1}{\sqrt{N}\vartheta}\sum_{i=1}^{N}\bigg[\frac{1}{\vartheta v_{i}-z-m_{fc}^{\vartheta}(z)}-\expct{\frac{1}{\vartheta v_{i}-z-m_{fc}^{\vartheta}(z)}}\bigg] 
	+\wt{\zeta}^{\vartheta}_{N}(z)\bigg(\int \frac{1}{(\vartheta x-z-m_{fc}^{\vartheta}(z))^{2}}(\dd\wh{\nu}^{\vartheta}(x)-\dd\nu^{\vartheta}(x))\bigg) \\
	+\frac{\vartheta\wt{\zeta}_{N}(z)^{2}}{\sqrt{N}}\int\frac{1}{(\vartheta x-z-m_{fc}^{\vartheta}(z))^{2}(\vartheta x-z-\wh{m}_{fc}^{\vartheta}(z))}\dd\wh{\nu}(x),
\end{multline}
so that
\begin{multline}
	\bigg(1-\int\frac{1}{(\vartheta x-z-m_{fc}^{\vartheta}(z))^{2}}\dd\nu(x)\bigg)\wt{\zeta}_{N}(z) \\
	=\frac{1}{\sqrt{N}\vartheta}\sum_{i=1}^{N}\bigg[\frac{1}{\vartheta v_{i}-z-m_{fc}^{\vartheta}(z)}-\expct{\frac{1}{\vartheta v_{i}-z-m_{fc}^{\vartheta}(z)}}\bigg] 
	+\wt{\zeta}^{\vartheta}_{N}(z)\big((m_{\wh{\nu}}^{\vartheta})'(z+m_{fc}^{\vartheta}(z))-(m_{\nu}^{\vartheta})'(z+m_{fc}(z))\big) \\
	+\frac{\vartheta\wt{\zeta}_{N}(z)^{2}}{\sqrt{N}}\int\frac{1}{(\vartheta x-z-m_{fc}^{\vartheta}(z))^{2}(\vartheta x-z-\wh{m}_{fc}^{\vartheta}(z))}\dd\wh{\nu}(x)
\end{multline}

First, we recall the existence of a constant $C>1$ such that
\beq
\absv{1-\int\frac{1}{(\vartheta x-z-m_{fc}^{\vartheta}(z))^{2}}\dd\nu(x)}\geq C^{-1}\sqrt{\kappa+\eta} \quad\text{ for }z=E+\ii\eta \in\caD'
\eeq
uniformly in $\vartheta\in\Theta_{\varpi}$, given in \eqref{eq:1-s1s2}.
By a standard continuity argument, the bound can be extended to $z=a_{\pm}$ without changing the constant, so that we may divide the equality \eqref{eq:barxiexpand} by the quantity above. 

Then \eqref{eq:esdconveqtheta} together with  {Lemma \ref{lem:apriori}} and Cauchy integral formula implies
\beq
\absv{\wt{\zeta}_{N}^{\vartheta}(z)\Big(\frac{\dd}{\dd z} m_{\wh{\nu}}^{\vartheta}(z+m_{fc}^{\vartheta}(z)) -\frac{\dd}{\dd z} m_{\nu}^{\vartheta}(z+m_{fc}^{\vartheta}(z))\Big)} =O(\vartheta N^{-\frac{1}{2}+2\epsilon_{0}+\epsilon_{1}})
\eeq
on $\Omega$.

Again recalling \eqref{eq:mfczmaps} and other bounds following it, we have
\beq
\Bigabsv{\int\frac{1}{(\vartheta x-z-m_{fc}^{\vartheta}(z))^{2}(\vartheta x-z-\wh{m}_{fc}^{\vartheta}(z))}\dd\wh{\nu}(x)}\leq C
\eeq
on $\Omega$, giving
\beq
\frac{\wt{\zeta}_{N}^{\vartheta}(z)^{2}}{\sqrt{N}}\int\frac{1}{(\vartheta x-z-m_{fc}^{\vartheta}(z))^{2}(\vartheta x-z-\wh{m}_{fc}^{\vartheta}(z))}\dd\wh{\nu}(x) =O(N^{-\frac{1}{2}+2\epsilon_{0}+2\epsilon_{1}})
\eeq
on $\Omega$.

Finally, recalling that
\beq
1-\int \frac{1}{(\vartheta x-z-m_{fc}^{\vartheta}(z))^{2}}\dd\nu(x)=\frac{1}{1+(m_{fc}^{\vartheta})'(z)},
\eeq
we can conclude that on $\Omega$,
\begin{equation}\label{eq:reduc}
	\wt{\zeta}_{N}^{\vartheta}(z) =\big(1+(m_{fc}^{\vartheta})'(z)\big)\frac{1}{\sqrt{N}\vartheta}\sum_{i=1}^{N}\bigg[\frac{1}{\vartheta v_{i}-z-m_{fc}^{\vartheta}(z)}-\expct{\frac{1}{\vartheta v_{i}-z-m_{fc}^{\vartheta}(z)}}\bigg]
	+O(\vartheta N^{-\frac{1}{2}+2(\epsilon_{0}+\epsilon_{1})}).
\end{equation}
\begin{remark}\label{rem:wtximean}
	Noting that the expression of $\wt{\zeta}_{N}^{\vartheta}$ holds also for $z\in\Gamma_{r}\cup\Gamma_{l}$, we remark that 
	\beq
	\expct{\wt{\xi}_{N}^{\vartheta}(z)\lone_{\Omega}}\to0 \quad\text{for  } z\in\Gamma\setminus\Gamma_{0},
	\eeq
	since we have  {Remark \ref{rem:preduc1}}, and from Cauchy-Schwarz inequality,  
	\begin{multline}
	\absv{\expct{\wt{\zeta}^{\vartheta}_{N}(z)\cdot\lone_{\Omega}}}
	=\frac{\absv{1+(m_{fc}^{\vartheta})'(z)}}{\sqrt{N}\vartheta}\biggabsv{\biggexpct{\bigg(\sum_{i}\frac{1}{\vartheta v_{i}-z-m_{fc}^{\vartheta}(z)}-\biggexpct{\frac{1}{\vartheta v_{i}-z-m_{fc}^{\vartheta}(z)}}\bigg)\lone_{\Omega^{c}}}} 
	+O(\vartheta N^{-\frac{1}{2}+2(\epsilon_{0}+\epsilon_{1})}) \\
	\leq \bigg(\prob{\Omega^{c}}\frac{1}{N\vartheta^{2}}\sum_{i}\biggexpct{\biggabsv{\frac{1}{\vartheta v_{i}-z-m_{fc}^{\vartheta}(z)}-\biggexpct{\frac{1}{\vartheta v_{i}-z-m_{fc}}}}^{2}}\bigg)^{\frac{1}{2}}+o(1)\to0.
	\end{multline}
\end{remark}

\subsection{Finite-dimensional convergence}

To serve our purpose of proving the finite-dimensional convergence, we assume in this section that we have a fixed number of points $z_{1},\cdots,z_{p}$ in $\caK$.

\subsubsection{The case $\vartheta_{\infty}>0$}

As easily seen in \eqref{eq:reduc}, the term
\beq\label{eq:clt}
\frac{1}{\sqrt{N}\vartheta}\sum_{i}\bigg[\frac{1}{\vartheta v_{i}-z-m_{fc}^{\vartheta}(z)} -\biggexpct{\frac{1}{\vartheta v_{i}-z-m_{fc}^{\vartheta}(z)}}\bigg]
\eeq 
results in the Gaussian convergence. Indeed, if we consider the sum corresponding to \eqref{eq:clt} with $\vartheta$ replaced by $\vartheta_{\infty}$, from  {Lemma \ref{lem:mfc-msc} } we have
\begin{equation}
	\biggabsv{\frac{1}{\vartheta x-z-m_{fc}^{\vartheta}(z)}-\frac{1}{\vartheta_{\infty}x-z-m_{fc}^{\vartheta_{\infty}}(z)}}
	\leq\biggabsv{\frac{(\vartheta_{\infty}-\vartheta)x-(m_{fc}^{\vartheta}(z)-m_{fc}^{\vartheta_{\infty}}(z))}{(\vartheta x-z-m_{fc}^{\vartheta}(z))(\vartheta_{\infty}x-z-m_{fc}^{\vartheta_{\infty}})}} \leq C\absv{\vartheta-\vartheta_{\infty}},
\end{equation}
where the constant $C$ is chosen uniformly in $x\in\supp\nu$ and $N\in\N$ sufficiently large.
Then the central limit theorem together with the fact that $\prob{\Omega_{N}\cap\Omega_{N}'}\to1$ gives the weak convergence of the random vector
\beq
(\wt{\zeta}_{N}(z_{1}),\cdots,\wt{\zeta}_{N}(z_{p}))
\eeq
to a Gaussian random vector with mean zero and covariance matrix
\beq
\wt{\Gamma}^{\vartheta}(z_{i},z_{j}) =\vartheta_{\infty}^{-2}\bigg(1+\frac{\dd}{\dd z}m_{fc}^{\vartheta_{\infty}}(z_{i})\bigg)\bigg(1+\frac{\dd}{\dd z}m_{fc}^{\vartheta_{\infty}}(z_{j})\bigg)\bigg(I^{\vartheta_{\infty}}(z_{i},z_{j})-m_{fc}^{\vartheta_{\infty}}(z_{i})m_{fc}^{\vartheta_{\infty}}(z_{j})\bigg)
\eeq
for $1\leq i,j\leq p$.
Recalling that $\wt{\xi}_{N}^{\vartheta}(z)-\wt{\zeta}_{N}^{\vartheta}=\caO_{p}(N^{-1/2}\vartheta^{-1})$, we obtain the same convergence for $(\wt{\xi}_{N}^{\vartheta}(z_{1}),\cdots,\wt{\xi}_{N}^{\vartheta}(z_{p}))$.

\subsubsection{The case $\vartheta_{\infty}=0$}
For $\vartheta_{\infty}=0$, we reduce each summand to as follows:
\beq
\frac{1}{\vartheta x-z-m_{fc}^{\vartheta}(z)} = -\frac{1}{z+m_{fc}^{\vartheta}(z)} -\frac{\vartheta x}{(z+m_{fc}^{\vartheta}(z))^{2}}+\frac{\vartheta^{2}x^{2}}{(\vartheta x-z-m_{fc}^{\vartheta}(z))(z+m_{fc}^{\vartheta}(z))^{2}}.
\eeq
Since the first term is a constant, it vanishes after subtracting its expectation. The contribution of the last term is negligible, using the bound $\absv{z+m_{fc}^{\vartheta}(z)}\geq C$ obtained from \eqref{eq:mfczmaps}, together with the following bound of variance:
\begin{multline}
\biggexpct{\biggabsv{\frac{1}{\sqrt{N}\vartheta}\sum_{i}\bigg(\frac{\vartheta^{2}v_{i}^{2}}{(\vartheta v_{i}-z-m_{fc}^{\vartheta}(z))}-\biggexpct{\frac{\vartheta^{2}v_{i}^{2}}{(\vartheta v_{i}-z-m_{fc}^{\vartheta}(z))}}\bigg)}^{2}} \\
\leq\vartheta^{2}\int\biggabsv{\frac{x^{2}}{\vartheta x-z-m_{fc}^{\vartheta}(x)}}^{2}\dd\nu(x)\leq C\vartheta^{2}.
\end{multline}

To annihilate the dependence on $\vartheta$, with the help of  {Lemma \ref{lem:mfc-msc}}, we reduce the summand further to
\beq
\frac{x}{(z+m_{fc}^{\vartheta}(z))^{2}}-\frac{x}{(z+m_{sc}(z))^{2}} =\frac{x\big(2z+m_{sc}(z)+m_{fc}^{\vartheta}(z)\big)}{(z+m_{fc}^{\vartheta})^{2}(z+m_{sc}(z))^{2}}\big(m_{fc}^{\vartheta}(z)-m_{sc}(z)\big) =O(\vartheta),
\eeq
so that
\beq
\biggexpct{\biggabsv{\frac{2z+m_{sc}(z)+m_{fc}^{\vartheta}(z)}{(z+m_{fc}^{\vartheta})^{2}(z+m_{sc}(z))^{2}}\big(m_{fc}^{\vartheta}(z)-m_{sc}(z)\big)\frac{1}{\sqrt{N}}\sum_{i} v_{i}}^{2}}\leq C\vartheta^{2}
\eeq

Noting that the terms of variance $O(\vartheta^{2})$ vanishes for $\vartheta_{\infty}=0$, we may rewrite 
\begin{multline}
	\frac{1}{\sqrt{N}\vartheta}\sum_{i=1}^{N}\bigg[\frac{1}{\vartheta v_{i}-z-m_{fc}^{\vartheta}(z)}-\biggexpct{\frac{1}{\vartheta v_{i}-z-m_{fc}^{\vartheta}(z)}}\bigg] \\
	=-\frac{1}{(z+m_{sc}(z))^{2}}\bigg[\frac{1}{\sqrt{N}}\sum_{i}v_{i}\bigg] + X_{N} =-m_{sc}(z)^{2}\bigg[\frac{1}{\sqrt{N}}\sum_{i}v_{i}\bigg] +X_{N}
\end{multline}
where $X$ is a random variable with $\expct{\absv{X}^{2}}=O(\vartheta^{2})$.

As above, by the central limit theorem, the random vector
\beq
(\wt{\zeta}_{N}(z_{1}),\cdots,\wt{\zeta}_{N}(z_{p}))
\eeq
converges weakly to the Gaussian random vector with mean zero and covariance
\beq
\wt{\Gamma}(z_{i},z_{j})=(1+m_{sc}'(z_{i}))(1+m_{sc}'(z_{j}))m_{sc}(z_{i})^{2}m_{sc}(z_{j})^{2}\Var{v_{1}} =m_{sc}'(z_{i})m_{sc}'(z_{j})\Var{v_{1}}.
\eeq
Finally, by the same reasoning as above, the convergence can be extended to that of $(\xi_{N}^{\vartheta}(z_{1}),\cdots,\xi_{N}^{\vartheta}(z_{p}))$.

\subsection{Tightness of $\wt{\xi}_{N}^{\vartheta}(z)$}

In this section, we prove the tightness of $\{\wt{\xi}_{N}(z):z\in\caK\}$, which completes the proof of {Propositions \ref{prop:Gprocsqrttheta}} and {\ref{prop:Gprocsqrt}}. Given the finite-dimensional convergence of $\wt{\xi}_{N}(z)$, the tightness for a fixed point $z\in\caK$ directly follows, so that it remains to prove the H\"{o}lder condition, as given in {Section \ref{sec:Gaussianprocprf}}:
\beq
\frac{1}{N\vartheta^{2}}\expct{\absv{(\Tr R_{N}(z_{1})-\expct{\Tr R_{N}(z_{1})})-(\Tr R_{N}(z_{2})-\expct{\Tr R_{N}(z_{2})})}^{2}}\leq C\absv{z_{1}-z_{2}}^{2}
\eeq
for $z_{1},z_{2}\in\caK$ where $K$ is a fixed constant independent of $z_{1},z_{2}\in\caK$ and $N\in\N$.

As in {Section \ref{sec:Gaussianprocprf}}, we start with an application of the resolvent equation to get
\begin{multline}
\frac{1}{N\vartheta^{2}}\expct{\absv{(\Tr R_{N}(z_{1})-\expct{\Tr R_{N}(z_{1})})-(\Tr R_{N}(z_{2})-\expct{\Tr R_{N}(z_{2})})}^{2}} \\
=\frac{1}{N\vartheta^{2}}\absv{z_{1}-z_{2}}^{2}\expct{\absv{\Tr R_{N}(z_{1})R_{N}(z_{2})-\expct{\Tr R_{N}(z_{1})\Tr R_{N}(z_{2})}}^{2}}.
\end{multline}

The proof of the corresponding bound is also given in \ref{sec:tight}:
\begin{lemma}\label{lem:tightprocsqrt}
	For $z_{1},z_{2}\in\caK$ and $N\in\N$ sufficiently large, we have
	\beq
	\frac{1}{N\vartheta^{2}}\expct{\absv{\Tr R_{N}(z_{1})R_{N}(z_{2})-\expct{\Tr R_{N}(z_{1})R_{N}(z_{2})}}^{2}}\leq K,
	\eeq
	where $K>0$ is a constant independent of $z_{1},z_{2}\in\caK$ and $N\in\N$ sufficiently large.
\end{lemma}

\section*{Acknowledgments}
We are grateful to the anonymous referee for carefully reading our manuscript and providing helpful comments.
This work was supported in part by Samsung Science and Technology Foundation project number SSTF-BA1402-04.

\appendix

\section{Mean Function $b(z)$}\label{sec:meanb}
As mentioned in {Section \ref{sec:Gaussianprocprf}}, in this section we suppose that $V$ is \textbf{deterministic} and calculate the limiting formula of the mean $\expct{\xi_{N}(z)}$. First, in {Section \ref{sec:mean1}--\ref{sec:mean3}}, we reduce the mean of $\xi_{N}^{\vartheta}(z)$ to a form depending on $\vartheta$. Then, we repeatedly use {Lemma \ref{lem:mfc-msc}} to replace $\vartheta$ with its limit $\vartheta_{\infty}$. Throughout the primary simplification, all the bounds we require are irrelevant of $\vartheta$, therefore we drop the superscript $\vartheta$ in {Sections \ref{sec:mean1}--\ref{sec:mean3}}. Note that we are assuming that $N$ is sufficient large so that $\prob{\Omega_{N}}=1$.

Set
\beq
b_{N}(z):=\expct{\xi_{N}(z)}=N\expct{m_{N}(z)-\wh{m}_{fc}(z)}.
\eeq
Then, letting 
\beq
Q_{i}\deq -W_{ii}+\sum_{p,q}^{(i)}W_{ip}R_{pq}^{(i)}W_{qi},
\eeq
we get 
\begin{multline}\label{eq:resodiagexpa}
	R_{ii}= \frac{1}{-z-Q_{i}} =\frac{1}{-z-(\wh{m}_{fc}(z)-v_{i})} +\frac{Q_{i}-(\wh{m}_{fc}(z)-v_{i})}{(-z-(\wh{m}_{fc}(z)-v_{i}))^{2}} \\
	+\frac{(Q_{i}-(\wh{m}_{fc}(z)-v_{i})^{2}}{(-z-(\wh{m}_{fc}(z)-v_{i}))^{3}} +\frac{1}{-z-Q_{i}}\bigg(\frac{Q_{i}-(\wh{m}_{fc}(z)-v_{i})}{-z-(\wh{m}_{fc}(z)-v_{i})}\bigg)^{3} \\
	=\wh{g}_{i}(z) +\wh{g}_{i}(z)^{2}(Q_{i}-(\wh{m}_{fc}(z)-v_{i})) 
	+\wh{g}_{i}(z)^{3}(Q_{i}-(\wh{m}_{fc}(z)-v_{i}))^{2} +R_{ii}\bigg(\frac{Q_{i}-(\wh{m}_{fc}(z)-v_{i})}{-z-(\wh{m}_{fc}(z)-v_{i})}\bigg)^{3}.
\end{multline}

By the local law $\absv{R_{ii}-\hat{g}_{i}(z)}=\caO(N^{-1/2})$, 
\beq\label{eq:qcentred}
Q_{i}-(\wh{m}_{fc}(z)-v_{i})=-\frac{1}{R_{ii}}-z-\wh{m}_{fc}(z)+v_{i}=-\frac{1}{R_{ii}}+\frac{1}{\wh{g}_{i}(z)}=\caO(N^{-\frac{1}{2}}),
\eeq
which implies 

\begin{equation}
b_{N}(z)=\sum_{i=1}^{N}\expct{\wh{g}_{i}(z)^{2}(Q_{i}-(\wh{m}_{fc}(z)-v_{i})) +\wh{g}_{i}(z)^{3}(Q_{i}-(\wh{m}_{fc}(z)-v_{i}))^{2}}+O(N^{-\frac{1}{2}+\epsilon})
\end{equation}

\subsection{$\expct{Q_{i}-(\wh{m}_{fc}(z)-v_{i})}$}\label{sec:mean1}
By the definition of $Q_{i}$, 
\begin{equation}
\expct{Q_{i}-\wh{m}_{fc}(z)+v_{i}} =\expct{-W_{ii}+\sum_{p,q}^{(i)}W_{ip}R_{pq}^{(i)}W_{qi}}-\wh{m}_{fc}(z)+v_{i}
=\frac{1}{N}\biggexpct{\sum_{p}^{(i)}R_{pp}^{(i)}}-\wh{m}_{fc}(z)
\end{equation}

On the other hand, \eqref{eq:mat id 3} together with the local law gives 
\begin{multline}
\sum_{p}^{(i)}(R_{pp}^{(i)}-R_{pp}) =-\sum_{p}^{(i)}\frac{R_{pi}R_{ip}}{R_{ii}} =-\sum_{p}^{(i)}\bigg(\wh{g}_{i}(z)^{-1}R_{pi}R_{ip}+\frac{R_{pi}R_{ip}(\wh{g}_{i}(z)-R_{ii})}{R_{ii}\wh{g}_{i}(z)}\bigg) \\
=-\wh{g}_{i}(z)^{-1}\sum_{p}^{(i)}R_{pi}R_{ip} + \caO(N^{-\frac{1}{2}}),
\end{multline}
which in turn implies 
\beq
\sum_{p}^{(i)}R_{pp}^{(i)} =Nm_{N}(z)-R_{ii}-\frac{1}{\wh{g}_{i}(z)}\big((R^{2})_{ii}-R_{ii}^{2}\big)+\caO(N^{-\frac{1}{2}}) =Nm_{N}(z)-\frac{1}{\wh{g}_{i}}(R^{2})_{ii}+\caO(N^{-\frac{1}{2}}),
\eeq
so that
\begin{multline}
\sum_{i}\wh{g}_{i}(z)^{2}\expct{Q_{i}-\wh{m}_{fc}(z)+v_{i}} =\sum_{i}\wh{g}_{i}(z)^{2}\bigg(\biggexpct{\frac{1}{N}\sum_{p}^{(i)}R_{pp}}-\wh{m}_{fc}(z)\bigg) \\
=\frac{1}{N}\sum_{i}\wh{g}_{i}(z)^{2}\biggexpct{Nm_{N}(z)-\frac{1}{\wh{g}_{i}}(R^{2})_{ii}} -\wh{m}_{fc}(z)\sum_{i}\wh{g}_{i}(z)^{2} +O(N^{-\frac{1}{2}+\epsilon}) \\
=b_{N}(z)\frac{1}{N}\Big(\sum_{i}\wh{g}_{i}(z)^{2}\Big)-\frac{1}{N}\Bigexpct{\sum_{i}\wh{g}_{i}(z)(R^{2})_{ii}} +O(N^{-\frac{1}{2}+\epsilon}).
\end{multline}

Considering $(R^{2})_{ii}$, we have 
\beq
R_{ij}=\bilin{e_{i}}{Re_{j}} =\sum_{\alpha}\frac{1}{\lambda_{\alpha}-z}\bilin{e_{i}}{v_{\alpha}}\bilin{v_{\alpha}}{e_{j}}
\eeq
where $v_{\alpha}$ denotes the eigenvector of $W$ corresponding to $\lambda_{\alpha}$. Thus
\begin{multline}
(R^{2})_{ij}=\sum_{k}R_{ik}R_{kj}=\sum_{k}\Big(\sum_{\alpha}\frac{1}{\lambda_{\alpha}-z}\bilin{e_{i}}{v_{\alpha}}\bilin{v_{\alpha}}{e_{k}}\Big)\Big(\sum_{\beta}\frac{1}{\lambda_{\beta}-z}\bilin{e_{k}}{v_{\beta}}\bilin{v_{\beta}}{e_{j}}\Big) \\
=\sum_{\alpha,\beta}\frac{\bilin{e_{i}}{v_{\alpha}}\bilin{v_{\beta}}{e_{j}}}{(\lambda_{\alpha}-z)(\lambda_{\beta}-z)}\sum_{k}\bilin{v_{\alpha}}{e_{k}}\bilin{e_{k}}{v_{\beta}} =\sum_{\alpha,\beta}\frac{\bilin{e_{i}}{v_{\alpha}}\bilin{v_{\beta}}{e_{j}}}{(\lambda_{\alpha}-z)(\lambda_{\beta}-z)}\delta_{\alpha,\beta} \\
=\sum_{\alpha}\frac{\bilin{e_{i}}{v_{\alpha}}\bilin{v_{\alpha}}{e_{j}}}{(\lambda_{\alpha}-z)^{2}} =\frac{\dd}{\dd z}R_{ij}(z).
\end{multline}
Then the local law $\absv{R_{ii}-\wh{g}_{i}}\prec N^{-1/2}$ together with the Cauchy integral formula (applied on a small contour of length $N^{-\delta}$ enclosing $z$) gives 
\beq
\biggabsv{(R^{2})_{ii}-\frac{\dd}{\dd z}\wh{g}_{i}} \prec N^{-\frac{1}{2}+\delta}.
\eeq
Plugging this in, we get
\beq
\frac{1}{N}\expct{\sum_{i}\wh{g}_{i}(R^{2})_{ii}}=\frac{1}{N}\sum_{i}\wh{g}_{i}\frac{\dd}{\dd z}\wh{g}_{i} +O(N^{-\frac{1}{2}+\delta+\epsilon}).
\eeq

On the other hand, 
\begin{multline}
	\frac{1}{N}\sum_{i}\wh{g}_{i}\frac{\dd}{\dd z}\wh{g}_{i} =\frac{1}{2}\frac{\dd}{\dd z}\Big(\frac{1}{N}\sum_{i}\wh{g}_{i}^{2}\Big) =\frac{1}{2}\frac{\dd}{\dd z}\int_{\R}\frac{1}{(x-z-\wh{m}_{fc}(z))^{2}}\dd\wh{\nu}(x) \\
	=\frac{1}{2}\frac{\dd}{\dd z}\Big(\frac{\wh{m}_{fc}'(z)}{1+\wh{m}_{fc}'(z)}\Big) =\frac{1}{2}\frac{\wh{m}_{fc}''(z)}{(1+\wh{m}_{fc}'(z))^{2}}=\frac{1}{2}\frac{m_{fc}''(z)}{(1+m_{fc}'(z))^{2}}+O(\vartheta N^{-\alpha_{0}}).
\end{multline}

Combining the results above, we get
\begin{equation}
	\sum_{i}\expct{\wh{g}_{i}(z)^{2}(Q_{i}-\wh{m}_{fc}(z)+v_{i})} =b_{N}(z)\frac{m_{fc}'(z)}{1+m_{fc}'(z)} -\frac{1}{2}\frac{m_{fc}''(z)}{(1+m_{fc}'(z))^{2}} 
	+O(\vartheta N^{-\alpha_{0}}+N^{-\frac{1}{2}+\delta+\epsilon}).
\end{equation}

\subsection{$\expct{(Q_{i}-(\wh{m}_{fc}(z)-v_{i}))^{2}}$}

\begin{multline}
\expct{(Q_{i}-(\wh{m}_{fc}(z)-v_{i}))^{2}} =\biggexpct{\bigg(-\frac{1}{\sqrt{N}}A_{ii}+\sum_{p,q}^{(i)}W_{ip}R_{pq}^{(i)}W_{qi}-\wh{m}_{fc}(z)\bigg)^{2}} \\
=\wh{m}_{fc}(z)^{2}+\frac{w_{2}}{N}+\biggexpct{\bigg(\sum_{p,q}^{(i)}W_{ip}R_{pq}^{(i)}W_{qi}\bigg)^{2}}-2\wh{m}_{fc}(z)\biggexpct{\sum_{p,q}^{(i)}W_{ip}R_{pq}^{(i)}W_{qi}} \\
=\wh{m}_{fc}^{2}-2\wh{m}_{fc}\expct{m_{N}^{(i)}}+\frac{w_{2}}{N}+\biggexpct{\bigg(\sum_{p,q}^{(i)}W_{ip}R_{pq}^{(i)}W_{qi}\bigg)^{2}}.
\end{multline}

Expanding the last term,
\beq
\biggexpct{\bigg(\sum_{p,q}^{(i)}W_{ip}R_{pq}^{(i)}W_{qi} \bigg)^{2}} =\sum_{p,q,r,t}^{(i)}\expct{W_{ip}W_{qi}W_{ir}W_{ti}}\expct{R_{pq}^{(i)}R_{rt}^{(i)}}
\eeq

As $\expct{W_{ip}W_{qi}W_{ir}W_{ti}}\neq0$ implies $\absv{\{p,q,r,t\}}=1$ or each index is repeated twice, we separate into two cases.
\begin{enumerate}[(i)]
	\item $\absv{\{p,q,r,t\}}=1$.
	\beq
	\sum_{p}^{(i)}\expct{W_{ip}^{4}}\expct{(R_{pp}^{(i)})^{2}} =\frac{W_{4}}{N^{2}} \sum_{p}^{(i)}\expct{(R_{pp}^{(i)})^{2}} =\frac{W_{4}}{N^{2}}\sum_{p}\wh{g}_{p}(z)^{2}+O(N^{-\frac{3}{2}+\epsilon})
	\eeq
	
	\item $\absv{\{p,q,r,t\}}=2$.
	\begin{enumerate}[(a)]
		\item For $p=q\neq r=t$:
		Since
		\beq\label{eq:mr-emg}
		R_{pp}^{(i)}-\wh{g}_{p}(z)=R_{pp}-\wh{g}_{p}(z)+\frac{R_{pi}R_{ip}}{R_{ii}}=\caO(N^{-\frac{1}{2}}),
		\eeq
		we get
		\begin{multline}
		\sum_{p\neq r}^{(i)}\expct{W_{ip}^{2}W_{ir}^{2}}\expct{R_{pp}^{(i)}R_{rr}^{(i)}} =\frac{1}{N^{2}}\expct{\sum_{p\neq r}^{(i)}R_{pp}R_{rr}} \\
		=\frac{1}{N^{2}}\biggexpct{\sum_{p}^{(i)}R_{pp}^{(i)}(Nm_{N}^{(i)}-R_{pp}^{(i)})} =\expct{(m_{N}^{(i)})^{2}}-\frac{1}{N^{2}}\sum_{p}(\wh{g}_{p}(z))^{2}+O(N^{-\frac{3}{2}+\epsilon}).
		\end{multline}
		
		\item For $p=t \neq q=r$:
		\begin{multline}
		\sum_{p\neq r}^{(i)}\expct{W_{ip}^{2}W_{ir}^{2}}\expct{(R_{pr}^{(i)})^{2}} 
		=\frac{1}{N^{2}}\biggexpct{\sum_{p\neq r}^{(i)}(R_{pr}^{(i)})^{2}}
		=\frac{1}{N^{2}}\bigg(\expct{\Tr(R^{(i)})^{2}}-\biggexpct{\sum_{p}^{(i)}(R_{pp}^{(i)})^{2}}\bigg) \\
		=\frac{1}{N}(m_{N}^{(i)})'-\frac{1}{N^{2}}\sum_{p}\wh{g}_{p}(z)^{2}+O(N^{-\frac{3}{2}+\epsilon}) =\frac{1}{N}\wh{m}_{fc}'-\frac{1}{N^{2}}\sum_{p}\wh{g}_{p}(z)^{2}+O(N^{-\frac{3}{2}+\epsilon}),
		\end{multline}
		where we have used \eqref{eq:mr-emg} in the third equality and 
		\beq
		m_{N}^{(i)}-\wh{m}_{fc}=\frac{1}{N}\sum_{p}^{(i)}(R_{pp}^{(i)}-R_{pp})-\frac{1}{N}R_{ii}+(m_{N}-\wh{m}_{fc})=\caO(N^{-1})
		\eeq
		together with the Cauchy integral formula in the last equality.
		
		\item For $p=r \neq q=t$:
		
		By symmetry, the result is same as above.
		
	\end{enumerate}
	Hence we get 
	\begin{multline}
	\biggexpct{\Big(\sum_{p,q}^{(i)}W_{ip}R_{pq}^{(i)}W_{qi}\Big)^{2}} =\frac{W_{4}}{N^{2}}\sum_{p}\wh{g}_{p}^{2} +\frac{2}{N}\wh{m}_{fc}' -\frac{3}{N^{2}}\sum_{p}\wh{g}_{p}^{2}+\expct{(m_{N}^{(i)})^{2}} +O(N^{-\frac{3}{2}+\epsilon}) \\
	=\expct{(m_{N}^{(i)})^{2}} +\frac{W_{4}-3}{N^{2}}\sum_{p}\wh{g}_{p}^{2} +\frac{2}{N}\wh{m}_{fc}' +O(N^{-\frac{3}{2}+\epsilon})
	\end{multline}
\end{enumerate}

\begin{remark}
	For complex Hermitian $A$ where we assume the independence of the real and the complex entries and $\expct{A_{ij}^{2}}=0$ for $i\neq j$, (ii)-(c) leads to $0$.
\end{remark}

Summing up, we get from $\absv{m_{N}^{(i)}-\wh{m}_{fc}}\prec N^{-1}$
\begin{multline}
	\expct{(Q_{i}-\wh{m}_{fc}(z)+v_{i})^{2}}
	=\wh{m}_{fc}^{2}-2\wh{m}_{fc}\expct{m_{N}^{(i)}}+\frac{w_{2}}{N}+\expct{(m_{N}^{(i)})^{2}} +\frac{W_{4}-3}{N^{2}}\sum_{p}\wh{g}_{p}^{2} +\frac{2}{N}\wh{m}_{fc}' +O(N^{-\frac{3}{2}+\epsilon}) \\
	=\frac{1}{N}(w_{2}+2\wh{m}_{fc}')+\frac{W_{4}-3}{N^{2}}\sum_{p}\wh{g}_{p}^{2}+O(N^{-\frac{3}{2}+\epsilon}).
\end{multline}

Therefore
\beq
\sum_{i}\wh{g}_{i}^{3}\expct{(Q_{i}-\wh{m}_{fc}+v_{i})^{2}} =\Big(\frac{1}{N}\sum_{i}\wh{g}_{i}^{3}\Big)(w_{2}+2\wh{m}_{fc}')+(W_{4}-3)\Big(\frac{1}{N}\sum_{i}\wh{g}_{i}^{3}\Big)\Big(\frac{1}{N}\sum_{p}\wh{g}_{p}^{2}\Big).
\eeq

On the other hand, {Corollary \ref{cor:covwh}} implies
\beq\label{eq:g^2sum}
\frac{1}{N}\sum_{i}\wh{g}_{i}^{2}=\int_{\R}\frac{1}{(x-z-\wh{m}_{fc}(z))^{2}}\dd\wh{\nu}(x) =\frac{\wh{m}_{fc}'(z)}{1+\wh{m}_{fc}'(z)} =\frac{m_{fc}'(z)}{1+m_{fc}'(z)}+O(\vartheta N^{-\alpha_{0}}).
\eeq
Similarly, differentiating \eqref{eq:g^2sum} we have
\beq
\frac{1}{N}\sum_{i}\wh{g}_{i}^{3} =\int_{\R}\frac{1}{(x-z-\wh{m}_{fc}(z))^{3}}\dd\wh{\nu}(x) = \frac{1}{2}\frac{\wh{m}_{fc}''(z)}{(1+\wh{m}_{fc}'(z))^{3}} =\frac{1}{2}\frac{m_{fc}''(z)}{(1+m_{fc}'(z))^{3}} +O(\vartheta N^{-\alpha_{0}}).
\eeq

\subsection{The mean $b(z)$}\label{sec:mean3}
Summing up the results, we get
\begin{multline}
b_{N}(z) =\sum_{i=1}^{N}\expct{\wh{g}_{i}(z)^{2}(Q_{i}-(\wh{m}_{fc}(z)-v_{i})) +\wh{g}_{i}(z)^{3}(Q_{i}-(\wh{m}_{fc}(z)-v_{i}))^{2}} +O(N^{-\frac{1}{2}+\epsilon}) \\
=b_{N}(z)\frac{m_{fc}'(z)}{1+m_{fc}'(z)} -\frac{1}{2}\frac{m_{fc}''(z)}{(1+m_{fc}'(z))^{2}}+\frac{1}{2}\frac{m_{fc}''(z)}{(1+m_{fc}'(z))^{3}}(w_{2}+2m_{fc}'(z)) \\
+\frac{W_{4}-3}{2}\frac{m_{fc}'(z)}{1+m_{fc}'(z)}\frac{m_{fc}''(z)}{(1+m_{fc}'(z))^{3}} +O(\vartheta N^{-\alpha_{0}}+N^{-\frac{1}{2}+\epsilon}),
\end{multline}
hence
\begin{multline}
b_{N}(z) = -\frac{1}{2}\frac{m_{fc}''(z)}{(1+m_{fc}'(z))}+\frac{1}{2}\frac{m_{fc}''(z)}{(1+m_{fc}'(z))^{2}}(w_{2}+2m_{fc}'(z)) \\
+\frac{W_{4}-3}{2}\frac{m_{fc}'(z)m_{fc}''(z)}{(1+m_{fc}'(z))^{3}} +O(\vartheta N^{-\alpha_{0}}+N^{-\frac{1}{2}+\epsilon}) \\
= \frac{1}{2}\frac{m_{fc}''(z)}{(1+m_{fc}'(z))^{2}} \bigg[(w_{2}-1)+m_{fc}'(z)+(W_{4}-3)\frac{m_{fc}'(z)}{1+m_{fc}'(z)}\bigg] +O(\vartheta N^{-\alpha_{0}}+N^{-\frac{1}{2}+\epsilon}).
\end{multline}

\subsubsection{$\vartheta_{\infty}>0$}
Now retrieving the dependence on $\vartheta$, we have
\begin{multline}
	b_{N}^{\vartheta}(z)= \frac{1}{2}\frac{(m_{fc}^{\vartheta})''(z)}{(1+(m_{fc}^{\vartheta})'(z))^{2}} \bigg[(w_{2}-1)+(m_{fc}^{\vartheta})'(z)+(W_{4}-3)\frac{(m_{fc}^{\vartheta})'(z)}{1+(m_{fc}^{\vartheta})'(z)}\bigg] \\
	+O(\vartheta N^{-\alpha_{0}}+N^{-\frac{1}{2}+\epsilon}).
\end{multline}

Using {Lemma \ref{lem:mfc-msc}}, it suffices to prove that $1+\frac{\dd}{\dd z}m_{fc}^{\vartheta}(z)$ is lower bounded, uniformly in $\vartheta$. Recalling the self-consistent equation \eqref{eq:funceq}, we write
\beq
\frac{\dd}{\dd z}m_{fc}^{\vartheta}(z)=\int_{\R}\frac{1+\frac{\dd}{\dd z}m_{fc}^{\vartheta}(z)}{(\vartheta x-z-m_{fc}^{\vartheta}(z))^{2}}\dd\nu(x),
\eeq
so that
\beq
\biggabsv{1+\frac{\dd}{\dd z}m_{fc}^{\vartheta}(z)}=\biggabsv{\bigg(1-\int_{\R}\frac{1}{(\vartheta x-z-m_{fc}^{\vartheta}(z))^{2}}\dd\nu(x)\bigg)^{-1}}.
\eeq
Then \eqref{eq:1-s1s2} implies the required lower bound, which together with several applications of Cauchy integral proves that
\beq
b_{N}^{\vartheta}(z) =b^{\vartheta_{\infty}}(z) +O(\absv{\vartheta-\vartheta_{\infty}}+N^{-\alpha_{0}}+N^{-\frac{1}{2}+\epsilon}),
\eeq
where
\beq
b^{\vartheta_{\infty}}(z)\deq \frac{(m_{fc}^{\vartheta_{\infty}})''(z)}{2(1+(m_{fc}^{\vartheta_{\infty}})'(z))^{2}}\bigg[(w_{2}-1)+(m_{fc}^{\vartheta_{\infty}})'(z)+(W_{4}-3)\frac{(m_{fc}^{\vartheta_{\infty}})'(z)}{1+(m_{fc}^{\vartheta_{\infty}})'(z)}\bigg].
\eeq

\begin{remark}
	If we consider $V=0$, then we have $\rho_{fc}=\wh{\rho}_{fc}=\rho_{sc}$ and $m_{fc}=\wh{m}_{fc}=\wh{g}_{i}=m_{sc}$. In this case, the self-consistent equation $z+m_{sc}(z)=-\frac{1}{m_{sc}(z)}$ gives
	\beq\label{eq:s^2}
	m_{sc}(z)^{2} =\frac{1}{(z+m_{sc}(z))^{2}} =\frac{m_{sc}'(z)}{1+m_{sc}'(z)}
	\eeq
	and similarly differentiating the second equality in \eqref{eq:s^2} implies
	\beq
	m_{sc}(z)^{3} = \frac{1}{2}\frac{m_{sc}''(z)}{(1+m_{sc}'(z))^{3}}
	\eeq
	so that
	\beq
	b_{N}(z) \to m_{sc}(z)^{3}(1+m_{sc}'(z))((w_{2}-1)+m_{sc}'(z)+(W_{4}-3)m_{sc}(z)^{2}),
	\eeq
	which is given in Proposition 3.1 of \cite{Bai-Yao2005}.
\end{remark}

\subsubsection{$\vartheta=o(1)$}
To handle with the case $\vartheta=o(1)$, we rewrite the result above with the superscript $\vartheta$ recaptured:
\begin{multline}
b_{N}^{\vartheta}(z) = \frac{1}{2} \frac{(m_{fc}^{\vartheta})''(z)}{(1+(m_{fc}^{\vartheta})'(z))^{2}}\bigg[(w_{2}-1)+(m_{fc}^{\vartheta})'(z)+(W_{4}-3)\frac{(m_{fc}^{\vartheta})'(z)}{1+(m_{fc}^{\vartheta})'(z)}\bigg] \\
+O(\vartheta N^{-\alpha_{0}}+N^{-\frac{1}{2}+\epsilon}),
\end{multline}
uniformly for $z\in\Gamma\setminus\Gamma_{0}$.

Recalling {Lemma \ref{lem:mfc-msc}}, again the Cauchy integral formula gives the same bound for $\absv{(m_{fc}^{\vartheta})'-m_{sc}'}$ and $\absv{(m_{fc}^{\vartheta})''-m_{sc}''}$. Then from the stability bound \eqref{eq:1-s1s2} gives the uniform convergence
\begin{multline}
b_{N}^{\vartheta}(z) = \frac{1}{2} \frac{m_{sc}''(z)}{(1+m_{sc}'(z))^{2}}\bigg[(w_{2}-1)+m_{sc}'(z)+(W_{4}-3)\frac{m_{sc}'(z)}{1+m_{sc}'(z)}\bigg] +O(\vartheta+N^{-\frac{1}{2}+\epsilon}) \\
=m_{sc}(z)^{3}(1+m_{sc}'(z))((w_{2}-1)+m'_{sc}(z)+(W_{4}-3)m_{sc}(z)^{2})+O(\vartheta+N^{-\frac{1}{2}+\epsilon}),
\end{multline}
by the preceding remark.

\section{Covariance Function}\label{sec:covariance}
As indicated in {Section \ref{sec:Gaussianprocprf}}, this section is devoted to the proof of {Lemma \ref{lem:martclt11}} and {\ref{lem:martclt22}}. Similar to the preceding section, we simplify the covariance in {Sections \ref{sec:martdecop}--\ref{sec:reduccov}} without tracking the superscript $\vartheta$, as the bounds used along the simplification is independent of $\vartheta$. Also, throughout this section, we assume $z\in\Gamma_{u}\cup\Gamma_{d}$.

\subsection{Martingale decomposition}\label{sec:martdecop}
The matrix identities \eqref{eq:mat id 3} implies
\begin{multline}
\Tr R -\Tr R^{(k)} =R_{kk}+\sum_{i}^{(k)}\frac{R_{ik}R_{ki}}{R_{kk}} =R_{kk}+\sum_{i}^{(k)}R_{kk}\cdot\frac{R_{ik}}{R_{kk}}\cdot\frac{R_{ki}}{R_{kk}} \\
=R_{kk}\Big(1+\sum_{i}^{(k)}\Big(-\sum_{p}^{(i)}W_{kp}R_{pi}^{(k)}\Big)^{2}\Big) =R_{kk}\Big(1+\sum_{i}^{(k)}\sum_{p,q}^{(i)}W_{kp}R_{pi}^{(k)}W_{qk}R_{iq}^{(k)}\Big),
\end{multline}
which can be further rewritten as
\begin{equation}
=R_{kk}\Big(1+\sum_{p,q}^{(k)}W_{kp}W_{qk}\Big(\sum_{i}^{(k)}R_{pi}^{(k)}R_{iq}^{(k)}\Big)\Big) =R_{kk}\Big(1+\sum_{p,q}^{(k)}W_{kp}(R^{(k)})^{2}_{pq}W_{qk}\Big),
\end{equation}
so that 
\beq\label{eq:martdecomp}
\zeta_{N} =\sum_{k=1}^{N}(\E_{k-1}-\E_{k})\Big(R_{kk}\Big(1+\sum_{p,q}^{(k)}W_{kp}(R^{(k)})^{2}_{pq}W_{qk}\Big)\Big).
\eeq

Recalling the expansion \eqref{eq:resodiagexpa}, we take
\beq
X_{k} \deq(\wh{g}_{k}+\wh{g}_{k}^{2}(Q_{k}-\wh{m}_{fc}+v_{k})) \Big(1+\sum_{p,q}^{(k)}W_{kp}(R^{(k)})^{2}_{pq}W_{qk}\Big)
\eeq
and
\beq
Y_{k} \deq R_{kk}\Big(\frac{Q_{k}-\wh{m}_{fc}+v_{k}}{-z-\wh{m}_{fc}+v_{k}}\Big)^{2}\Big(1+\sum_{p,q}^{(k)}W_{kp}(R^{(k)})^{2}_{pq}W_{qk}\Big),
\eeq
so that $X_{k}+Y_{k}=\Tr R-\Tr R^{(k)}$. Using {Lemma \ref{lem:lde}}, we have 
\beq
\sum_{p,q}^{(k)}W_{kp}(R^{(k)})^{2}_{pq}W_{qk}=\caO(1),
\eeq
so that $Y_{k}=\caO(N^{-1})$.

On the other hand, we observe that for $k>l$,
\beq
\expct{(\E_{k-1}-\E_{k})Y_{k}\cdot(\E_{l-1}-\E_{l})\overline{Y}_{l}} =\expct{(\E_{k-1}-\E_{k})Y_{k}\cdot\expctk{(\E_{l-1}-\E_{l})\overline{Y}_{l}}{k-1}}=0,
\eeq
which in turn gives
\beq\label{eq:orthogonal}
\biggexpct{\biggabsv{\sum_{k=1}^{N}(\E_{k-1}-\E_{k})Y_{k}}^{2}} =\biggexpct{\sum_{k=1}^{N}\biggabsv{(\E_{k-1}-\E_{k})Y_{k}}^{2}} =O(N^{-1+2\epsilon}).
\eeq
Hence a typical application of Markov inequality implies
\begin{multline}\label{eq:cov}
\zeta_{N} =\sum_{k=1}^{N}(\E_{k-1}-\E_{k})X_{k}+\caO_{p}(N^{-\frac{1}{2}}) =\sum_{k=1}^{N}\wh{g}_{k}(\E_{k-1}-\E_{k})\Big[1+\sum_{p,q}^{(k)}W_{kp}(R^{(k)})^{2}_{pq}W_{qk}\Big] \\
+\sum_{k=1}^{N}\wh{g}_{k}^{2}(\E_{k-1}-\E_{k})\Big[(Q_{k}-\wh{m}_{fc}+v_{k})\Big(1+\sum_{p,q}^{(k)}W_{kp}(R^{(k)})^{2}_{pq}W_{qk}\Big)\Big] +\caO_{p}(N^{-\frac{1}{2}}).
\end{multline}
where $\caO_{p}(N^{-\frac{1}{2}})$ stands for the terms bounded by $N^{-1/2+\epsilon}$ in probability.

\subsubsection{The first term}
We consider the first term of \eqref{eq:cov}:
\begin{multline}
\biggexpctk{\sum_{p,q}^{(k)}W_{kp}(R^{(k)})^{2}_{pq}W_{qk}}{k} =\sum_{p}^{(k)}\frac{1}{N}\expctk{(R^{(k)})^{2}_{pp}}{k} \\ =\biggexpctk{\frac{1}{N}\sum_{p}^{(k)}(R^{(k)})^{2}_{pp}}{k-1} =\expctk{(m_{N}^{(k)})'}{k-1} =\wh{m}_{fc}'+\caO(N^{-1}).
\end{multline} 
Hence the first term is given by
\beq
(\E_{k-1}-\E_{k})\Big[\sum_{p,q}^{(k)}W_{kp}(R^{(k)})^{2}_{pq}W_{qk}\Big] =\biggexpctk{\sum_{p,q}^{(k)}W_{kp}(R^{(k)})^{2}_{pq}W_{qk}-\wh{m}_{fc}'}{k-1}+\caO(N^{-1}).
\eeq

\subsubsection{The second term}
To calculate the second term, we first observe that the {Lemma \ref{lem:lde}} and \eqref{eq:qcentred} imply
\beq
(Q_{k}-\wh{m}_{fc}+v_{k})\bigg(\sum_{p,q}^{(k)}W_{kp}(R^{(k)})^{2}_{pq}W_{qk}-\frac{1}{N}\sum_{p}^{(k)}(R^{(k)})^{2}_{pp}\bigg) \prec N^{-1}.
\eeq

Therefore
\begin{multline}
(\E_{k-1}-\E_{k})\Big[(Q_{k}-\wh{m}_{fc}+v_{k})\Big(1+\sum_{p,q}^{(k)}W_{kp}(R^{(k)})^{2}_{pq}W_{qk}\Big)\Big] \\ =(\E_{k-1}-\E_{k})\bigg[(Q_{k}-\wh{m}_{fc}+v_{k})\bigg(1+\frac{1}{N}\sum_{p}^{(k)}(R^{(k)})^{2}_{pp}\bigg)\bigg] +\caO(N^{-1})\\
=(\E_{k-1}-\E_{k})[(Q_{k}-\wh{m}_{fc}+v_{k})](1+\wh{m}_{fc}') +\caO(N^{-1}).
\end{multline}

As above, we first reduce the term concerning $\E_{k}$:
\begin{multline}
	\expctk{Q_{k}-\wh{m}_{fc}+v_{k}}{k} =\biggexpctk{-W_{kk}+\sum_{p,q}^{(k)}W_{kp}R_{pq}^{(k)}W_{qk}-\wh{m}_{fc}+v_{k}}{k}
	=\biggexpctk{\sum_{p,q}^{(k)}W_{kp}R_{pq}^{(k)}W_{qk}-\wh{m}_{fc}}{k}\\ 
	=\biggexpctk{\frac{1}{N}\sum_{p}^{(k)}R^{(k)}_{pp}-\wh{m}_{fc}}{k} 
	=\expctk{ m_{N}^{(k)}-\wh{m}_{fc}}{k}=\caO(N^{-1}).
\end{multline}
Hence the second term is given by
\begin{multline}
	(\E_{k-1}-\E_{k})\Big[(Q_{k}-\wh{m}_{fc}+v_{k})\Big(1+\sum_{p,q}^{(k)}W_{kp}(R^{(k)})^{2}_{pq}W_{qk}\Big)\Big]
	=\expctk{Q_{k}-\wh{m}_{fc}+v_{k}}{k-1}(1+\wh{m}_{fc}')+\caO(N^{-1}) \\
	=\biggexpctk{-W_{kk}+v_{k}+\sum_{p,q}^{(k)}W_{kp}R^{(k)}_{pq}W_{qk}-\wh{m}_{fc}}{k-1}(1+\wh{m}_{fc}')+\caO(N^{-1}).
\end{multline}

\subsection{Simplification}
Combining the results above and using the argument of \eqref{eq:orthogonal}, we get
\beq
\zeta_{N}=\sum_{k=1}^{N}\expctk{\phi_{k}}{k-1}+\caO_{p}(N^{-\frac{1}{2}})
\eeq
where
\begin{equation}
\phi_{k}=\wh{g}_{k}\Big(\sum_{p,q}^{(k)}W_{kp}(R^{(k)})^{2}_{pq}W_{qk}-\wh{m}_{fc}'\Big) 
+(1+\wh{m}_{fc}')\wh{g}_{k}^{2}\Big(-W_{kk}+v_{k}+\sum_{p,q}^{(k)}W_{kp}R^{(k)}_{pq}W_{qk}-\wh{m}_{fc}\Big).
\end{equation}

Using the identities $\frac{\dd}{\dd z}R_{p,q}^{(k)}=(R^{(k)})^{2}_{pq}$ and $\wh{g}_{k}'=(1+\wh{m}_{fc}')\wh{g}_{k}^{2}$, we get
\beq
\phi_{k}=\frac{\dd}{\dd z}\Big[\wh{g}_{k}\Big(-W_{kk}+v_{k}+\sum_{p,q}^{(k)}W_{kp}R^{(k)}_{pq}W_{qk}-\wh{m}_{fc}\Big)\Big].
\eeq

Since
\beq
-W_{kk}+v_{k}+\sum_{p,q}^{(k)}W_{kp}R_{pq}^{(k)}W_{qk}-\wh{m}_{fc}=Q_{k}+v_{k}-\wh{m}_{fc} =\caO(N^{-\frac{1}{2}})
\eeq
and $\wh{g}_{k}(z)\sim 1$ as $z\in\Gamma_{u}\cup\Gamma_{d}$, we have $\phi_{k}\prec N^{-1/2}$.

\subsection{Covariance}
Let $z_{1},\cdots,z_{p}\in\Gamma_{u}$ be distinct points. (Note that by symmetry $\xi_{N}(\overline{z})=\overline{\xi_{N}(z)}$, it suffices to consider points in $\Gamma_{u}$.) Using the martingale convergence theorem, we prove that the distribution of random vector $(\xi_{N}(z_{1}),\cdots,\xi_{N}(z_{p}))$ converges weakly to the $p$-dimensional centered Gaussian of covariance matrix given in {Proposition \ref{prop:Gproc}}.

For distinct points $z_{1},z_{2}\in\Gamma_{u}$, we let
\beq
\Gamma_{N}(z_{1},z_{2})=\sum_{k=1}^{N}\expctk{\expctk{\phi_{k}(z_{1})}{k-1}\cdot\expctk{\phi_{k}(z_{2})}{k-1}}{k}
\eeq
and
\begin{multline}\label{eq:covantider}
\wt{\Gamma}_{N}(z_{1},z_{2}) =\sum_{k=1}^{N}\wh{g}_{k}(z_{1})\wh{g}_{k}(z_{2})\E_{k}\bigg[\biggexpctk{-W_{kk}+v_{k}+\sum_{p,q}^{(k)}W_{kp}R^{(k)}_{pq}(z_{1})W_{qk}-\wh{m}_{fc}(z_{1})}{k-1} \\
\cdot\biggexpctk{-W_{kk}+v_{k}+\sum_{p,q}^{(k)}W_{kp}R_{pq}^{(k)}(z_{2})W_{qk}-\wh{m}_{fc}(z_{2})}{k-1} \Big]
\end{multline}
so that
\beq
\Gamma_{N}(z_{1},z_{2})=\frac{\partial^{2}}{\partial z_{1} \partial z_{2}}\wt{\Gamma}_{N}(z_{1},z_{2}).
\eeq

\subsection{Reduction of $\wt{\Gamma}_{N}$}\label{sec:reduccov}
For simplicity, we define
\beq
S_{k}(z) \deq \sum_{p,q}^{(k)}W_{kp}R_{pq}^{(k)}(z)W_{qk}\quad\text{ and }\quad T_{k}:=-W_{kk}+v_{k}.
\eeq
Then each summand of \eqref{eq:covantider} is 
\beq
\wh{g}_{k}(z_{1})\wh{g}_{k}(z_{2})\expctk{\expctk{T_{k}+S_{k}(z_{1})-\wh{m}_{fc}(z_{1})}{k-1}\cdot\expctk{T_{k}+S_{k}(z_{2})-\wh{m}_{fc}(z_{2})}}{k}.
\eeq

Using {Lemma \ref{lem:lde}}, we get
\beq
\absv{S_{k}(z)-\wh{m}_{fc}(z)} \leq\biggabsv{\sum_{p,q}^{(k)}W_{kp}R^{(k)}_{pq}(z)W_{qk}-\frac{1}{N}\sum_{p}^{(k)}R^{(k)}_{pp}(z)}+\absv{\wh{m}_{fc}(z)-m_{N}^{(k)}(z)}\prec N^{-\frac{1}{2}}.
\eeq
On the other hand, we have
\begin{align}
& \expctk{\expctk{T_{k}}{k-1}\expctk{T_{k}}{k-1}}{k}=\frac{1}{N}\expctk{A_{kk}^{2}}{k}=\frac{w_{2}}{N}, \\
& \expctk{\expctk{T_{k}}{k-1}\expctk{S_{k}(z)-\wh{m}_{fc}(z)}{k-1}}{k} =\expctk{T_{k}(S_{k}(z)-\wh{m}_{fc}(z))}{k} =0, \\
& \expctk{\expctk{S_{k}(z_{1})}{k-1}\expctk{\wh{m}_{fc}(z_{2})}{k-1}}{k} =\wh{m}_{fc}(z_{2})\expctk{S_{k}(z_{1})}{k} =\wh{m}_{fc}(z_{2})\expctk{m_{N}^{(k)}(z_{1})}{k}.
\end{align}
Hence
\begin{multline}
\expctk{\expctk{T_{k}+S_{k}(z_{1})-\wh{m}_{fc}(z_{1})}{k-1} \expctk{T_{k}+S_{k}(z_{2})-\wh{m}_{fc}(z_{2})}{k-1}}{k} \\ =\frac{w_{2}}{N} +\expctk{\expctk{S_{k}(z_{1})}{k-1}\expctk{S_{k}(z_{2})}{k-1}}{k} -\wh{m}_{fc}(z_{1})\expctk{m_{N}^{(k)}(z_{2})}{k}\\
 -\wh{m}_{fc}(z_{2})\expctk{m_{N}^{(k)}(z_{1})}{k} +\wh{m}_{fc}(z_{1})\wh{m}_{fc}(z_{2}) \\
=\frac{w_{2}}{N} +\expctk{\expctk{S_{k}(z_{1})}{k-1}\expctk{S_{k}(z_{2})}{k-1}}{k} -\expctk{\expctk{m_{N}^{(k)}(z_{1})}{k-1}\expctk{m_{N}^{(k)}(z_{2})}{k-1}}{k} \\ +\expctk{\expctk{m_{N}^{(k)}(z_{1})-\wh{m}_{fc}(z_{1})}{k-1}\expctk{m_{N}^{(k)}(z_{2})-\wh{m}_{fc}(z_{2})}{k-1}}{k}.
\end{multline}

\subsubsection{$\expctk{\expctk{S_{k}(z_{1})}{k-1}\expctk{S_{k}(z_{2})}{k-1}}{k}$}
By the definition of $S_{k}$, we get
\begin{equation}
\expctk{\expctk{S_{k}(z_{1})}{k-1}\expctk{S_{k}(z_{2})}{k-1}}{k} \\
=\sum_{p,q,r,t}^{(k)}\expctk{\expctk{W_{kp}R^{(k)}_{pq}(z_{1})W_{qk}}{k-1} \expctk{W_{kr}R^{(k)}_{rt}(z_{2})W_{tk}}{k-1}}{k}.
\end{equation}
Note that
\beq
\expctk{W_{kp}R^{(k)}_{pq}W_{qk}}{k-1}
=\left\{
\begin{array}{lc}
	W_{kp}W_{qk}\expctk{R^{(k)}_{pq}}{k-1} &\text{ if }p,q>k, \\
	W_{kp}\expct{W_{qk}}\expctk{R^{(k)}_{pq}}{k-1}=0 & \text{ if }p>k,q<k, \\
	\expct{W_{kp}}W_{qk}\expctk{R^{(k)}_{pq}}{k-1}=0 & \text{ if }p<k, q>k, \\
	\expct{W_{kp}W_{qk}}\expctk{R^{(k)}_{pq}}{k-1} & \text{ if }p,q<k,
\end{array}
\right.
\eeq
Thus
\begin{multline}
\expctk{\expctk{W_{kp}R^{(k)}_{pq}(z_{1})W_{qk}}{k-1} \expctk{W_{kr}R^{(k)}_{rt}(z_{2})W_{tk}}{k-1}}{k} \\
=\left\{
\begin{array}{ll}
\expct{W_{kp}W_{qk}W_{kr}W_{tk}}\expctk{\expctk{R^{(k)}_{pq}(z_{1})}{k-1}\expctk{R^{(k)}_{rt}(z_{2})}{k-1}}{k} & \text{ if }p,q,r,t>k, \\
0 & \begin{array}{c}\text{ if }(p-k)(q-k)<0 \\ \text{ or }(r-k)(t-k)<0,\end{array} \\
\expct{W_{kp}W_{qk}}\expct{W_{kr}W_{tk}}\expctk{\expctk{R^{(k)}_{pq}(z_{1})}{k-1}\expctk{R^{(k)}_{rt}(z_{2})}{k-1}}{k} & \text{otherwise}.
\end{array}
\right.
\end{multline}
Using the result above, we deduce that each summand vanishes if there exists an index repeated only once.
Hence, we divide the sum into following:
\begin{enumerate}[(i)]
	\item If $p=q=r=t$, 
	\begin{multline}
	\sum_{p}^{(k)}\expctk{\expctk{W_{kp}R^{(k)}_{pp}(z_{1})W_{pk}}{k-1} \expctk{W_{kp}R^{(k)}_{pp}(z_{2})W_{pk}}{k-1}}{k} \\
	=\sum_{p>k}\frac{W_{4}}{N^{2}}\expctk{\expctk{R_{pp}^{(k)}(z_{1})}{k-1}\expctk{R_{pp}^{(k)}(z_{2})}{k-1}}{k}
	+\sum_{p<k}\frac{1}{N^{2}}\expctk{\expctk{R_{pp}^{(k)}(z_{1})}{k-1}\expctk{R_{pp}^{(k)}(z_{2})}{k-1}}{k} \\
	=\sum_{p>k}\frac{W_{4}}{N^{2}}\wh{g}_{p}(z_{1})\wh{g}_{p}(z_{2})+\sum_{p<k}\frac{1}{N^{2}}\wh{g}_{p}(z_{1})\wh{g}_{p}(z_{2})+\caO(N^{-\frac{3}{2}}).
	\end{multline}
	
	\item If $p=q\neq r=t$,
	\begin{multline}
	\frac{1}{N^{2}}\sum_{p\neq r}^{(k)}\expctk{\expctk{R^{(k)}_{pp}(z_{1})}{k-1}\expctk{R^{(k)}_{rr}(z_{2})}{k-1}}{k} \\
	=\expctk{\expctk{m_{N}^{(k)}(z_{1})}{k-1}\expctk{m_{N}^{(k)}(z_{2})}{k-1}}{k}- 
	\frac{1}{N^{2}}\sum_{p}\expctk{\expctk{R_{pp}^{(k)}(z_{1})}{k-1}\expctk{R_{pp}^{(k)}(z_{2})}{k-1}}{k} \\
	=\expctk{\expctk{m_{N}^{(k)}(z_{1})}{k-1}\expctk{m_{N}^{(k)}(z_{2})}{k-1}}{k}-\frac{1}{N^{2}}\sum_{p}\wh{g}_{p}(z_{1})\wh{g}_{p}(z_{2})+\caO(N^{-2}).
	\end{multline}
	
	\item If $p=t\neq q=r$,
	\begin{multline}
	\sum_{p\neq q}^{(k)} \expctk{\expctk{W_{kp}R_{pq}^{(k)}(z_{1})W_{qk}}{k-1} \expctk{W_{kq}R_{qp}^{(k)}(z_{2})W_{pk}}{k-1}}{k} \\
	=\frac{1}{N^{2}}\sum_{ p\neq q, p, q>k} \expctk{\expctk{R_{pq}^{(k)}(z_{1})}{k-1} \expctk{R_{qp}^{(k)}(z_{2})}{k-1}}{k} =: Z_{k}.
	\end{multline}
	We again use \eqref{eq:mat id 3} to expand $Z_{k}$ as
	\begin{equation}
	\frac{1}{N^{2}}\sum_{p\neq q, p, q>k}\E_{k}\bigg[\biggexpctk{R_{pp}^{(k)}(z_{1})\sum_{a}^{(k,p)}W_{pa}R_{aq}^{(k,p)}(z_{1})}{k-1}
	\cdot\biggexpctk{R_{pp}^{(k)}(z_{2})\sum_{b}^{(k,p)}R_{qb}^{(k,p)}(z_{2})W_{bp}}{k-1}\bigg].
	\end{equation}
	As $R_{pp}^{(k)}(z)=\wh{g}_{p}(z)+\caO(N^{-1/2})$ and $R^{(k)}_{pq}(z)=\caO(N^{-1/2})$, we get
	\begin{multline}
	Z_{k} =\frac{1}{N^{2}}\sum_{p\neq q, p, q>k}\Big( \wh{g}_{p}(z_{1})\wh{g}_{p}(z_{2}) \\
	\sum_{a,b}^{(k,p)}\expctk{\expctk{W_{pa}R_{aq}^{(k,p)}(z_{1})}{k-1}\expctk{R_{qb}^{(k,p)}(z_{2})W_{bp}}{k-1}}{k}+\caO(N^{-\frac{3}{2}})\Big) \\
	=\frac{1}{N^{3}}\sum_{p\neq q, p, q>k}\wh{g}_{p}(z_{1})\wh{g}_{p}(z_{2})\sum_{a>k}^{(p)}\expctk{\expctk{R^{(k,p)}_{aq}(z_{1})}{k-1}\expctk{R^{(k,p)}_{qa}(z_{2})}{k-1}}{k} +\caO(N^{-\frac{3}{2}}) \\
	=\frac{1}{N^{3}}\sum_{p\neq q, p, q>k}\wh{g}_{p}(z_{1})\wh{g}_{p}(z_{2})\sum_{a>k}\expctk{\expctk{R^{(k)}_{aq}(z_{1})}{k-1}\expctk{R^{(k)}_{qa}(z_{2})}{k-1}}{k} +\caO(N^{-\frac{3}{2}}) \\
	=\frac{1}{N^{3}}\sum_{q,a>k}\expctk{\expctk{R^{(k)}_{aq}(z_{1})}{k-1}\expctk{R^{(k)}_{qa}(z_{2})}{k-1}}{k}\Big(\sum_{p>k}^{(q)}\wh{g}_{p}(z_{1})\wh{g}_{p}(z_{2})\Big) +\caO(N^{-\frac{3}{2}}) \\
	=\frac{1}{N^{3}}\Big(\sum_{p>k}\wh{g}_{p}(z_{1})\wh{g}_{p}(z_{2})\Big)\sum_{q,a>k}\expctk{\expctk{R^{(k)}_{aq}(z_{1})}{k-1}\expctk{R^{(k)}_{qa}(z_{2})}{k-1}}{k} +\caO(N^{-\frac{3}{2}}) \\
	=\frac{1}{N}\Big(\sum_{p>k}\wh{g}_{p}(z_{1})\wh{g}_{p}(z_{2})\Big)\Big(Z_{k}+\frac{1}{N^{2}}\sum_{q>k}\wh{g}_{q}(z_{1})\wh{g}_{q}(z_{2})\Big)+\caO(N^{-\frac{3}{2}}).
	\end{multline}
	Therefore by denoting $\frac{1}{N}\sum_{p>k}\wh{g}_{p}(z_{1})\wh{g}_{p}(z_{2})=\wh{I}_{k}$, we get
	\beq
	(1-\wh{I}_{k})Z_{k} =\frac{1}{N}\wh{I}_{k}^{2}+\caO(N^{-\frac{3}{2}}). 
	\eeq
	
	Also {Lemma \ref{cor:path}} below implies $(1-\wh{I}_{k})^{-1}=O(1)$, hence
	\beq
	Z_{k}=\frac{1}{N}\frac{\wh{I}_{k}^{2}}{1-\wh{I}_{k}}+\caO(N^{-\frac{3}{2}}).
	\eeq
	
	\item\label{item:compvavar} If $p=r\neq q=t$, the symmetry gives 
	\beq
	\sum_{p\neq q}^{(k)}\expctk{\expctk{W_{kp}R^{(k)}_{pq}(z_{1})W_{qk}}{k-1}\expctk{W_{kp}R^{(k)}_{pq}(z_{2})W_{qk}}{k-1}}{k} =Z_{k}.
	\eeq
	\begin{remark}
		If $W$ is complex, then \eqref{item:compvavar} vanishes.
	\end{remark}
\end{enumerate}
Adding up, we get
\begin{multline}
\expctk{\expctk{T_{k}+S_{k}(z_{1})-\wh{m}_{fc}(z_{1})}{k-1} \expctk{T_{k}+S_{k}(z_{2})-\wh{m}_{fc}(z_{2})}{k-1}}{k} \\
=\frac{w_{2}}{N}+\frac{W_{4}}{N^{2}}\sum_{p>k}\wh{g}_{p}(z_{1})\wh{g}_{p}(z_{2})+\frac{1}{N^{2}}\sum_{p<k}\wh{g}_{p}(z_{1})\wh{g}_{p}(z_{2}) \\
=\frac{w_{2}}{N} +\frac{W_{4}}{N}\wh{I}_{k} +\frac{1}{N}(\wh{I}_{0}-\wh{I}_{k}) -\frac{1}{N}\wh{I}_{0} +\frac{2}{N}\frac{\wh{I}_{k}^{2}}{1-\wh{I}_{k}} +\caO(N^{-\frac{3}{2}}) \\
=\frac{w_{2}-2}{N} +\frac{W_{4}-3}{N}\wh{I}_{k} +\frac{2}{N}\frac{1}{1-\wh{I}_{k}}+\caO(N^{-\frac{3}{2}}).
\end{multline}

\subsection{Conclusion for $\Gamma(z_{1},z_{2})$}
Now we retrieve the dependence on $\vartheta$ and fully analyze the covariance. Summing over $k$, we get
\begin{multline}
\wt{\Gamma}_{N}^{\vartheta}(z_{1},z_{2}) =(w_{2}-2)\frac{1}{N}\sum_{k}\wh{g}^{\vartheta}_{k}(z_{1})\wh{g}_{k}^{\vartheta}(z_{2}) +(W_{4}-3)\frac{1}{N}\sum_{k=1}^{N}\wh{I}_{k}^{\vartheta}(z_{1},z_{2})(\wh{g}^{\vartheta}_{k}(z_{1})\wh{g}^{\vartheta}_{k}(z_{2})) \\
+\frac{2}{N}\sum_{k=1}^{N}\frac{1}{1-I_{k}^{\vartheta}(z_{1},z_{2})}\wh{g}^{\vartheta}_{k}(z_{1})\wh{g}^{\vartheta}_{k}(z_{2}) +\caO(N^{-\frac{1}{2}}).
\end{multline}
For simplicity, we denote $f_{i}(x)=(x-z_{i}-m_{fc}(z_{i}))^{-1}$ and $\wh{f}_{i}(x)=(x-z_{i}-\wh{m}_{fc}(z_{i}))^{-1}$ for $x\in\R$.

\subsubsection{The first term}
The convergence of the first term follows directly from {Lemma \ref{cor:covwh}}:
\beq
(w_{2}-2)\frac{1}{N}\sum_{k}\wh{g}_{k}^{\vartheta}(z_{1})\wh{g}_{k}^{\vartheta}(z_{2}) =(w_{2}-2)\wh{I}_{0}^{\vartheta}(z_{1},z_{2}) =(w_{2}-2)I^{\vartheta}(z_{1},z_{2}) +O(\vartheta N^{-\alpha_{0}}).
\eeq

\subsubsection{The second term}
By the definition of $\wh{I}_{k}$, again by {Lemma \ref{cor:covwh} } we get
\begin{multline}
\frac{1}{N}\sum_{k=1}^{N} \wh{I}_{k}^{\vartheta}(z_{1},z_{2})\wh{g}_{k}^{\vartheta}(z_{1})\wh{g}_{k}^{\vartheta}(z_{2}) =\frac{1}{N^{2}}\sum_{p>k}\wh{g}_{k}^{\vartheta}(z_{1})\wh{g}_{k}^{\vartheta}(z_{2})\wh{g}_{p}^{\vartheta}(z_{1})\wh{g}_{p}^{\vartheta}(z_{2}) \\
=\frac{1}{2N^2}\sum_{p,k}\wh{g}_{k}^{\vartheta}(z_{1})\wh{g}_{k}^{\vartheta}(z_{2})\wh{g}_{p}^{\vartheta}(z_{1})\wh{g}_{p}^{\vartheta}(z_{2}) -\frac{1}{2N^2}\sum_{m}(\wh{g}_{m}^{\vartheta}(z_{1})\wh{g}_{m}^{\vartheta}(z_{2}))^{2} \\
=\frac{1}{2N^2}\Big(\sum_{k}\wh{g}_{k}^{\vartheta}(z_{1})\wh{g}_{k}^{\vartheta}(z_{2})\Big)^{2}+\caO(N^{-1}) =\frac{1}{2}I^{\vartheta}(z_{1},z_{2})+O(N^{-1}+\vartheta N^{-\alpha_{0}}).
\end{multline}

\subsubsection{The third term}

To consider the third term, we define a polygonal path $C_{N}=\bigcup_{k}C_{k,N}$ connecting $\wh{I}^{\vartheta}_{N},\wh{I}^{\vartheta}_{N-1},\cdots,\wh{I}^{\vartheta}_{0}$ by
\beq
C_{k,N}:[0,1]\to\C,\quad C_{k,N}(t)= \wh{I}^{\vartheta}_{k} +\frac{t}{N}\wh{g}^{\vartheta}_{k}(z_{1})\wh{g}^{\vartheta}_{k}(z_{2}).
\eeq

Then letting $F(z)=(1-z)^{-1}$, we get
\begin{multline}
\frac{1}{N}\frac{1}{1- \wh{I}_{k}^{\vartheta}}\wh{g}_{k}^{\vartheta}(z_{1})\wh{g}_{k}^{\vartheta}(z_{2}) =\int_{0}^{1}F(C_{k,N}(0))\cdot C_{k,N}'(t)dt \\
= \int_{C_{N}}F(z)dz +\int_{0}^{1}(F(C_{k,N})(0)-F(C_{k,N}(t)))C_{k,N}(t)'dt.
\end{multline}
By {Lemma \ref{cor:path} } we have $\absv{F(C_{k,N}(0))-F(C_{k,N}(t))}=O(N^{-1})$, so that
\beq
\int_{0}^{1}(F(C_{n,k}(t))-F(C_{n,k}(0)))C_{k,N}(t)dt =O(N^{-2}),
\eeq
hence again by {Lemma \ref{cor:path}}
\begin{multline}
\frac{1}{N}\sum_{k}\frac{1}{1-\wh{I}_{k}^{\vartheta}}\wh{g}_{k}^{\vartheta}(z_{1})\wh{g}_{k}^{\vartheta}(z_{2}) =\int_{C_{N}} F(z)dz +O(N^{-1}) = -\log(1-\wh{I}^{\vartheta}_{0}(z_{1},z_{2})) +O(N^{-1}) \\ =-\log(1-I^{\vartheta}(z_{1},z_{2})) +O(N^{-1}+\vartheta N^{-\alpha_{0}}).
\end{multline}

\subsubsection{Summary}
In summary, we have
\begin{multline}\label{eq:covsum}
\wt{\Gamma}_{N}(z_{1},z_{2}) =(w_{2}-2)I^{\vartheta}(z_{1},z_{2}) +\frac{1}{2}(W_{4}-3)I^{\vartheta}(z_{1},z_{2})^{2} \\
-2\log(1-I^{\vartheta}(z_{1},z_{2})) +O(N^{-1}+\vartheta N^{-\alpha_{0}})+\caO(N^{-\frac{1}{2}}).
\end{multline}

\subsection{Proof of Lemma \ref{lem:martclt11}}

\subsubsection{$\vartheta_{\infty}>0$}

For $\vartheta_{\infty}>0$, recalling {Corollaries \ref{cor:path}} and {\ref{cor:I-s1s2}}, we have
\beq
\wt{\Gamma}_{N}(z_{1},z_{2})=\wt{\Gamma}^{\vartheta_{\infty}}+O(N^{-1}+\vartheta N^{-\alpha_{0}}+\absv{\vartheta-\vartheta_{\infty}})+\caO(N^{-\frac{1}{2}}),
\eeq
where 
\beq
\wt{\Gamma}^{\vartheta_{\infty}}(z_{1},z_{2}) = (w_{2}-2)I^{\vartheta_{\infty}}(z_{1},z_{2}) +\frac{1}{2}(W_{4}-3)I^{\vartheta_{\infty}}(z_{1},z_{2})^{2} -2\log(1-I^{\vartheta_{\infty}}(z_{1},z_{2})).
\eeq

Then by differentiating, again the Cauchy integral formula implies the in probability convergence of $\Gamma_{N}$ to $\Gamma$ defined in {Proposition \ref{prop:Gproc}}.

\subsubsection{$\vartheta=o(1)$}
If $\vartheta=o(1)$, similarly we reduce \eqref{eq:covsum} further, using {Corollaries \ref{cor:path}} and {\ref{cor:I-s1s2}}, to
\begin{multline}
\wt{\Gamma}_{N}(z_{1},z_{2}) =(w_{2}-2)I^{\vartheta}(z_{1},z_{2}) +\frac{1}{2}(W_{4}-3)I^{\vartheta}(z_{1},z_{2})^{2} \\
-2\log(1-I^{\vartheta}(z_{1},z_{2})) +O(N^{-1}+\vartheta N^{-\alpha_{0}})+\caO(N^{-\frac{1}{2}}) \\
=(w_{2}-2)m_{sc}(z_{1})m_{sc}(z_{2})+\frac{1}{2}(W_{4}-3)(m_{sc}(z_{1})m_{sc}(z_{2}))^{2} \\
-2\log(1-m_{sc}(z_{1})m_{sc}(z_{2}))+O(\vartheta)+\caO(N^{-\frac{1}{2}}).
\end{multline}
Differentiating, we get
\begin{multline}
\Gamma_{N}(z_{1},z_{2})=m_{sc}'(z_{1})m_{sc}'(z_{2})\bigg((w_{2}-2)+2(W_{4}-3)m_{sc}(z_{1})m_{sc}(z_{2}) \\
+\frac{2}{(1-m_{sc}(z_{1})m_{sc}(z_{2}))^{2}}\bigg)+O(\vartheta)+\caO(N^{-\frac{1}{2}}),
\end{multline}
which converges to $\Gamma(z_{1},z_{2})$ in probability.
\begin{remark}
	The result above and also  coincides precisely with $\Gamma(z_{1},z_{2})$ in Proposition 4.1 of \cite{Bai-Yao2005}.
\end{remark}

\subsection{Proof of Lemma \ref{lem:martclt22}}
Now  it suffices to prove
\beq
\sum_{k}\expct{\absv{\expctk{\phi_{k}}{k-1}}^{2}\lone_{[\absv{\expctk{\phi_{k}}{k-1}}\geq\epsilon]}}\to 0
\eeq
for any fixed $\epsilon>0$. For
\beq
\expct{\absv{\expctk{\phi_{k}}{k-1}}\lone_{[\absv{\expctk{\phi_{k}}{k-1}}\geq\epsilon]}}\leq\epsilon^{-2}\expct{\absv{\expctk{\phi_{k}}{k-1}}^{4}},
\eeq
the bound $\phi_{k}\prec N^{-1/2}$ gives 
\beq
\sum_{k}\expct{\absv{\expctk{\phi_{k}}{k-1}}^{2}\lone_{[\absv{\expctk{\phi_{k}}{k-1}}\geq\epsilon]}} \leq\epsilon^{-2}\sum_{k}\expct{\absv{\expctk{\phi_{k}}{k-1}}^{4}} =O(N^{-1+\epsilon'})\to0.
\eeq

\section{Tightness of $\xi_{N}$}\label{sec:tight}

As mentioned in {Section \ref{sec:Gaussianprocprf}}, this section provides the proofs of {Lemmas \ref{lem:tightproc}} and {\ref{lem:tightprocsqrt}}. As in the sections above, we omit the superscript $\vartheta$ to prevent unnecessary complication. Recalling the aim, in order to prove the H\"{o}lder conditions, we prove the existence of a constant $K>0$ such that
\beq\label{eq:C1}
\expct{\absv{\Tr R(z_{1})R(z_{2})-\expct{\Tr R(z_{1})R(z_{2})}}^{2}}\leq K,\quad^{\forall}z_{1},z_{2}\in\caK
\eeq
for deterministic $V$ and
\beq\label{eq:C2}
\frac{1}{N\vartheta^{2}}\expct{\absv{\Tr R_{N}(z_{1})R_{N}(z_{2})-\expct{\Tr R_{N}(z_{1})R_{N}(z_{2})}}^{2}\lone_{\Omega}}\leq K,\quad^{\forall}z_{1},z_{2}\in\caK
\eeq
for random $V$.

\subsection{Proof of Lemma \ref{lem:tightproc}}\label{sec:C11}

To prove the existence of the constant $K$ in \eqref{eq:C1}, we again use the $\sigma$-algebra $\caF_{k}$ introduced in {Section \ref{sec:Gaussianprocprf}}:
\begin{multline}
\expct{\absv{\Tr R(z_{1})R(z_{2})-\expct{\Tr R(z_{1})R(z_{2})}}^{2}} \\
=\biggexpct{\biggabsv{\sum_{k=1}^{N}(\E_{k-1}-\E_{k})\big(\Tr R(z_{z})R(z_{2})-\Tr R^{(k)}(z_{1})R^{(k)}(z_{2})\big)}^{2}}.
\end{multline}

In the following, we denote
\beq
R\deq R(z_{1}),\quad S\deq R(z_{2})\quad\text{and}\quad R^{(k)}\deq R^{(k)}(z_{1}),\quad S^{(k)}\deq R^{(k)}(z_{2})
\eeq
for simplicity. Then one can observe that
\beq
\norm{R}\leq \frac{1}{\eta}\leq C
\eeq
for $C\geq c^{-1}$ and similarly \beq
\norm{S},\norm{R^{(k)}},\norm{S^{(k)}}\leq C.
\eeq
Now for $i,j\neq k$, \eqref{eq:mat id 3} gives
\begin{multline}
R_{ij}S_{ji}-R^{(k)}_{ij}S^{(k)}_{ji} =(R_{ij}-R^{(k)}_{ij})S^{(k)}_{ji} +R^{(k)}_{ij}(S_{ji}-S^{(k)}_{ji}) +(R_{ij}-R^{(k)}_{ij})(S_{ji}-S^{(k)}_{ji}) \\
=\frac{R_{ik}R_{kj}}{R_{kk}}S^{(k)}_{ji} +R_{ij}^{(k)}\frac{S_{jk}S_{ki}}{S_{kk}} +\frac{R_{ik}R_{kj}}{R_{kk}}\frac{S_{jk}S_{ki}}{S_{kk}}.
\end{multline}

Hence we get
\begin{multline}
\Tr RS-\Tr R^{(k)}S^{(k)} =\sum_{i,j}^{(k)}\big(R_{ij}S_{ji}-R^{(k)}_{ij}S^{(k)}_{ji}\big) +\sum_{j}^{(k)}R_{kj}S_{jk} +\sum_{i}^{(k)}R_{ik}S_{ki} +R_{kk}S_{kk} \\
=\sum_{i,j}^{(k)}\Big(\frac{R_{ik}R_{kj}}{R_{kk}}S^{(k)}_{ji}+R^{(k)}_{ij}\frac{S_{jk}S_{ki}}{S_{kk}}+\frac{R_{ik}R_{kj}}{R_{kk}}\frac{S_{jk}S_{ki}}{S_{kk}}\Big) +2\sum_{i}^{(k)}R_{ki}S_{ik} +R_{kk}S_{kk},
\end{multline}
which together with \eqref{eq:orthogonal} implies
\begin{multline}
\expct{\absv{\Tr R(z_{1})R(z_{2})-\expct{R(z_{1})R(z_{2})}}^{2}}\\
\leq C\biggexpct{\sum_{k=1}^{N}\biggabsv{(\E_{k-1}-\E_{k})\sum_{i,j}^{(k)}\Big(\frac{R_{ik}R_{kj}}{R_{kk}}S^{(k)}_{ji}+R^{(k)}_{ij}\frac{S_{jk}S_{ki}}{S_{kk}}+\frac{R_{ik}R_{kj}}{R_{kk}}\frac{S_{jk}S_{ki}}{S_{kk}}\Big)}^{2}} \\
+C\biggexpct{\sum_{k=1}^{N}\absv{(\E_{k-1}-\E_{k})2(RS)_{kk} -R_{kk}S_{kk}}^{2}}
\end{multline}
\subsubsection{The first and the second term}

We first rewrite
\begin{multline}
\sum_{i,j}^{(k)}\frac{R_{ik}R_{kj}}{R_{kk}}S_{ji}^{(k)} =\sum_{i,j}^{(k)}\frac{1}{R_{kk}}\Big(R_{kk}\sum_{p}^{(k)}R_{ip}^{(k)}W_{pk}\Big)\Big(R_{kk}\sum_{q}^{(k)}W_{kq}R_{qj}^{(k)}\Big)S_{ji}^{(k)} \\
=\sum_{i,j,p,q}^{(k)}R_{kk}W_{kq}R_{qj}^{(k)}S_{ji}^{(k)}R_{ip}^{(k)}W_{pk} =R_{kk}\sum_{p,q}^{(k)}W_{kq}(R^{(k)}S^{(k)}R^{(k)})_{qp}W_{pk}.
\end{multline}

Noting that 
\beq
(\E_{k-1}-\E_{k})\bigg[\frac{\wh{g}_{k}(z_{1})}{N}\sum_{p}^{(k)}(R^{(k)}S^{(k)}R^{(k)})_{pp}\bigg]=0,
\eeq
we get
\begin{multline}\label{eq:above1}
\biggexpct{\sum_{k=1}^{N}\biggabsv{(\E_{k-1}-\E_{k})\sum_{i,j}^{(k)}\frac{R_{ik}R_{kj}}{R_{kk}}S_{ji}^{(k)}}^{2}} \\
=\E\Big[\sum_{k=1}^{N}\Big\vert(\E_{k-1}-\E_{k})\Big[R_{kk}\sum_{p,q}^{(k)}W_{kq}(R^{(k)}S^{(k)}R^{(k)})_{qp}W_{pk}
-\frac{\wh{g}_{k}(z_{1})}{N}\sum_{p}^{(k)}(R^{(k)}S^{(k)}R^{(k)})_{pp}\Big]\Big\vert^{2}\Big].
\end{multline}

From $\absv{a+b}^{2} \leq 2(\absv{a}^{2}+\absv{b}^{2})$, \eqref{eq:above1} is bounded by 
\begin{multline}
2\sum_{k=1}^{N}\biggexpct{\biggabsv{\biggexpctk{R_{kk}\sum_{p,q}^{(k)}W_{kq}(R^{(k)}S^{(k)}R^{(k)})_{qp}W_{pk}-\frac{\wh{g}_{k}(z_{1})}{N}\sum_{p}^{(k)}(R^{(k)}S^{(k)}R^{(k)})_{pp}}{k-1}}^{2}} \\
+2\sum_{k=1}^{N}\biggexpct{\biggabsv{\biggexpctk{R_{kk}\sum_{p,q}^{(k)}W_{kq}(R^{(k)}S^{(k)}R^{(k)})_{qp}W_{pk}-\frac{\wh{g}_{k}(z_{1})}{N}\sum_{p}^{(k)}(R^{(k)}S^{(k)}R^{(k)})_{pp}}{k}}^{2}} \\
\leq 4\sum_{k=1}^{N}\biggexpct{\biggabsv{R_{kk}\sum_{p,q}^{(k)}W_{kq}(R^{(k)}S^{(k)}R^{(k)})_{qp}W_{pk}-\frac{\wh{g}_{k}(z_{1})}{N}\sum_{p}^{(k)}(R^{(k)}S^{(k)}R^{(k)})_{pp}}^{2}}
\end{multline}
where we have used the Jensen's inequality in the third line.

Since $\absv{R_{kk}},\norm{R^{(k)}},\norm{S^{(k)}}\leq C$, from {Lemma \ref{lem:lde}} we have
\begin{multline}
\biggexpct{\biggabsv{R_{kk}\sum_{p,q}^{(k)}W_{kq}(R^{(k)}S^{(k)}R^{(k)})_{qp}W_{pk}-\frac{R_{kk}}{N}\sum_{p}^{(k)}(R^{(k)}S^{(k)}R^{(k)})_{pp}}^{2}} \\
\leq C\biggexpct{\biggabsv{\sum_{p,q}^{(k)}W_{kq}(R^{(k)}S^{(k)}R^{(k)})_{qp}W_{pk}-\frac{1}{N}\sum_{p}^{(k)}(R^{(k)}S^{(k)}R^{(k)})_{pp}}^{2}}
\leq C\frac{\norm{R^{(k)}S^{(k)}R^{(k)}}^{2}}{N}\leq \frac{C}{N}.
\end{multline}

On the other hand, we also have $\tr (R^{(k)}S^{(k)}R^{(k)})\leq \norm{R^{(k)}S^{(k)}R^{(k)}}\leq C$, so that
\beq
\biggexpct{\biggabsv{\frac{R_{kk}}{N}\sum_{p}^{(k)}(R^{(k)}S^{(k)}R^{(k)})_{pp}-\frac{\wh{g}_{k}(z_{1})}{N}\sum_{p}^{(k)}(R^{(k)}S^{(k)}R^{(k)})_{pp}}^{2}} \leq C\expct{\absv{R_{kk}-\wh{g}_{k}(z_{1})}^{2}}.
\eeq
Using the expansion \eqref{eq:resodiagexpa}, we get
\begin{multline}
R_{kk}-\wh{g}_{k}(z_{1}) =\wh{g}_{k}(z_{1})^{2}(Q_{k}(z_{1})-\wh{m}_{fc}(z_{1})+v_{k}) +\caO(N^{-1}) \\
=\wh{g}_{k}(z_{1})^{2}(Q_{k}(z_{1})-m_{N}^{(k)}(z_{1})+v_{k}) +\caO(N^{-1}).
\end{multline}

Now using {Lemma \ref{lem:lde} } we get
\begin{multline}\label{eq:resolvconcent}
\expct{\absv{R_{kk}-\wh{g}_{k}(z_{1})}^{2}} \leq C\expct{\absv{Q_{k}(z_{1})-m_{N}^{(k)}+v_{k}}^{2}} 
=C\biggexpct{\biggabsv{-N^{-\frac{1}{2}}A_{ii}+\sum_{p,q}^{(k)}W_{kp}R^{(k)}_{pq}W_{qk}-\frac{1}{N}\sum_{p}^{(k)}R_{pp}^{(k)}}^{2}} \\
\leq 2C\bigg(\frac{1}{N}+\biggexpct{\biggabsv{\sum_{p,q}^{(k)}W_{kp}R^{(k)}_{pq}W_{qk}-\frac{1}{N}\sum_{p}^{(k)}R_{pp}^{(k)}}^{2}}\bigg) \leq \frac{C}{N},
\end{multline}
so that 
\beq
\biggexpct{\biggabsv{\frac{R_{kk}}{N}\sum_{p}^{(k)}(R^{(k)}S^{(k)}R^{(k)})_{pp}-\frac{\wh{g}_{k}(z_{1})}{N}\sum_{p}^{(k)}(R^{(k)}S^{(k)}R^{(k)})_{pp}}^{2}} \leq\frac{C}{N}.
\eeq
Altogether, we get a bound for the first term:
\beq\label{eq:first}
\biggexpct{\sum_{k=1}^{N}\biggabsv{(\E_{k-1}-\E_{k})\sum_{i,j}^{(k)}\frac{R_{ik}R_{kj}}{R_{kk}}S^{(k)}_{ji}}^{2}} \leq C.
\eeq

By symmetry, we have the same bound for the second term:
\beq\label{eq:second}
\biggexpct{\sum_{k=1}^{N}\biggabsv{(\E_{k-1}-\E_{k})\sum_{i,j}^{(k)}R^{(k)}_{ij}\frac{S_{jk}S_{ki}}{S_{kk}}}^{2}}\leq C.
\eeq

\subsubsection{The third term}
We first expand the third term using \eqref{eq:mat id 3}:
\begin{multline}
\sum_{i,j}^{(k)}\frac{R_{ik}R_{kj}}{R_{kk}}\frac{S_{jk}S_{ki}}{S_{kk}} =\sum_{i,j}^{(k)}R_{kk}S_{kk}\sum_{p,q,r,t}^{(k)}W_{pk}R_{ip}^{(k)}R_{qj}^{(k)}W_{kq}W_{rk}S_{jr}^{(k)}S_{ti}^{(k)}W_{kt} \\
=R_{kk}S_{kk}\sum_{p,q,r,t,i,j}^{(k)}\Big(W_{kt}(S_{ti}^{(k)}R_{ip}^{(k)})W_{pk}\Big)\Big(W_{kq}(R_{qj}^{(k)}S_{jr}^{(k)})W_{rk}\Big) \\
=R_{kk}S_{kk}\Big(\sum_{t,p}^{(k)}W_{kt}(S^{(k)}R^{(k)})_{tp}W_{pk}\Big)\Big(\sum_{q,r}^{(k)}W_{kq}(R^{(k)}S^{(k)})_{qr}W_{rk}\Big) \\
=R_{kk}S_{kk}\Big(\sum_{p,q}^{(k)}W_{kp}(R^{(k)}S^{(k)})_{pq}W_{qk}\Big)^{2},
\end{multline}
since $R^{(k)}$ and $S^{(k)}$ commutes. As above, we note that
\beq
(\E_{k-1}-\E_{k})\bigg[\wh{g}_{k}(z_{1})\wh{g}_{k}(z_{2})\Big(\frac{1}{N}\sum_{p}^{(k)}(R^{(k)}S^{(k)})_{pp}\Big)^{2}\bigg]=0,
\eeq
hence
\begin{multline}
\biggexpct{\biggabsv{(\E_{k-1}-\E_{k})\Big[\sum_{i,j}^{(k)}\frac{R_{ik}R_{kj}}{R_{kk}}\frac{S_{jk}S_{ki}}{S_{kk}}\Big]}^{2}} \\
=\E\Big[\Big\vert(\E_{k-1}-\E_{k})\Big[R_{kk}S_{kk}\Big(\sum_{p,q}^{(k)}W_{kp}(R^{(k)}S^{(k)})_{pq}W_{qk}\Big)^{2} 
-\wh{g}_{k}(z_{1})\wh{g}_{k}(z_{2})\Big(\frac{1}{N}\sum_{p}(R^{(k)}S^{(k)})_{pp}\Big)^{2}\Big]\Big\vert^{2}\Big].
\end{multline}

As above, $\absv{R_{kk}S_{kk}},\absv{\frac{1}{N}\sum_{p}^{(k)}(R^{(k)}S^{(k)})_{pp}}\leq \norm{R}\norm{S}\leq C$ gives
\begin{multline}
\biggexpct{\biggabsv{(\E_{k-1}-\E_{k})\Big[\sum_{i,j}^{(k)}\frac{R_{ik}R_{kj}}{R_{kk}}\frac{S_{jk}S_{ki}}{S_{kk}}\Big]}^{2}} \\
\leq C\biggexpct{\biggabsv{\Big(\sum_{p,q}^{(k)}W_{kp}(R^{(k)}S^{(k)})_{pq}W_{qk}\Big)^{2}-\Big(\frac{1}{N}\sum_{p}^{(k)}(R^{(k)}S^{(k)})_{pp}\Big)^{2}}^{2}} 
+C\expct{\absv{R_{kk}S_{kk}-\wh{g}_{k}(z_{1})\wh{g}_{k}(z_{2})}^{2}}.
\end{multline}
Using $A^{2}-B^{2}=(A-B)^{2}+2B(A-B)$ together with {Lemma \ref{lem:lde}}, the first term is bounded as
\begin{multline}
=\E\Big[\Big\vert\Big(\sum_{p,q}^{(k)}W_{kp}(R^{(k)}S^{(k)})_{pq}W_{qk}-\frac{1}{N}\sum_{p}^{(k)}(R^{(k)}S^{(k)})_{pp}\Big)^{2} \\
+\frac{2}{N}\sum_{p}^{(k)}(R^{(k)}S^{(k)})_{pp}\Big(\sum_{p,q}^{(k)}W_{kp}(R^{(k)}S^{(k)})_{pq}W_{qk}-\frac{1}{N}\sum_{p}^{(k)}(R^{(k)}S^{(k)})_{kk}\Big)\Big\vert^{2}\Big] \\
\leq C\Big(\biggexpct{\biggabsv{\sum_{p,q}^{(k)}W_{kp}(R^{(k)}S^{(k)})_{pq}W_{qk}-\frac{1}{N}\sum_{p}^{(k)}(R^{(k)}S^{(k)})_{pp}}^{4}} \\
+\biggexpct{\biggabsv{\sum_{p,q}^{(k)}W_{kp}(R^{(k)}S^{(k)})_{pq}W_{qk}-\frac{1}{N}\sum_{p}^{(k)}(R^{(k)}S^{(k)})_{pp}}^{2}}\Big) \leq\frac{C}{N}.
\end{multline}

On the other hand, by \eqref{eq:resolvconcent} we have
\begin{multline}
\expct{\absv{R_{kk}S_{kk}-\wh{g}_{k}(z_{1})\wh{g}_{k}(z_{2})}^{2}} =\expct{\absv{R_{kk}(S_{kk}-\wh{g}_{k}(z_{2}))+\wh{g}_{k}(z_{2})(R_{kk}-\wh{g}_{k}(z_{1}))}^{2}} \\
\leq C\big(\expct{\absv{S_{kk}-\wh{g}_{k}(z_{2})}^{2}}+\expct{\absv{R_{kk}-\wh{g}_{k}(z_{1})}^{2}}\big) \leq \frac{C}{N}.
\end{multline}

Altogether, we have a bound of the third term:
\beq\label{eq:third}
\biggexpct{\sum_{k=1}^{N}\biggabsv{(\E_{k-1}-\E_{k})\sum_{i,j}^{(k)}\frac{R_{ik}R_{kj}}{R_{kk}}\frac{S_{jk}S_{ki}}{S_{kk}}}^{2}} \leq C
\eeq

\subsubsection{Remainders}
We again expand the fourth term:
\begin{multline}
\sum_{i}^{(k)}R_{ki}S_{ik} =R_{kk}S_{kk}\sum_{i}^{(k)}\Big(\sum_{p}^{(k)}W_{kp}R_{pi}^{(k)}\Big) \Big(\sum_{q}^{(k)}S_{iq}^{(k)}W_{qk}\Big) \\
=R_{kk}S_{kk}\sum_{p,q}^{(k)}W_{kp}\Big(\sum_{i}R_{pi}^{(k)}S_{iq}^{(k)}\Big)W_{qk} =R_{kk}S_{kk}\sum_{p,q}^{(k)}W_{kp}(R^{(k)}S^{(k)})_{pq}W_{qk}.
\end{multline}
Using this expansion, we write
\begin{multline}\label{eq:fourth}
\Bigexpct{\Bigabsv{(\E_{k-1}-\E_{k})\Big[\sum_{i}^{(k)}R_{ki}S_{ik}\Big]}^{2}} \\
=\biggexpct{\biggabsv{(\E_{k-1}-\E_{k})\Big[R_{kk}S_{kk}\sum_{p,q}^{(k)}W_{kp}(R^{(k)}S^{(k)})_{pq}W_{qk} -\wh{g}_{k}(z_{1})\wh{g}_{k}(z_{2})\frac{1}{N}\sum_{p}^{(k)}(R^{(k)}S^{(k)})_{pp}\Big]}^{2}}.
\end{multline}
Then by the same argument as above, $\absv{R_{kk}S_{kk}},\absv{\frac{1}{N}\sum_{p}^{(k)}(R^{(k)}S^{(k)})_{pp}}\leq C$ implies that above is bounded by
\begin{multline} 2\biggexpct{\biggabsv{(\E_{k-1}-\E_{k})\Big[R_{kk}S_{kk}\Big(\sum_{p,q}^{(k)}W_{kp}(R^{(k)}S^{(k)})_{pq}W_{qk}-\frac{1}{N}\sum_{p}^{(k)}(R^{(k)}S^{(k)})_{pp}\Big)\Big]}^{2}} \\
+2\biggexpct{\biggabsv{(\E_{k-1}-\E_{k})\Big[(R_{kk}S_{kk}-\wh{g}_{k}(z_{1})\wh{g}_{k}(z_{2}))\frac{1}{N}\sum_{p}^{(k)}(R^{(k)}S^{(k)})_{pp}\Big]}^{2}}, \\
\leq C\biggexpct{\biggabsv{\sum_{p,q}^{(k)}W_{kp}(R^{(k)}S^{(k)})_{pq}W_{qk}-\frac{1}{N}\sum_{p}^{(k)}(R^{(k)}S^{(k)})_{pp}}^{2}} 
+C\expct{\absv{R_{kk}S_{kk}-\wh{g}_{k}(z_{1})\wh{g}_{k}(z_{2})}^{2}} \leq\frac{C}{N}.
\end{multline}

And similarly we have
\beq\label{eq:last}
\expct{\absv{(\E_{k-1}-\E_{k})[R_{kk}S_{kk}]}^{2}} =\expct{\absv{(\E_{k-1}-\E_{k})[R_{kk}S_{kk}-\wh{g}_{k}(z_{1})\wh{g}_{k}(z_{2})]}^{2}} \leq\frac{C}{N}.
\eeq

Using \eqref{eq:first}, \eqref{eq:second}, \eqref{eq:third}, \eqref{eq:fourth} and \eqref{eq:last}, we conclude that the H\"{o}lder condition holds, proving the tightness of $\{\xi_{N}(z):z\in\caK\}$.

\subsection{Proof of Lemma \ref{lem:tightprocsqrt}}\label{sec:tightprocsqrtprf}
The proof of \eqref{eq:C2} follows the same line as {Section \ref{sec:C11}}, except we replace the conditional expectation $\expctk{\cdot}{k}$ with
\beq
\wh{\E}_{k} [\,\cdot \,] := \cexpct{\cdot}{\caG_{k}},
\eeq
where $\caG_{k}$ is defined in {Definition \ref{def:filt}}. To be specific, noting that $R^{(k)}$ is independent of $\{v_{k},W_{k,i}:1\leq i\leq N\}$, we write 
\begin{multline}
\expct{\absv{\Tr R(z_{1})R(z_{2})-\expct{R(z_{1})R(z_{2})}}^{2}} \\
=\biggexpct{\biggabsv{\sum_{k=1}^{N}(\wh{\E}_{k-1}-\wh{\E}_{k})(\Tr R(z_{z})R(z_{2})-\Tr R^{(k)}(z_{1})R^{(k)}(z_{2}))}^{2}} \\
\leq C\biggexpct{\sum_{k=1}^{N}\biggabsv{(\wh{\E}_{k-1}-\wh{\E}_{k})\sum_{i,j}^{(k)}\Big(\frac{R_{ik}R_{kj}}{R_{kk}}S^{(k)}_{ji}+R^{(k)}_{ij}\frac{S_{jk}S_{ki}}{S_{kk}}+\frac{R_{ik}R_{kj}}{R_{kk}}\frac{S_{jk}S_{ki}}{S_{kk}}\Big)}^{2}} \\
+C\biggexpct{\sum_{k=1}^{N}\absv{2(RS)_{kk} -R_{kk}S_{kk}}^{2}}.
\end{multline}

To bound the first term, we replace the estimate
\beq
\expct{\absv{R_{kk}(z)-\wh{g}_{k}^{\vartheta}(z)}^{2}}\leq \frac{C}{N}
\eeq
in {Section \ref{sec:C11}} with
\beq\label{eq:localintlaw}
\biggexpct{\biggabsv{R_{kk}(z)+\frac{1}{z+m_{N}^{(k)}(z)}}^{2}}\leq C\vartheta^{2}.
\eeq
To prove the bound above, we first recall the Schur complement formula \eqref{eq:Schur}:
\beq
R_{kk}=\frac{1}{\vartheta v_{k}+\frac{1}{\sqrt{N}}A_{kk}-z-\sum_{p,q}^{(k)}W_{kp}R_{pq}^{(i)}W_{qk}},
\eeq
so that
\beq
R_{kk}+\frac{1}{z+m_{N}^{(k)}(z)} =\frac{R_{kk}}{z+m_{N}^{(k)}}\bigg(\Big(\vartheta v_{i}+\frac{1}{\sqrt{N}}A_{kk}\Big)+\Big(\frac{1}{N}\sum_{p}^{(k)}R_{pp}^{(k)}-\sum_{p,q}^{(k)}W_{kp}R_{pq}^{(k)}W_{qk}\Big)\bigg).
\eeq
On the other hand, $z\in\caK$ implies
\beq
\biggabsv{\frac{R_{kk}}{z+m_{N}^{(k)}(z)}}=\eta^{-2}=O(1),
\eeq
and by {Lemma \ref{lem:lde}} together with the assumption $\vartheta\gg N^{-1/2}$,
\beq
\biggexpct{\biggabsv{\vartheta v_{i}+\frac{1}{\sqrt{N}}A_{kk}+\Big(\frac{1}{N}\sum_{p}^{(k)}R_{pp}^{(k)}-\sum_{p,q}^{(k)}W_{kp}R_{pq}^{(k)}W_{qk}\Big)}^{2}} =O(\vartheta^{2}).
\eeq
By H\"{o}lder inequality, with two bounds above, we deduce \eqref{eq:localintlaw}.

After the replacement, using the fact that
\beq
(\wh{\E}_{k}-\wh{\E}_{k-1})\Big[\frac{1}{-z-m_{N}^{(k)}(z)}\cdot\frac{1}{N}\sum_{p}^{(k)}(R^{(k)}S^{(k)}R^{(k)})_{pp}\Big]=0,
\eeq
we obtain
\begin{multline}
\biggexpct{\sum_{k=1}^{N}\biggabsv{(\wh{\E}_{k-1}-\wh{\E}_{k})\sum_{i,j}^{(k)}\frac{R_{ik}R_{kj}}{R_{kk}}S_{ji}^{(k)}}^{2}} \\
=\E\bigg[\sum_{k=1}^{N}\Big\vert(\wh{\E}_{k-1}-\wh{\E}_{k})\Big[R_{kk}\sum_{p,q}^{(k)}W_{kq}(R^{(k)}S^{(k)}R^{(k)})_{qp}W_{pk}
-\Big(\frac{1}{-z-m_{N}^{(k)}(z)}\Big)\cdot\frac{1}{N}\sum_{p}^{(k)}(R^{(k)}S^{(k)}R^{(k)})_{pp}\Big]\Big\vert^{2}\bigg] \\
\leq 4\sum_{k=1}^{N}\Bigexpct{\Bigabsv{R_{kk}\sum_{p,q}^{(k)}W_{kq}(R^{(k)}S^{(k)}R^{(k)})_{qp}W_{pk}-\frac{1}{-z-m_{N}^{(k)}(z)}\cdot\frac{1}{N}\sum_{p}^{(k)}(R^{(k)}S^{(k)}R^{(k)})_{pp}}^{2}}
\leq CN\vartheta^{2},
\end{multline}
and a similar bound for the second term.

For the third term, we start by noting that
\beq
(\wh{\E}_{k}-\wh{\E}_{k-1})\bigg[\frac{1}{(-z_{2}-m_{N}^{(k)}(z_{2}))(-z_{2}-m_{N}^{(k)}(z_{2}))}\Big(\frac{1}{N}\sum_{p}^{(k)}(R^{(k)}S^{(k)})_{pp}\Big)^{2}\bigg]=0
\eeq
and use the estimate
\beq
\biggexpct{\biggabsv{R_{kk}S_{kk}-\frac{1}{(-z_{2}-m_{N}^{(k)}(z_{2}))(-z_{2}-m_{N}^{(k)}(z_{2}))}}^{2}}\leq C\vartheta^{2}
\eeq
following from \eqref{eq:localintlaw}, to conclude that
\beq
\biggexpct{\sum_{k=1}^{N}\biggabsv{(\wh{\E}_{k-1}-\wh{\E}_{k})\sum_{i,j}^{(k)}\frac{R_{ik}R_{kj}}{R_{kk}}\frac{S_{jk}S_{ki}}{S_{kk}}}^{2}} \leq CN\vartheta^{2}.
\eeq
The corresponding bounds for the other terms follows similarly, essentially by \eqref{eq:localintlaw}.

\section{Proof of Lemma \ref{lem:rl0igno}}\label{sec:lemmarl0igno}
In this section, we omit the superscript $\vartheta$ for simplicity unless otherwise stated, as the bounds are uniform over $\vartheta$.

\subsection{Proof of Lemma \ref{lem:varbound}}\label{sec:profvarb}

We start from the martingale decomposition as in \eqref{eq:martdecomp}:
\beq
m_{N}-\cexpct{m_{N}}{V}=m_{N}-\expctk{m_{N}}{N}=\sum_{k=1}^{N}(\E_{k-1}-\E_{k})\Big[R_{kk}\Big(1+\sum_{p,q}^{(k)}W_{kp}(R^{(k)})^{2}_{pq}W_{qk}\Big)\Big].
\eeq
Similar to proofs in {Section \ref{sec:tight}}, we use the fact
\beq
(\E_{k-1}-\E_{k})\bigg[\frac{1}{v_{k}-z-m_{N}^{(k)}(z)}\Big(1+\frac{1}{N}\sum_{p}^{(k)}(R^{(k)})^{2}_{pp}\Big)\bigg]=0,
\eeq
so that
\begin{multline}\label{eq:varbound0}
		\expct{\absv{m_{N}-\expctk{m_{N}}{N}}}^{2} \\
		=\biggexpct{\sum_{k=1}^{N}\biggabsv{(\E_{k-1}-\E_{k})\bigg[R_{kk}\Big(1+\sum_{p,q}^{(k)}W_{kp}(R^{(k)})^{2}_{pq}W_{qk}\Big)
		-\frac{1}{v_{k}-z-m_{N}^{(k)}(z)}\Big(1+\frac{1}{N}\sum_{p}^{(k)}(R^{(k)})^{2}_{pp}\Big)\bigg]}^{2}} \\
		\leq 4\sum_{k=1}^{N}\biggexpct{\biggabsv{R_{kk}\Big(1+\sum_{p,q}^{(k)}W_{kp}(R^{(k)})^{2}_{pq}W_{qk}\Big)
		-\frac{1}{v_{k}-z-m_{N}^{(k)}(z)}\Big(1+\frac{1}{N}\sum_{p}^{(k)}(R^{(k)})^{2}_{pp}\Big)}^{2}}\\
		\leq C\sum_{k=1}^{N}\biggexpct{\biggabsv{R_{kk}\sum_{p,q}^{(k)}W_{kp}(R^{(k)})^{2}_{pq}W_{qk}	-\frac{1}{v_{k}-z-m_{N}^{(k)}(z)}\frac{1}{N}\sum_{p}^{(k)}(R^{(k)})^{2}_{pp}}^{2}}
		+C\sum_{k=1}^{N}\Bigexpct{\Bigabsv{R_{kk}-\frac{1}{v_{k}-z-m_{N}^{(k)}(z)}}^{2}}
\end{multline}

Now we express each summand in the first sum by
\begin{multline}
R_{kk}\sum_{p,q}^{(k)}W_{kp}(R^{(k)})^{2}_{pq}W_{qk}-\frac{1}{N}\frac{1}{v_{k}-z-m_{N}^{(k)}(z)}\sum_{p}^{(k)}(R^{(k)})^{2}_{pp}	\\
=\big(R_{kk}-\frac{1}{v_{k}-z-m_{N}^{(k)}(z)}\Big)\Big(\sum_{p,q}^{(k)}W_{kp}(R^{(k)})^{2}_{pq}W_{qk}\Big)
+\frac{1}{v_{k}-z-m_{N}^{(k)}(z)}\Big(\sum_{p,q}^{(k)}W_{kp}(R^{(k)})^{2}_{pq}W_{qk}-\frac{1}{N}\sum_{p}^{(k)}(R^{(k)})^{2}_{pp}\Big).
\end{multline}

The first term above is bounded by
\begin{multline}\label{eq:varbound1}
\Bigabsv{R_{kk}-\frac{1}{v_{k}-z-m_{N}^{(k)}(z)}}\Bigabsv{\sum_{p,q}^{(k)}W_{kp}(R^{(k)})^{2}_{pq}W_{qk}}	\\
=\Bigabsv{\frac{1}{v_{k}-z-m_{N}^{(k)}(z)}}\Bigabsv{R_{kk}\Big(\sum_{p,q}^{(k)}W_{kp}(R^{(k)})^{2}_{pq}W_{qk}\Big)}\Bigabsv{(v_{k}-z-m_{N}^{(k)}(z))-\frac{1}{R_{kk}}}
\end{multline}

\begin{multline}
\leq\Bigabsv{\frac{1}{v_{k}-z-m_{N}^{(k)}(z)}}\biggabsv{\frac{1+\sum_{p,q}^{(k)}W_{kp}\absv{R^{(k)}}^{2}_{pq}W_{qk}}{W_{kk}-z-\sum_{p,q}^{(k)}W_{kp}R^{(k)}_{pq}W_{qk}}}\\
\times\Bigabsv{v_{k}-z-m_{N}^{(k)}(z)-\Big(W_{kk}-z-\sum_{p,q}^{(k)}W_{kp}R_{pq}^{(k)}W_{qk}\Big)}.
\end{multline}

Noting that
\beq
-\im\Big[W_{kk}-z-\sum_{p,q}^{(k)}W_{kp}R_{pq}^{(k)}W_{qk}\Big]=\eta\Big(1+\sum_{p,q}^{(k)}W_{kp}\absv{R^{(k)}}^{2}_{pq}W_{qk}\Big),
\eeq
The last line of \eqref{eq:varbound1} is bounded by
\beq\label{eq:varbound2}
\frac{1}{\eta}\Bigabsv{\frac{1}{v_{k}-z-m_{N}^{(k)}(z)}}\Bigabsv{\frac{1}{\sqrt{N}}A_{kk}+\Big(\sum_{p,q}^{(k)}W_{kp}R^{(k)}_{pq}W_{qk}-m_{N}^{(k)}(z)\Big)}.
\eeq

Then we use Lemma \ref{lem:lde}, to obtain
\begin{equation}
\biggcexpct{\biggabsv{\frac{1}{\sqrt{N}}A_{kk}+\Big(\sum_{p,q}^{(k)}W_{kp}R^{(k)}_{pq}W_{qk}-m_{N}^{(k)}(z)\Big)}^{2}}{W_{pq}:p,q\neq k}
\leq C\big(\frac{1}{N}+\frac{1}{N^{2}}\Tr \absv{R^{(k)}}^{2}\big).
\end{equation}

Thus,
\begin{multline}
\biggexpct{\Bigabsv{\frac{1}{v_{k}-z-m_{N}^{(k)}(z)}}^{2}\Bigabsv{\frac{1}{\sqrt{N}}A_{kk}+\Big(\sum_{p,q}^{(k)}W_{kp}R^{(k)}_{pq}W_{qk}-m_{N}^{(k)}(z)\Big)}^{2}}	\\
\leq \frac{C}{N}\biggexpct{\Bigabsv{\frac{1}{v_{k}-z-m_{N}^{(k)}(z)}}^{2}\Big(1+\frac{1}{N}\Tr \absv{R^{(k)}}^{2}\Big)}	
\leq \frac{C}{N}\E\Bigg[\Bigabsv{\frac{\big(1+\frac{1}{N}\Tr \absv{R^{(k)}}^{2}\big)}{v_{k}-z-m_{N}^{(k)}(z)}}^{1-\epsilon}\times\frac{\Bigabsv{\big(1+\frac{1}{N}\Tr \absv{R^{(k)}}^{2}\big)}^{\epsilon}}{\Bigabsv{v_{k}-z-m_{N}^{(k)}(z)}^{1+\epsilon}}\Bigg]	\\
\leq \frac{C}{N} \eta^{-1+\epsilon}\times \eta^{-2\epsilon}\times \Bigexpct{\frac{1}{\absv{v_{k}-z-m_{N}^{(k)}(z)}^{1+\epsilon}}}\\
\leq \frac{C}{N}\eta^{-1-\epsilon}\Bigexpct{\Bigabsv{\Bigcexpct{R_{kk}^{-1}}{W_{pq}:p,q\neq k}}^{1+\epsilon}}
\leq \frac{C}{N}\eta^{-1-\epsilon}\expct{\absv{R_{kk}}^{1+\epsilon}},
\end{multline}
where we used the Jensen's inequality in the last line.

For the second term, we use again the similar bound, to obtain
\begin{multline}\label{eq:varbound3}
\biggexpct{\frac{1}{\absv{v_{k}-z-m_{N}^{(k)}(z)}^{2}}\Bigabsv{\sum_{p,q}^{(k)}W_{kp}(R^{(k)})^{2}_{pq}W_{qk}-\frac{1}{N}\sum_{p}^{(k)}(R^{(k)})^{2}_{pp}}^{2}}	\\
\leq\biggexpct{\frac{1}{\absv{v_{k}-z-m_{N}^{(k)}(z)}^{2}}\biggcexpct{\biggabsv{\sum_{p,q}^{(k)}W_{kp}(R^{(k)})^{2}_{pq}W_{qk}-\frac{1}{N}\sum_{p}^{(k)}(R^{(k)})^{2}_{pp}}^{2}}{A_{pq}:p,q\neq k}}	\\
\leq \frac{C}{N}\biggexpct{\frac{\frac{1}{N}\Tr \absv{R^{(k)}}^{4}}{\absv{v_{k}-z-m_{N}^{(k)}}^{2}}}	\leq\frac{C}{N\eta^{2}}\times\biggexpct{\frac{\frac{1}{N}\Tr \absv{R^{(k)}}^{2}}{\absv{v_{k}-z-m_{N}^{(k)}(z)}^{2}}}\leq \frac{C}{N\eta^{3+\epsilon}}\expct{\absv{R_{kk}}^{1+\epsilon}}.
\end{multline}

Finally, each summand in the second sum of \eqref{eq:varbound0} is bounded as follows
\begin{multline}\label{eq:varbound4}
\Bigabsv{R_{kk}-\frac{1}{v_{k}-z-m_{N}^{(k)}(z)}}
\leq \Bigabsv{\frac{R_{kk}}{v_{k}-z-m_{N}^{(k)}(z)}}\times \Bigabsv{v_{k}-z-m_{N}^{(k)}(z)-\big(W_{ii}-z-\sum_{p,q}^{(k)}W_{kp}R^{(k)}_{pq}W_{qk}\big)}	\\
=\Bigabsv{\frac{R_{kk}}{v_{k}-z-m_{N}^{(k)}(z)}}\times \Bigabsv{\frac{1}{\sqrt{N}}A_{kk}+\sum_{p,q}^{(k)}W_{kp}R^{(k)}_{pq}W_{qk}-m_{N}^{(k)}(z)}.
\end{multline}
Noting that the last line is precisely \eqref{eq:varbound2} expect for the first factor replaced by $\absv{R_{kk}}$, the bound directly follows using the trivial bound $\absv{R_{kk}}\leq \eta^{-1}$.

For $\expct{\absv{m_{N}(z)-\expct{m_{N}(z)}}^{2}}$, the proof is similar to above, except we again use the filtration $\caG_{k}$ and $\frac{1}{-z-m_{N}^{(k)}(z)}$ as the auxiliary factor, instead of $\frac{1}{v_{k}-z-m_{N}^{(k)}(z)}$.

We start with the same martingale decomposition:
\begin{multline}
\expct{\absv{m_{N}(z)-\expct{m_{N}(z)}}^{2}}	\\
\leq C\sum_{k=1}^{N}\biggexpct{\biggabsv{R_{kk}\Big(1+\sum_{p,q}^{(k)}W_{kp}(R^{(k)})^{2}_{pq}W_{qk}\Big)-\frac{1}{-z-m_{N}^{(k)}(z)}\Big(1+\frac{1}{N}\sum_{p}^{(k)}(R^{(k)})^{2}_{pp}\Big)}^{2}}	\\
\leq C\sum_{k=1}^{N}\biggexpct{\Bigabsv{R_{kk}-\frac{1}{-z-m_{N}^{(k)}(z)}}^{2}\Bigabsv{\sum_{p,q}^{(k)}W_{kp}(R^{(k)})^{2}_{pq}W_{qk}}^{2}}	\\
+C\sum_{k=1}^{N}\biggexpct{\Bigabsv{\frac{1}{-z-m_{N}^{(k)}(z)}}^{2}\Bigabsv{\sum_{p,q}^{(k)}W_{kp}(R^{(k)})^{2}_{pq}W_{qk}-\frac{1}{N}\sum_{p}^{(k)}(R^{(k)})^{2}_{pp}}^{2}}
+C\sum_{k=1}^{N}\biggexpct{\biggabsv{R_{kk}+\frac{1}{z+m_{N}^{(k)}(z)}}^{2}}.
\end{multline}

We here give proof for the first sum, and the rest follows similar lines. Also as the order of $\vartheta$ has to be taken into account, we keep the dependency on $\vartheta$. As above, we see that
\begin{multline}
\Bigabsv{R_{kk}-\frac{1}{-z-m_{N}^{(k)}(z)}}\Bigabsv{\sum_{p,q}^{(k)}W_{kp}(R^{(k)})^{2}_{pq}W_{qk}}
=\Bigabsv{\frac{1}{-z-m_{N}^{(k)}(z)}}\Bigabsv{R_{kk}\sum_{p,q}^{(k)}W_{kp}(R^{(k)})^{2}_{pq}W_{qk}}\Bigabsv{-z-m_{N}^{(k)}(z)-\frac{1}{R_{kk}}}\\
\leq\frac{1}{\eta}\Bigabsv{\frac{1}{z+m_{N}^{(k)}(z)}}\Bigabsv{-W_{kk}+\sum_{p,q}^{(k)}W_{kp}R^{(k)}_{pq}W_{qk}-m_{N}^{(k)}(z)}.
\end{multline}

And we can proceed as follows:
\begin{multline}
\biggexpct{\Bigabsv{\frac{1}{-z-m_{N}^{(k)}(z)}}^{2}\Bigabsv{\frac{1}{\sqrt{N}}A_{kk}+\vartheta v_{k}+\sum_{p,q}^{(k)}W_{kp}R^{(k)}_{pq}W_{qk}-m_{N}^{(k)}(z)}^{2}}	\\
\leq C\biggexpct{\Bigabsv{\frac{1}{-z-m_{N}^{(k)}(z)}}^{2}\Big(\frac{1}{N}+\vartheta^{2}\Var{v_{1}}+\frac{1}{N^{2}}\Tr \absv{R^{(k)}}^{2}\Big)}	\\
\leq C\vartheta^{2}\Var{v_{1}}\biggexpct{\Bigabsv{\frac{1}{-z-m_{N}^{(k)}(z)}}^{2}\Big(1+\frac{1}{N\vartheta^{2}}\times\frac{1}{N}\Tr\absv{R^{(k)}}^{2}\Big)}.
\end{multline}
Using the condition $\vartheta\gg N^{-\frac{1}{2}}$, above is bounded by
\begin{equation} C\vartheta^{2}\Var{v_{1}}\biggexpct{\Bigabsv{\frac{1}{-z-m_{N}^{(k)}(z)}}^{2}\Big(1+\frac{1}{N}\Tr\absv{R^{(k)}}^{2}\Big)}
\leq C\vartheta^{2}\Var{v_{1}}\eta^{-1-\epsilon}\expct{\absv{R_{kk}}^{1+\epsilon}}
\end{equation}

\subsection{Proof of \eqref{eq:sharp} for $\xi^{\vartheta_{\infty}}$}
Recalling the formula of $\Gamma(z_{1},z_{2})$ in {Proposition \ref{prop:Gproc}}, we observe from the symmetry $\overline{\xi(z)}=\xi(\overline{z})$ that
\begin{multline}
\expct{\absv{\xi-\expct{\xi}}^{2}} =\Gamma(z,\overline{z}) =(w_{2}-2)\int_{\R}\frac{\absv{1+m_{fc}'(z)}^{2}}{\absv{x-z-m_{fc}(z)}^{4}}\dd\nu(x) \\
+(W_{4}-3)\bigg[\Big(\int_{\R}\frac{1}{\absv{x-z-m_{fc}(z)}^{2}}\dd\nu(x)\Big)\Big(\int_{\R}\frac{\absv{1+m'_{fc}(z)}^{2}}{\absv{x-z-m_{fc}(z)}^{4}}\dd\nu(x)\Big)\\
 +\biggabsv{\int_{\R}\frac{(1+m_{fc}'(z))}{(x-z-m_{fc}(z))\absv{x-z-m_{fc}(z)}^{2}}\dd\nu(x)}^{2}\bigg] \\
+2\Big(1-\int_{\R}\frac{1}{\absv{x-z-m_{fc}(z)}^{2}}\dd\nu(x)\Big)^{-2}\biggabsv{\int_{\R}\frac{(1+m_{fc}'(z))}{(x-z-m_{fc}(z))\absv{x-z-m_{fc}(z)}^{2}}\dd\nu(x)}^{2}\\
+2\Big(1-\int_{\R}\frac{1}{\absv{x-z-m_{fc}(z)}^{2}}\dd\nu(x)\Big)^{-1}\int_{R}\frac{\absv{1+m_{fc}'(z)}^{2}}{\absv{x-z-m_{fc}(z)}^{4}}\dd\nu(x).
\end{multline}
Thus, applying {Lemma \ref{lem:sqrtbehav}} repeatedly, we find that $\Var{\xi(z)}$ is bounded for $z\in\Gamma_{r}\cup\Gamma_{l}\cup\Gamma_{0}$. By a similar argument, we also have the $O(1)$ bound for
\beq
\absv{\expct{\xi(z)}} =\frac{1}{2}\biggabsv{\frac{m''_{fc}(z)}{(1+m'_{fc}(z))^{2}}}\biggabsv{(w_{2}-1) +m'_{fc}(z) +(W_{4}-3)\frac{m'_{fc}(z)}{1+m'_{fc}(z)}}.
\eeq
Noting that $\expct{\absv{\xi(z)}^{2}}=\Var{\xi(z)}+\absv{\expct{\xi(z)}}^{2}$ and $\absv{\Gamma_{r}\cup\Gamma_{l}\cup\Gamma_{0}}\to0$, we conclude
\beq
\int_{\Gamma_{\#}}\expct{\absv{\xi(\Gamma_{\#}(t))}^{2}}\absv{\Gamma_{\#}'(t)}\dd t\leq C\absv{\Gamma_{\#}} \to 0.
\eeq

\subsection{Proof of \eqref{eq:sharpN} for $\xi_{N}^{\vartheta}$}\label{sec:sharpNprf}
We fix $\epsilon>0$ and take $\Omega_{N}$ to be the event
\beq
\Omega_{N}\deq \Big\{\sup_{z\in\Gamma_{0}\cup\Gamma_{r}\cup\Gamma_{l}}\absv{m_{N}(z)-\wh{m}_{fc}(z)}\leq N^{-1+\epsilon}\Big\}
\eeq 
For $z\in\Gamma_{0}$, \eqref{eq:stieltjesconcen} implies for any large (but fixed) $D>0$,
\beq
\int_{\Gamma_{0}} \expct{\absv{\xi_{N}(z)}^{2}\lone_{\Omega_{N}}} \dd z \leq N^{\epsilon}\absv{\Gamma_{0}} \leq N^{\epsilon-\delta}\quad\text{and}\quad \prob{\Omega_{N}^{c}}\leq N^{-D}
\eeq
for sufficiently large $N$, hence \eqref{eq:sharpN} follows for $\Gamma_{\#}=\Gamma_{0}$ letting $\epsilon=\frac{\delta}{2}$.

Now for $\Gamma_{r}$, as $\lim_{v_{0}\to 0^{+}}\absv{\Gamma_{r}}=0$, it suffices to prove $\expct{\absv{\xi_{N}(z)}^{2}}\leq M$ uniformly on $\Gamma_{r}$, for some ($N$-independent) constant $M$. Recalling the results in {Section \ref{sec:meanb}}, we have
\beq
b_{N}(z) =-\frac{1}{2}\frac{m_{fc}''(z)}{(1+m_{fc}'(z))^{2}}\bigg[(w_{2}-1)+m_{fc}'(z)+(W_{4}-3)\frac{m_{fc}'(z)}{1+m_{fc}'(z)}\bigg] + o(1),
\eeq
hence $\absv{\expct{\xi_{N}(z)}}^{2}\leq C$ for $z\in\Gamma_{r}$.

Defining the event \beq
\Lambda_{N}:=\big[\lambda_{N} \leq\wh{\gamma}_{N} +N^{-\frac{1}{3}}\big]\cap\Omega_{N},
\eeq
we have $\prob{\Lambda_{N}^{c}} <N^{-D}$ for any large (but fixed) $D>0$ by {Lemma \ref{lem:rigid}}. On the event $\Lambda_{N}$, $\absv{\wh{L}_{+}-L_{+}}\leq N^{-c\alpha_{0}}$ gives the bound
\beq
\absv{R_{kk}} \leq\norm{R} =\max_{1\leq i\leq N}\biggabsv{\frac{1}{\lambda_{i}-z}} \leq\frac{1}{a_{+}-\wh{\gamma}_{N} -N^{-\frac{1}{3}}} \leq\frac{1}{a_{+}-\wh{L}_{+}-N^{-\frac{1}{3}}}\leq C,
\eeq
for any $k=1,\cdots,N$, uniformly for $z\in\Gamma_{r}$. Also the Cauchy interlacing gives the same bound $\norm{R^{(k)}}\leq C$ and from \eqref{eq:mfczmapswh} we obtain
\beq\label{eq:low}
	\absv{v_{k}-z-m_{N}^{(k)}(z)}\geq \absv{v_{k}-z-\wh{m}_{fc}(z)}-\absv{\wh{m}_{fc}(z)-m_{N}^{(k)}(z)}\geq C.
\eeq. 

To bound $\Var{\xi_{N}(z)}=\expct{\absv{\xi_{N}(z)-\expct{\xi_{N}(z)}}^{2}}=\expct{\absv{\zeta_{N}(z)}^{2}}$, we start with the bound in \eqref{eq:varbound0}, restricted on the event $\Lambda_{N}$. The proof is similar to that in Section \ref{sec:profvarb}, except that we replace bounds concerning $\eta$ into $\norm{R}$. For example, the second line of \eqref{eq:varbound1} is bounded by
\begin{multline}
	\Bigexpct{\lone_{\Lambda_{N}}\Bigabsv{\frac{1}{v_{k}-z-m_{N}^{(k)}(z)}}^{2}\Bigabsv{R_{kk}\Big(\sum_{p,q}^{(k)}W_{kp}(R^{(k)})^{2}_{pq}W_{qk}\Big)}^{2}\Bigabsv{(v_{k}-z-m_{N}^{(k)}(z))-\frac{1}{R_{kk}}}^{2}} \\
	\leq C\Bigexpct{\lone_{\Lambda_{N}}\Bigabsv{(v_{k}-z-m_{N}^{(k)}(z))-\frac{1}{R_{kk}}}^{2}}	\leq \frac{C}{N},
\end{multline}
where we used Cauchy-Schwarz inequality and Lemma \ref{lem:lde} in the last inequality. Likewise, bound for other terms in \eqref{eq:varbound0} follow by replacing $\eta$ into $\norm{R}$, $\norm{R^{(k)}}$, and \eqref{eq:low}. The proof for $\Gamma_{l}$ is the same, except we take the event $\Lambda_{N}$ as 
\beq
\Lambda_{N}\deq[\lambda_{1}\geq\wh{\gamma}_{1}-N^{-\frac{1}{3}}]\cap\Omega_{N}
\eeq
and take the bound for $\Lambda_{N}$ as
\beq
\norm{R} \leq \frac{1}{a_{-}-\lambda_{1}} \leq  \frac{1}{a_{-}-\wh{L}_{-}-N^{-\frac{1}{3}}}.
\eeq

\subsection{Proof of \eqref{eq:sharp} for $\wt{\xi}^{\vartheta_{\infty}}$}
As we saw above, we have
\beq
\expct{\absv{\wt{\xi}(z)}^{2}} \leq \absv{1+m_{fc}'(z)}^{2}\int \frac{1}{\absv{v-z-m_{fc}(z)}^{2}}\dd\nu(v)\leq C,
\eeq
uniformly for $z\in\Gamma$ from \eqref{eq:mfczmaps}. As $\absv{\Gamma_{0}\cup\Gamma_{r}\cup\Gamma_{l}}\to0$, \eqref{eq:sharp} for $\wt{\xi}^{\vartheta_{\infty}}$ directly follows.

\subsection{Proof of \eqref{eq:sharpN} for $\wt{\xi}_{N}^{\vartheta}$}
We take the event $\Omega_{N}$ to be contained in
\beq
\Big[\sup_{z\in\Gamma_{0}\cup\Gamma_{r}\cup\Gamma_{l}}:\absv{m_{N}(z)-m_{fc}(z)}\leq N^{-\frac{1}{2}+\epsilon}\Big] \cap\big[\max\{\absv{\lambda_{1}-\gamma_{1}},\absv{\lambda_{N}-\gamma_{N}}\}\leq cN^{-\frac{1}{3}+\epsilon}\big] \cap\Omega_{N}.
\eeq

Then for $\Gamma_{0}$, by {Corollary \ref{cor:olxiorder}}, we have
\beq
\int\expct{\absv{\wt{\xi}_{N}(\Gamma_{0}(t))}^{2}}\absv{\Gamma_{0}'(t)}\dd t \leq N^{\epsilon-\delta},
\eeq
hence we get the result by taking $\epsilon<\delta$.

For $\Gamma_{l}$, we first observe that from {Remark \ref{rem:wtximean}},
\beq
\absv{\expct{\wt{\xi}_{N}(z)\lone_{\Omega_{N}}}} =O(N^{-\frac{1}{2}+\epsilon}),
\eeq
hence it suffices to prove
\beq
\expct{\absv{\wt{\xi}_{N}(z)-\expct{\wt{\xi}_{N}(z)}}^{2}\lone_{\Omega_{N}}}\leq K
\eeq
for some $z,N$-independent constant $K\geq0$.

From the definition of $\Omega_{N}$, we have
\beq\label{eq:opnorm1}
\absv{R_{kk}}\leq \norm{R}=\max_{i}\frac{1}{\absv{\lambda_{i}-z}} \leq\frac{1}{\absv{a_{-}-\lambda_{1}}} \leq\frac{1}{L_{-}-a_{-}-cN^{-\frac{1}{3}+\epsilon}} \leq C
\eeq
on $\Omega_{N}$ uniformly for $z\in\Gamma_{l}$. Then following lines of Appendix \ref{sec:tightprocsqrtprf},
\begin{multline}
\expct{\absv{\wt{\xi}_{N}(z)-\expct{\wt{\xi}_{N}(z)}^{2}}} \\
\leq C\sum_{k}\bigg(\biggexpct{\biggabsv{R_{kk}\Big(\sum_{p,q}^{(k)}W_{kp}(R^{(k)})^{2}_{pq}W_{qk}-\frac{1}{N}\sum_{p}^{(k)}(R^{(k)})_{pp}^{2}\Big)}^{2}} \\
+\biggexpct{\biggabsv{\Big(R_{kk}+\frac{1}{z+m_{N}^{(k)}(z)}\Big)\Big(1+\frac{1}{N}\sum_{p}^{(k)}(R^{(k)})^{2}_{pp}\Big)}^{2}}\bigg).
\end{multline}
Now using {Lemma \ref{lem:lde}}, \eqref{eq:localintlaw}, and \eqref{eq:opnorm1}, the result follows from the same argument as in {Section \ref{sec:sharpNprf}}.

\section{Proofs of Lemmas in {Section~\ref{sec:prelim}}} \label{sec:lemmas3.2}
In this appendix, we provide the proofs of lemmas that were stated in {Section \ref{sec:prelim}}.

\begin{proof}[Proof of Lemma \ref{lem:chernoff}]
	Since {Assumption \ref{assump:regesdv}} implies that \eqref{eq:regesdveq} holds with probability $\geq 1-N^{-\frt}$, we focus on \eqref{eq:esdconveq} and \eqref{eq:esdconveqtheta}.
	
	To bound the probability of the event on which \eqref{eq:esdconveq} does not hold, we first note that
	\beq
	\absv{m_{\wh{\nu}}'(z)} \leq \int\frac{1}{\absv{x-z}^{2}}\dd\wh{\nu}(x)\leq \frac{1}{C_{1}^{2}}
	\eeq
	where $C_{1}=\dist(\caD,\supp\nu)$ and similarly $\absv{m_{\nu}'(z)}\leq C_{1}^{-2}$, so that
	\beq
	\absv{m_{\wh{\nu}}(z)-m_{\wh{\nu}}(z')}\leq C_{1}^{-2}\absv{z-z'}.
	\eeq
	Now we let $Q_{N}$ be a lattice in $\C^{+}$ with $\absv{Q_{N}}\leq C_{2}N^{1-\epsilon_{0}}$ such that for any $z\in \caD$, $\inf_{z'\in Q_{N}}\absv{z-z'}\leq C_{1}^{2}N^{-1/2+\epsilon_{0}}$ for sufficiently large $N$. Then, 
	\beq
	\prob{\sup_{z\in\caD}\absv{m_{\wh{\nu}}(z)-m_{\nu}(z)}\geq 3N^{-\frac{1}{2}+\epsilon_{0}}} \leq\sum_{z\in Q_{N}}\prob{\absv{m_{\wh{\nu}}(z)-m_{\nu}(z)}\geq N^{-\frac{1}{2}+\epsilon_{0}}}.
	\eeq
	
	On the other hand, for each fixed $z\in\caD$, a typical application of the Chernoff inequality implies
	\begin{multline}
	\prob{\absv{m_{\wh{\nu}}(z)-m_{\nu}(z)}\geq N^{-\frac{1}{2}+\epsilon_{0}}} \\
	=\biggprob{\biggabsv{\frac{1}{\sqrt{N}}\sum_{i=1}^{N}\Big(\frac{1}{v_{i}-z}-\expct{\frac{1}{v_{i}-z}}\Big)}\geq N^{\epsilon_{0}}} \leq C\exp(-cN^{2\epsilon_{0}})
	\end{multline}
	for some absolute constant $c,C>0$, hence
	\beq
	\prob{\sup_{z\in\caD}\absv{m_{\wh{\nu}}(z)-m_{\nu}(z)}\geq 3N^{-\frac{1}{2}+\epsilon_{0}}}\leq CN^{1-2\epsilon_{0}}\exp({cN^{-2\epsilon_{0}}})\leq N^{-\frt}.
	\eeq
	
	Now we bound the probability of the event on which \eqref{eq:esdconveqtheta} fails. We first note that
	\beq
	\vartheta^{-1}(m_{\wh{\nu}}^{\vartheta}(z)-m_{\nu}^{\vartheta}(z)) =\frac{1}{N}\sum_{i=1}^{N}\Big[\frac{v_{i}}{z(\vartheta v_{i}-z)}-\expct{\frac{v_{i}}{z(\vartheta v_{i}-z)}}\Big].
	\eeq
	Considering the functions
	\beq
	\wh{F}(\vartheta,z)\deq \int \frac{x}{z(\vartheta x-z)}\dd\wh{\nu}(x)\quad\text{and}\quad F(\vartheta,z)\deq\int\frac{x}{z(\vartheta x-z)}\dd\nu(x)
	\eeq
	defined on $\caD\subset\Theta_{\varpi}\times\C^{+}$, $F$ is jointly Lipschitz for 
	\begin{multline}
	\absv{F(\vartheta_{1},z_{1})-F(\vartheta_{2},z_{2})} \leq \int\biggabsv{v_{i}\frac{v_{i}(z_{2}\vartheta_{2}-z_{1}\vartheta z_{1})+(z_{1}^{2}-z_{2}^{2})}{z_{1}z_{2}(\vartheta x-z_{1})(\vartheta x-z_{2})}}\dd\nu(x) \\
	\leq d^{-4}(C^{2}\absv{z_{2}\vartheta_{2}-z_{1}\vartheta_{1}}+C\absv{z_{1}-z_{2}}),
	\end{multline}
	where $d=\inf_{(\vartheta,z)\in\caD}\dist(z,\supp\nu^{\vartheta})$, and similarly $\wh{F}$ is also jointly Lipschitz with constant bounded uniformly in $N$. Thus we conclude that there exists a constant $C>0$ independent of $(\vartheta,z)\in\caD$ and $N\in\N$ satisfying
	\begin{multline}
	\biggabsv{\frac{1}{\vartheta_{1}}(m_{\wh{\nu}}^{\vartheta_{1}}(z_{1})-m_{\nu}^{\vartheta_{1}}(z_{1})) -\frac{1}{\vartheta_{2}}(m_{\wh{\nu}}^{\vartheta_{2}}(z_{2})-m_{\nu}^{\vartheta_{2}}(z_{2}))} \\
	=\absv{\wh{F}(\vartheta_{1},z_{1})-\wh{F}(\vartheta_{2},z_{2})+F(\vartheta_{2},z_{2})-F(\vartheta_{1},z_{1})}\leq C(\absv{z_{1}-z_{2}}+\absv{\vartheta_{1}-\vartheta_{2}}).
	\end{multline}
	Now we let $Q_{N}$ be a lattice in $\caD$ with $\absv{Q_{N}}\leq CN^{6}$ such that for any $(\vartheta,z)\in \caD$, $\inf_{(\vartheta',z')\in Q_{N}}\absv{z-z'}+\absv{\vartheta-\vartheta'}\leq N^{-2}$ for sufficiently large $N$. Then, 
	\beq
	\biggprob{\sup_{(\vartheta,z)\in\caD}\biggabsv{\frac{1}{\vartheta}(m_{\wh{\nu}}^{\vartheta}(z)-m_{\nu}^{\vartheta}(z))}\geq 2N^{-\frac{1}{2}+\epsilon_{0}}} 
	\leq\sum_{(\vartheta,z)\in Q_{N}}\biggprob{\frac{1}{\vartheta}\absv{m_{\wh{\nu}}^{\vartheta}(z)-m_{\nu}^{\vartheta}(z)}\geq N^{-\frac{1}{2}+\epsilon_{0}}}.
	\eeq
	
	On the other hand, for each fixed $(\vartheta,z)\in\caD$, a typical application of McDiamird's inequality implies
	\begin{multline}
	\biggprob{\frac{1}{\vartheta}\absv{m_{\wh{\nu}}^{\vartheta}(z)-m_{\nu}^{\vartheta}(z)}\geq N^{-\frac{1}{2}+\epsilon_{0}}} \\
	=\biggprob{\biggabsv{\frac{1}{\sqrt{N}}\sum_{i=1}^{N}\Big(\frac{v_{i}}{z(\vartheta v_{i}-z)}-\biggexpct{\frac{v_{i}}{z(\vartheta v_{i}-z)}}\Big)}\geq N^{\epsilon_{0}}} \leq C\exp(-cN^{2\epsilon_{0}})
	\end{multline}
	for some absolute constant $c,C>0$, hence
	\beq
	\biggprob{\sup_{(\vartheta,z)\in\caD}\biggabsv{\frac{1}{\vartheta}(m_{\wh{\nu}}^{\vartheta}(z)-m_{\nu}^{\vartheta}(z))}\geq 2N^{-\frac{1}{2}+\epsilon_{0}}} \leq CN^{6}\exp({cN^{-2\epsilon_{0}}})\leq N^{-\frt}.
	\eeq
\end{proof}

\begin{proof}[Proof of Corollary \ref{cor:olxiorder}]
	We start with the bound
	\beq
	\absv{m_{N}^{\vartheta}(z)-m_{fc}^{\vartheta}(z)} \leq\biggabsv{ \sum_{i}\frac{1}{N}\frac{1}{\gamma_{i}^{\vartheta}-z} -\int_{\wt{\gamma}_{i-1}^{\vartheta}}^{\wt{\gamma}_{i}^{\vartheta}}\frac{1}{x-z}\dd\rho_{fc}^{\vartheta}(x)} +\frac{1}{N}\sum_{i}\biggabsv{\frac{1}{\gamma_{i}^{\vartheta}-z} -\frac{1}{\lambda_{i}^{\vartheta}-z}},
	\eeq
	where $\wt{\gamma}_{0}^{\vartheta}=L_{-}^{\vartheta}$, $\wt{\gamma}_{N}^{\vartheta}=L_{+}^{\vartheta}$ and $\int_{\R}^{\wt{\gamma}_{i}}\dd\rho_{fc}^{\vartheta}(x)=\frac{i}{N}$ for $1\leq i\leq N-1$.
	
	From $\kappa_{E}+\eta\sim 1$, we find that
	\begin{multline}
	\biggabsv{\sum_{i}\frac{1}{N}\frac{1}{\gamma_{i}^{\vartheta}-z} -\int_{\gamma_{i-1}^{\vartheta}}^{\gamma_{i}^{\vartheta}}\frac{1}{x-z}\dd\rho_{fc}^{\vartheta}(x)}
	=\biggabsv{\sum_{i=1}^{N}\int_{\wt{\gamma}_{i-1}^{\vartheta}}^{\wt{\gamma}_{i}^{\vartheta}}\Big(\frac{1}{\gamma_{i}^{\vartheta}-z}-\frac{1}{x-z}\Big)\dd\rho_{fc}^{\vartheta}(x)} \\
	\leq\sum_{i}\int_{\wt{\gamma}^{\vartheta}_{i-1}}^{\wt{\gamma}^{\vartheta}_{i}}\biggabsv{\frac{\gamma_{i}^{\vartheta}-x}{(x-z)(\gamma_{i}^{\vartheta}-z)}}\dd\rho_{fc}^{\vartheta}(x) 
	\leq\kappa_{E}^{-2}\sum_{i}\absv{\wt{\gamma}_{i}^{\vartheta}-\wt{\gamma}_{i-1}^{\vartheta}}\int_{\wt{\gamma}_{i-1}^{\vartheta}}^{\wt{\gamma}_{i}^{\vartheta}}\dd\rho_{fc}^{\vartheta}(x)
	=\frac{L_{+}^{\vartheta}-L_{-}^{\vartheta}}{\kappa_{E}^{2}N}=O(N^{-1}).
	\end{multline}
	
	On the other hand, using the rigidity estimate \eqref{eq:rigidncon} we get the bound
	\begin{multline}
	\frac{1}{N}\sum_{i}\biggabsv{\frac{1}{\gamma_{i}^{\vartheta}-z}-\frac{1}{\lambda_{i}^{\vartheta}-z}} \leq \frac{1}{\kappa_{E}(\kappa_{E}-\max(\absv{\gamma_{1}^{\vartheta}-\lambda_{1}^{\vartheta}},\absv{\gamma_{N}^{\vartheta}-\lambda_{N}^{\vartheta}})}\frac{1}{N}\sum_{i}\absv{\lambda_{i}^{\vartheta}-\gamma_{i}^{\vartheta}} \\
	\prec \frac{2}{N}\sum_{i=1}^{\lfloor\frac{N}{2}\rfloor}\big[N^{-\frac{2}{3}}i^{-\frac{1}{3}}+cN^{-\frac{1}{3}}i^{-\frac{2}{3}}\big] +\vartheta N^{-\frac{1}{2}}+2N^{-\frac{5}{3}+\epsilon}(1+cN^{\frac{1}{4}})= O(\theta N^{-\frac{1}{2}}),
	\end{multline}
	since $\vartheta \geq C N^{-1/2}$, where we let $\epsilon<\frac{5}{12}$.
\end{proof}

\begin{proof}[Proof of Corollary \ref{cor:xiorder}]
	We omit the superscript $\vartheta$ for simplicity. For $z\in\caD_{c}$, the bounds immediately follow from the strong local deformed semicircle laws {Lemma \ref{lem:locallaw}} as $\eta\sim 1$. For $z\in\Gamma_{r}\cup\Gamma_{l}\cup\Gamma_{0}$, we bound $\absv{m_{N}(z)-\wh{m}_{fc}(z)}$ by
	\begin{multline}
	\absv{m_{N}(z)-\wh{m}_{fc}(z)} =\biggabsv{\frac{1}{N}\sum_{i}\frac{1}{\lambda_{i}-z}-\int_{\R}\frac{1}{x-z}\dd\wh{\rho}_{fc}(x)} \\
	\leq \biggabsv{\sum_{i}\Big(\frac{1}{N}\frac{1}{\wh{\gamma}_{i}-z}-\int_{\wh{\lambda}_{i-1}}^{\wh{\lambda}_{i}}\frac{1}{x-z}\Big)} 
	+\biggabsv{\frac{1}{N}\sum_{i}\Big(\frac{1}{\lambda_{i}-z}-\frac{1}{\wh{\gamma}_{i}-z}\Big)},
	\end{multline}
	where $\wh{\lambda}_{0}=\wh{L}_{-}$, $\wh{\lambda}_{N}=\wh{L}_{+}$ and $\int_{-\infty}^{\wh{\lambda}_{i}}\dd\wh{\rho}_{fc}=\frac{i}{N}$ for $1\leq i\leq N-1$.
	
	Then $\wh{\kappa}_{E}\sim 1$ and $\absv{\wh{L}_{\pm}-L_{\pm}}\leq N^{-c\alpha_{0}}$ implies
	\begin{multline}
	\biggabsv{\sum_{i}\Big(\int_{\wh{\lambda}_{i-1}}^{\wh{\lambda}_{i}}\frac{\wh{\gamma}_{i}-x}{(x-z)(\wh{\gamma}_{i}-z)}\dd\wh{\rho}_{fc}(x)\Big)}  \leq\wh{\kappa}_{E}^{-2}\sum_{i}\int_{\wh{\lambda}_{i-1}}^{\wh{\lambda}_{i}}\absv{\wh{\gamma}_{i}-x}\dd\wh{\rho}_{fc}(x) \\
	\leq\frac{\wh{\kappa}_{E}^{-2}}{N}\sum_{i}(\wh{\lambda}_{i}-\wh{\lambda}_{i-1}) =\frac{\wh{\kappa}_{E}^{-2}(\wh{L}_{+}-\wh{L}_{-})}{N} =O(N^{-1})
	\end{multline}
	and {Lemma \ref{lem:rigid}} implies
	\begin{multline}
	\frac{1}{N}\sum_{i}\biggabsv{\frac{1}{\lambda_{i}-z}-\frac{1}{\wh{\gamma}_{i}-z}} \leq\wh{\kappa}_{E}^{-2}\frac{1}{N}\sum_{i}\absv{\wh{\gamma}_{i}-\lambda_{i}} \prec N^{-\frac{5}{3}}\sum_{i}\check{\alpha}_{i}^{-\frac{1}{3}} \\
	=2N^{-1}\sum_{i\leq\frac{N}{2}}\frac{1}{N}\Big(\frac{i}{N}\Big)^{-\frac{1}{3}} \sim N^{-1}\int_{0}^{2}x^{-\frac{1}{3}}dx =O(N^{-1}),
	\end{multline}
	so that the bound \eqref{eq:stieltjesconcen} for $z\in \Gamma_{l}\cup\Gamma_{r}\cup\Gamma_{0}$ follows.
	
	Now for $z\in\Gamma_{l}\cup\Gamma_{r}$, \eqref{eq:sqrtimwh} together with $\wh{\kappa}_{E}\sim 1$ implies
	\beq
	\sqrt{\frac{\Im\wh{m}_{fc}(z)}{N\eta}}\sim\sqrt{\frac{\eta}{\wh{\kappa}_{E}+\eta}\frac{1}{N\eta}}\sim\frac{1}{\sqrt{N}},
	\eeq
	so that
	\beq
	\absv{R_{ij}-\delta_{ij}\wh{g}_{i}(z)} \prec \sqrt{\frac{\Im \wh{m}_{fc}(z)}{N\eta}} +\frac{1}{N\eta} =O(N^{-\frac{1}{2}}+N^{-1+\delta}).
	\eeq
	Taking $\delta<\frac{1}{2}$, we get the result.
\end{proof}

\begin{proof}[Proof of Corollary \ref{cor:path}]
	For each $N$ and $0\leq k\leq N$ we have
	\begin{multline}
	\absv{\wh{I}_{k}^{\vartheta}(z_{1},z_{2})} \leq\frac{1}{N}\sum_{p>k}\absv{\wh{g}_{p}^{\vartheta}(z_{1})\wh{g}_{p}^{\vartheta}(z_{2})} \leq\frac{1}{N}\sum_{p}\absv{\wh{g}_{p}^{\vartheta}(z_{1})\wh{g}^{\vartheta}_{p}(z_{2})} \\
	\leq\frac{1}{2N}\sum_{p}(\absv{\wh{g}_{p}^{\vartheta}(z_{1})}^{2}+\absv{\wh{g}^{\vartheta}_{p}(z_{2})}^{2}) \leq\frac{1}{2}\int_{\R}\frac{1}{\absv{\vartheta x-z_{1}-\wh{m}^{\vartheta}_{fc}(z_{1})}^{2}} +\frac{1}{\absv{\vartheta x-z-\wh{m}^{\vartheta}_{fc}(z_{2})}^{2}}\dd\wh{\nu}(x) \\
	=\frac{1}{2}\Big(\frac{\Im \wh{m}_{fc}^{\vartheta}(z_{1})}{\Im (z_{1}+\wh{m}^{\vartheta}_{fc}(z_{1}))}+\frac{\Im\wh{m}^{\vartheta}_{fc}(z_{2})}{\Im(z_{2}+\wh{m}^{\vartheta}_{fc}(z_{2}))}\Big)
	\end{multline}
	where we used \eqref{eq:funceqhat} in the last equality. Similarly, from \eqref{eq:funceq},
	\beq
	\absv{I^{\vartheta}(z_{1},z_{2})} \leq\frac{1}{2}\Big(\frac{\Im m_{fc}^{\vartheta}(z_{1})}{\Im(z_{1}+m_{fc}^{\vartheta}(z_{1}))}+\frac{\Im m_{fc}^{\vartheta}(z_{2})}{\Im(z_{2}+m_{fc}^{\vartheta}(z_{2}))}\Big).
	\eeq
	
	From the inequalities above together with the uniform lower bound of $\Im m_{fc}^{\vartheta}(z)$ and $\Im \wh{m}_{fc}^{\vartheta}(z)$ given in {Corollary \ref{cor:prepath}}, it can be easily checked that there exists $r>0$ satisfying the condition.
\end{proof}

\begin{proof}[Proof of Corollary \ref{cor:covwh}]
	Using \eqref{eq:mfczmapswh} and {Corollary \ref{cor:path}}. we first reduce $\wh{I}_{0}^{\vartheta}(z_{1},z_{2})$ as follows:
	\begin{multline}
	\wh{I}_{0}^{\vartheta}(z_{1},z_{2}) =\int\frac{1}{(\vartheta x-z_{1}-\wh{m}_{fc}^{\vartheta}(z_{1}))(\vartheta x-z_{2}-\wh{m}_{fc}^{\vartheta}(z_{2}))}\dd\wh{\nu}(x) \\
	=\int\frac{1}{(\vartheta x-z_{1}-m_{fc}^{\vartheta}(z_{1}))(\vartheta x-z_{2}-m_{fc}^{\vartheta}(z_{2}))}\dd\wh{\nu}(x) +C\vartheta N^{-\frac{1}{2}+\epsilon_{0}}.
	\end{multline}
	Note that the uniformity of the constant $C$ follows from that of \eqref{eq:mfczmapswh} and {Corollary \ref{cor:path}}. Then, letting $w_{1}^{\vartheta}=z_{1}+m_{fc}^{\vartheta}(z_{1})$ and similarly $w_{2}^{{\vartheta}}$, we can rewrite the integrand as 
	\beq
	\frac{1}{(\vartheta x-w_{1}^{\vartheta})(\vartheta x-w_{2}^{\vartheta})} = \frac{1}{w_{1}^{\vartheta}-w_{2}^{\vartheta}}\Big[\frac{1}{\vartheta x-w_{1}^{\vartheta}}-\frac{1}{\vartheta x-w_{2}^{\vartheta}}\Big],
	\eeq
	so that we have
	\beq
	I^{\vartheta}(z_{1},z_{2})-\wh{I}_{0}^{\vartheta}(z_{1},z_{2}) =\frac{1}{w_{1}^{\vartheta}-w_{2}^{\vartheta}}[(m_{\nu}^{\vartheta}(w_{1}^{\vartheta})-m_{\wh{\nu}}^{\vartheta}(w_{1}^{\vartheta}))-(m_{\nu}^{\vartheta}(w_{2}^{\vartheta})-m_{\wh{\nu}}^{\vartheta}(w_{2}^{\vartheta}))]+C\vartheta N^{-\frac{1}{2}+\epsilon_{0}}.
	\eeq
	
	Now the Cauchy integral formula shows that
	\begin{equation}
	\frac{\big(m_{\nu}^{\vartheta}(w_{1}^{\vartheta})-m_{\wh{\nu}}^{\vartheta}(w_{1}^{\vartheta})\big)-\big(m_{\nu}^{\vartheta}(w_{2}^{\vartheta})-m_{\wh{\nu}}^{\vartheta}(w_{2}^{\vartheta})\big)}{w_{1}^{\vartheta}-w_{2}^{\vartheta}}
	=\frac{1}{2\pi \ii} \oint_{\gamma}\frac{\big(m_{\nu}^{\vartheta}(w)-m_{\wh{\nu}}^{\vartheta}(w)\big)-\big(m_{\nu}^{\vartheta}(w_{2}^{\vartheta})-m_{\wh{\nu}}^{\vartheta}(w_{2}^{\vartheta})\big)}{(w-w_{1})(w-w_{2})}\dd w
	\end{equation}
	for any contour $\gamma$ in $\C^{+}$ satisfying
	\beq
	\frac{1}{2\pi\ii}\oint_{\gamma}\frac{1}{w-w_{1}}\dd w=1.
	\eeq
	For our particular choice of the contour, we first consider a positively oriented simple closed curve $\gamma$ in $\caD_{c}$ such that $\caD_{1}$ lies inside $\gamma$. If we set $G(z)\equiv G^{\vartheta}(z)\deq z+m_{fc}^{\vartheta}(z)$, clearly the contour $G(\gamma)$ is contained in $\C^{+}$. Furthermore, the contour integral
	\beq
	\frac{1}{2\pi \ii}\oint_{G(\gamma)}\frac{1}{w-w_{1}}\dd w
	\eeq
	is precisely the number of zeros of 
	\beq
	G(z)-w_{1}=G(z)-G(z_{1})=0
	\eeq
	within the region bounded by $\gamma$. Now using {Lemma \ref{lem:G_diff}}, which implies the injectivity of $G$, together with the assumption on $\gamma$, we see that the number is precisely $1$, so that
	\beq
	\frac{1}{2\pi\ii}\oint_{G(\gamma)}\frac{1}{w-w_{1}}\dd w =1.
	\eeq
	Recalling the definition of $\gamma$, we have a constant $c_{0}>0$ such that $\absv{z-z_{0}}\geq c_{0}$ for any $z\in\gamma$ and $z_{0}\in\caD_{1}$, so that whenever $z\in\gamma$ and $z_{1},z_{2}\in\caD_{1}$, from {Lemma \ref{lem:G_diff}} we have the lower bound
	\beq
	\absv{(G(z)-G(z_{1}))(G(z)-G(z_{2}))}\geq d^{2}\absv{z-z_{1}}\absv{z-z_{2}}\geq d^{2}c_{0}^{2}.
	\eeq
	Therefore, we get the bound
	\begin{multline}
	\biggabsv{\frac{1}{2\pi \ii} \oint_{G(\gamma)} \frac{\big(m_{\nu}^{\vartheta}(w)-m_{\wh{\nu}}^{\vartheta}(w)\big)-\big(m_{\nu}^{\vartheta}(w_{2}^{\vartheta})-m_{\wh{\nu}}^{\vartheta}(w_{2}^{\vartheta})\big)}{(w-w_{1})(w-w_{2})}\dd w} \\
	=\biggabsv{\frac{1}{2\pi\ii} \oint_{\gamma}\frac{\big(m_{\nu}^{\vartheta}(G(z))-m_{\wh{\nu}}^{\vartheta}(G(z))\big)-\big(m_{\nu}^{\vartheta}(w_{2}^{\vartheta})-m_{\wh{\nu}}^{\vartheta}(w_{2}^{\vartheta})\big)}{(G(z)-G(z_{1}))(G(z)-G(z_{2}))} G'(z) \dd z} \\
	\leq \frac{\absv{\gamma}}{c_{0}^{2}d^{2}}\Big[\sup_{z\in\gamma}\absv{G'(z)}\big(\absv{m^{\vartheta}_{\nu}(G(z))-m^{\vartheta}_{\wh{\nu}}(G(z))}+\absv{m^{\vartheta}_{\nu}(G(z_{2}))-m_{\wh{\nu}}^{\vartheta}(G(z_{2}))}\big)\Big] \leq C\vartheta N^{-\frac{1}{2}+\epsilon_{0}},
	\end{multline}
	where the last inequality follows from \eqref{eq:esdconveqtheta}.
\end{proof}

\begin{proof}[Proof of Lemma~\ref{lem:mfc-msc}]
	We first note that by the self-consistent equation \eqref{eq:funceq}, together with the trivial bound
	\beq
	\absv{m_{fc}^{\vartheta}(z)} \leq \int_{\R}\biggabsv{\frac{1}{x-z}}\dd\rho_{fc}^{\vartheta}(x)\leq \dist(z,[L_{-}^{\vartheta},L_{+}^{\vartheta}])^{-1},
	\eeq
	implies that any pointwise limit $s(z)$ of $m_{fc}^{\vartheta}(z)$ must satisfy the self-consistent equation
	\beq
	s(z)=\int_{\R}\frac{1}{\vartheta_{\infty}x-z-s(z)}\dd\nu(x),\quad \Im s(z)>0 \text{ if }\Im z>0,
	\eeq
	which precisely coincides with that of $m_{fc}^{\vartheta_{\infty}}(z)$. Then, an application of Montel's theorem gives the uniform convergence $m_{fc}^{\vartheta}(z)\to m_{fc}^{\vartheta_{\infty}}(z)$ on $\caD$. 
	
	Now given the uniform convergence, we again use \eqref{eq:funceq} to get
	\begin{multline}
	m_{fc}^{\vartheta}(z)-m_{fc}^{\vartheta_{\infty}}(z) =\int_{\R} \Big[\frac{1}{\vartheta x-z-m_{fc}^{\vartheta}(z)}-\frac{1}{\vartheta_{\infty}x-z-m_{fc}^{\vartheta_{\infty}}(z)}\Big] \dd\nu(x) \\
	=\int_{\R}\frac{(m_{fc}^{\vartheta}(z)-m_{fc}^{\vartheta_{\infty}}(z))-(\vartheta-\vartheta_{\infty})x}{(\vartheta x-z-m_{fc}^{\vartheta}(z))(\vartheta_{\infty}x-z-m_{fc}^{\vartheta_{\infty}}(z))}\dd\nu(x) \\
	= (m_{fc}^{\vartheta}(z)-m_{fc}^{\vartheta_{\infty}}(z))\int_{\R}\frac{1}{(\vartheta x-z-m_{fc}^{\vartheta}(z))(\vartheta_{\infty}x-z-m_{fc}^{\vartheta_{\infty}}(z))}\dd\nu(x) \\
	-(\vartheta-\vartheta_{\infty}) \int_{\R}\frac{x}{(\vartheta x-z-m_{fc}^{\vartheta}(z))(\vartheta_{\infty}x-z-m_{fc}^{\vartheta_{\infty}}(z))}\dd\nu(x).
	\end{multline}
	so that 
	\begin{multline}
	\bigg(1-\biggabsv{\int_{\R}\frac{1}{(\vartheta x-z-m_{fc}^{\vartheta}(z))(\vartheta_{\infty}x-z-m_{fc}^{\vartheta_{\infty}}(z))}\dd\nu(x)}\Big)\absv{m_{fc}^{\vartheta}(z)-m_{fc}^{\vartheta_{\infty}}(z)} \\
	\leq \absv{\vartheta-\vartheta_{\infty}}\biggabsv{\int_{\R}\frac{x}{(\vartheta x-z-m_{fc}^{\vartheta}(z))(\vartheta_{\infty}x-z-m_{fc}^{\vartheta_{\infty}}(z))}\dd\nu(x)}.
	\end{multline}
	The stability bound \eqref{eq:mfczmaps} implies the existence of a constant $C>0$ satisfying
	\beq
	\int_{\R}\biggabsv{\frac{x}{(\vartheta x-z-m_{fc}^{\vartheta}(z))(\vartheta_{\infty}x-z-m_{fc}^{\vartheta_{\infty}})}}\dd\nu(x) \leq C,
	\eeq
	uniformly for $z\in\caD$. 
	On the other hand, our domain $\caD$ satisfies the assumptions of {Corollary \ref{cor:prepath}}, so that there exists $c'>0$ satisfying $\Im m_{fc}^{\vartheta}(z)>c'\Im z$ and similarly for $\vartheta_{\infty}$. Therefore we have
	\beq
	\frac{\Im m_{fc}^{\vartheta}(z)}{\eta+\Im m_{fc}^{\vartheta}(z)} = 1-\frac{\eta}{\eta+\Im m_{fc}^{\vartheta}(z)}>\frac{c'}{1+c'}
	\eeq
	and the same lower bound for $\vartheta_{\infty}$. Then we deduce
	\begin{multline}
	1-\int_{\R}\biggabsv{\frac{1}{(\vartheta x-z-m_{fc}^{\vartheta}(z))(\vartheta_{\infty}x-z-m_{fc}^{\vartheta}(z))}}\dd\nu(x) \\
	\leq 1-\Big[\frac{\Im m_{fc}^{\vartheta}(z)}{\eta+\Im m_{fc}^{\vartheta}(z)}\cdot\frac{\Im m_{fc}^{\vartheta_{\infty}}}{\eta+\Im m_{fc}^{\vartheta_{\infty}}(z)}\Big]^{\frac{1}{2}}	\leq 1-\frac{c'}{1+c'}=\frac{1}{1+c'},
	\end{multline}
	which completes the proof.
\end{proof}

\end{document}